\documentclass[12pt]{amsart}
\usepackage[left=1in,right=1in,top=1in,bottom=1in]{geometry}
\usepackage[latin1]{inputenc}
\usepackage[colorinlistoftodos]{todonotes}

\newcommand{\Josh}[1]{\todo[size=\tiny,inline,color=blue!30]{#1
\\ \hfill --- Josh}}
\usepackage{amssymb,amsmath,amsthm,comment,enumerate,hyperref,multicol,mathdots,tikz,ytableau, mathtools}
\usetikzlibrary{cd}
\usepackage{ stmaryrd }

\theoremstyle{definition}
\newtheorem{theorem}{Theorem}[section]
\newtheorem{lemma}[theorem]{Lemma}
\newtheorem{proposition}[theorem]{Proposition}
\newtheorem{corollary}[theorem]{Corollary}
\newtheorem{definition}[theorem]{Definition}

\newcommand{\Z}{\mathbb{Z}}

\newcommand{\BPD}[1]{\text{BPD}(#1)}
\newcommand{\bpd}[1]{\text{bpd}(#1)}
\newcommand{\co}[1]{\text{co}(#1)}
\definecolor{UF2}{HTML}{FA4616}
\definecolor{UF}{HTML}{0021A5}
\definecolor{Navy}{HTML}{0021A5}
\definecolor{Green}{HTML}{FA4616}
\newcommand{\cross}[1][UF]{\begin{tikzpicture}[scale = .4]
	\draw[lightgray] (0,0) -- (1,0) -- (1,1) -- (0,1) -- (0,0);
	\draw[thick,#1!100!black] (.5,0) -- (.5,1);
	\draw[thick,#1!100!black] (0,.5) -- (1,.5);
\end{tikzpicture}}
\newcommand{\oldhwire}[1][UF]{\begin{tikzpicture}[scale = .4]
	\draw[lightgray] (0,0) -- (1,0) -- (1,1) -- (0,1) -- (0,0);
	\draw[thick,#1!100!black] (0,.5) -- (1,.5);
\end{tikzpicture}}

\newcommand{\oldvwire}[1][UF]{\begin{tikzpicture}[scale = .4]
	\draw[lightgray] (0,0) -- (1,0) -- (1,1) -- (0,1) -- (0,0);
	\draw[thick,#1!100!black] (.5,0) -- (.5,1);
\end{tikzpicture}}

\newcommand{\oldjay}[1][UF]{\begin{tikzpicture}[scale = .4]
	\draw[lightgray] (0,0) -- (1,0) -- (1,1) -- (0,1) -- (0,0);
	\draw[thick,#1!100!black] (0,0.5) to[out=0,in=-90] (0.5,1);
\end{tikzpicture}}

\newcommand{\oldare}[1][UF]{\begin{tikzpicture}[scale = .4]
	\draw[lightgray] (0,0) -- (1,0) -- (1,1) -- (0,1) -- (0,0);
	\draw[thick,#1!100!black] (0.5,0) to[out=90,in=180] (1,0.5);
\end{tikzpicture}}
\newcommand{\inlinenowire}{\begin{tikzpicture}[scale = .3]
	\draw[lightgray] (0,0) -- (1,0) -- (1,1) -- (0,1) -- (0,0);
\end{tikzpicture}}

\newcommand{\olden}[1][UF]{\begin{tikzpicture}[scale = .4]
	\draw[lightgray] (0,0) -- (1,0) -- (1,1) -- (0,1) -- (0,0);
	\draw[thick,#1!100!black] (0,0.5) to[out=0,in=90] (0.5,0);
\end{tikzpicture}}
\newcommand{\oldel}[1][UF]{\begin{tikzpicture}[scale = .4]
	\draw[lightgray] (0,0) -- (1,0) -- (1,1) -- (0,1) -- (0,0);
	\draw[thick,#1!100!black] (0.5,1) to[out=-90,in=180] (1,0.5);
\end{tikzpicture}}

\usepackage[colorinlistoftodos]{todonotes}

\definecolor{Blue}{RGB}{0,112,192}
\definecolor{Brown}{RGB}{120,65,48}
\definecolor{Chartreuse}{RGB}{121,242,38}
\definecolor{Cyan}{RGB}{95,249,253}
\definecolor{Gold}{RGB}{255,192,0}
\definecolor{Green}{RGB}{0,176,80}
\definecolor{Navy}{HTML}{0021A5}
\definecolor{Orange}{HTML}{FA4616}
\definecolor{Pink}{RGB}{255,51,204}
\definecolor{Plum}{RGB}{142,69,133}
\definecolor{Purple}{RGB}{112,48,160}
\definecolor{Red}{RGB}{255,0,0}
\definecolor{Turquoise}{RGB}{55,191,168}

\newcommand{\grid}[2]{
        \foreach \x in {0,...,#1}{
        \foreach \y in {0,...,#2}{
            \draw[lightgray] (\x-0.5,\y-0.5) -- (\x+0.5,\y-0.5) -- (\x+0.5,\y+0.5) -- (\x-0.5,\y+0.5) -- (\x-0.5,\y-0.5);
        }
    }}

\newcommand{\hwire}[1][UF]{\begin{tikzpicture}[scale = .4]
	\draw[thick,#1!100!black] (-0.01,.5) -- (1.01,.5);
\end{tikzpicture}}

\newcommand{\vwire}[1][UF]{\begin{tikzpicture}[scale = .4]
	\draw[thick,#1!100!black] (.5,-0.01) -- (.5,1.01);
\end{tikzpicture}}

\newcommand{\newire}[1][UF]{\el[#1]}

\newcommand{\swwire}[1][UF]{\en[#1]}

\newcommand{\nowire}{\begin{tikzpicture}[scale = .4]
	\draw[lightgray] (0,0) -- (1,0) -- (1,1) -- (0,1) -- (0,0);
\end{tikzpicture}}

\newcommand{\shade}{\begin{tikzpicture}[scale = .4]
	\draw[fill=lightgray] (0,0) -- (1,0) -- (1,1) -- (0,1) -- (0,0);
\end{tikzpicture}}

\newcommand{\jay}[1][UF]{\begin{tikzpicture}[scale = .4]
	\draw[transparent] (0,0) -- (1,0) -- (1,1) -- (0,1) -- (0,0);
	\draw[thick,#1!100!black] (0,0.5) to[out=0,in=-90] (0.5,1);
\end{tikzpicture}}

\newcommand{\are}[1][UF]{\begin{tikzpicture}[scale = .4]
	\draw[transparent] (0,0) -- (1,0) -- (1,1) -- (0,1) -- (0,0);
	\draw[thick,#1!100!black] (0.5,0) to[out=90,in=180] (1,0.5);
\end{tikzpicture}}

\newcommand{\inlinebump}{\begin{tikzpicture}[scale = .3]
	\draw[lightgray] (0,0) -- (1,0) -- (1,1) -- (0,1) -- (0,0);
	\draw[thick,UF!100!black] (0.5,0) to[out=90,in=180] (1,0.5);
	\draw[thick,UF!100!black] (0,0.5) to[out=0,in=-90] (0.5,1);
\end{tikzpicture}}
\newcommand{\inlineare}{\begin{tikzpicture}[scale = .3]
	\draw[lightgray] (0,0) -- (1,0) -- (1,1) -- (0,1) -- (0,0);
	\draw[thick,UF!100!black] (0.5,0) to[out=90,in=180] (1,0.5);
\end{tikzpicture}}
\newcommand{\inlinejay}{\begin{tikzpicture}[scale = .3]
	\draw[lightgray] (0,0) -- (1,0) -- (1,1) -- (0,1) -- (0,0);
	\draw[thick,UF!100!black] (0,0.5) to[out=0,in=-90] (0.5,1);
\end{tikzpicture}}
\newcommand{\inlinehwire}{\begin{tikzpicture}[scale = .3]
	\draw[lightgray] (0,0) -- (1,0) -- (1,1) -- (0,1) -- (0,0);
	\draw[thick,UF!100!black] (0,.5) -- (1,.5);
\end{tikzpicture}}
\newcommand{\inlinevwire}{\begin{tikzpicture}[scale = .3]
	\draw[lightgray] (0,0) -- (1,0) -- (1,1) -- (0,1) -- (0,0);
	\draw[thick,UF!100!black] (.5,0) -- (.5,1);
\end{tikzpicture}}

\newcommand{\en}[1][UF]{\begin{tikzpicture}[scale = .4]
	\draw[transparent] (0,0) -- (1,0) -- (1,1) -- (0,1) -- (0,0);
	\draw[thick,#1!100!black] (0,0.5) to[out=0,in=90] (0.5,0);
\end{tikzpicture}}
\newcommand{\el}[1][UF]{\begin{tikzpicture}[scale = .4]
	\draw[transparent] (0,0) -- (1,0) -- (1,1) -- (0,1) -- (0,0);
	\draw[thick,#1!100!black] (0.5,1) to[out=-90,in=180] (1,0.5);
\end{tikzpicture}}

\begin{document}

\title[]{Permutations with only reduced co-BPDs}

\author[Arroyo]{Joshua Arroyo}
\address{Department of Mathematics,  University of Florida, Gainesville, FL 32611}
\email[Corresponding author]{joshuaarroyo@ufl.edu}

\author[Gregory]{Adam Gregory}
\address{Mathematics and Computer Science Department, Western Carolina University, Cullowhee, NC 28779}
\email{gregory@wcu.edu}

\date{}
\begin{abstract}
Bumpless pipe dreams (BPDs) are combinatorial objects used in the study of Schubert and Grothendieck polynomials. 
Weigandt recently introduced a co-BPD object associated to each BPD and used them to give an analogue to the change of bases formulas of Lenart and Lascoux between these polynomials. 
She posed the problem of characterizing the set of permutations whose BPDs have only reduced co-BPDs. 
We give a pattern-avoidance characterization for these permutations using a set of seven patterns.
\end{abstract}

\maketitle


\section{Introduction}
Schubert polynomials $\mathfrak{S}_w$~\cite{lascoux1982poly} and Grothendieck polynomials $\mathfrak{G}_w$~\cite{lascoux82groth}, both indexed by permutations $w \in S_n$, are distinguished representatives for the cohomology and $K$-theory classes of the complete flag variety, respectively. There are several known combinatorial formulas for $\mathfrak{S}_w$. One of the earliest examples is the {\em pipe dream} model of Billey--Bergeron~\cite{bergeron1993rc} and Billey--Jockusch--Stanley~\cite{billey1993comb}, which was extended to $\mathfrak{G}_w$ by Fomin--Kirillov~\cite{fomin1994groth}. A more recent example is the {\em bumpless pipe dream} (BPD) model of Lam--Lee--Shimozono~\cite{lam2021back} that was extended to $\mathfrak{G}_w$ by Weigandt~\cite{weigandt2021bumpless}. \\

Both families of polynomials are $\Z$-linear bases for $\Z[x_1, x_2,\ldots]$. Their change of bases formulas have been studied combinatorially by Lenart~\cite{lenart1999} and Lascoux~\cite{lascoux2004} using the pipe dream model. In recent work, Weigandt~\cite{weigandt2025changingbasespipedream} gives the analogous results for the bumpless pipe dream model by introducing a new object called a {\em co-bumpless pipe dream} (co-BPD).
In particular, the expansion of a Grothendieck polynomial into Schubert polynomials is given by
\[
    \mathfrak{G}_w = \sum_v (-1)^{\ell(v) - \ell(w)} a_{w,v} \cdot \mathfrak{S}_v \; ,
\]
where $a_{w,v}$ is the number of bumpless pipe dreams for $w$ whose co-bumpless pipe dream is {\em reduced} and traces out $v$. See Section~\ref{s:prelim} for precise definitions. \\

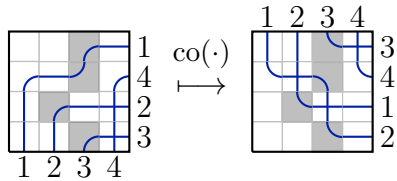
\begin{figure}[h!]
\centering
    \begin{tikzpicture}[scale =.4]
    \node at (2,3) {\shade};
    \node at (2,2) {\shade};
    \node at (1,1) {\shade};
    \node at (2,0) {\shade};
    
    \node at (0,3) {\nowire};\node at (1,3) {\nowire};\node at (2,3) {\oldare};\node at (3,3) {\oldhwire};\node at (0,2) {\oldare};\node at (1,2) {\oldhwire};\node at (2,2) {\oldjay};\node at (3,2) {\oldare};\node at (0,1) {\oldvwire};\node at (1,1) {\oldare};\node at (2,1) {\oldhwire};\node at (3,1) {\cross};\node at (0,0) {\oldvwire};\node at (1,0) {\oldvwire};\node at (2,0) {\oldare};\node at (3,0) {\cross};

    \draw[thick] (-0.5,-0.5) -- (-0.5,3.5) -- (3.5,3.5) -- (3.5,-0.5) -- (-0.5,-0.5);

    \node at (0,-1) {1};
    \node at (1,-1) {2};
    \node at (2,-1) {3};
    \node at (3,-1) {4};

    \node at (0,4) {\textcolor{white}{1}};
    \node at (1,4) {\textcolor{white}{2}};
    \node at (2,4) {\textcolor{white}{3}};
    \node at (3,4) {\textcolor{white}{4}};

    \node at (4,3) {1};
    \node at (4,2) {4};
    \node at (4,1) {2};
    \node at (4,0) {3};
    
    \end{tikzpicture}
    \raisebox{3em}{$\overset{\raisebox{0.5em}{\small \text{co}($\cdot$)}}{\longmapsto}$}
    \begin{tikzpicture}[scale =.4]
    \node at (2,3) {\shade};
    \node at (2,2) {\shade};
    \node at (1,1) {\shade};
    \node at (2,0) {\shade};
    
    \node at (0,3) {\oldvwire};\node at (1,3) {\oldvwire};\node at (2,3) {\oldel};\node at (3,3) {\cross};\node at (0,2) {\oldel};\node at (1,2) {\cross};\node at (2,2) {\olden};\node at (3,2) {\oldel};\node at (0,1) {\nowire};\node at (1,1) {\oldel};\node at (2,1) {\cross};\node at (3,1) {\oldhwire};\node at (0,0) {\nowire};\node at (1,0) {\nowire};\node at (2,0) {\oldel};\node at (3,0) {\oldhwire};

    \draw[thick] (-0.5,-0.5) -- (-0.5,3.5) -- (3.5,3.5) -- (3.5,-0.5) -- (-0.5,-0.5);

    \node at (0,4) {1};
    \node at (1,4) {2};
    \node at (2,4) {3};
    \node at (3,4) {4};
    
    \node at (0,-1) {\textcolor{white}{1}};
    \node at (1,-1) {\textcolor{white}{2}};
    \node at (2,-1) {\textcolor{white}{3}};
    \node at (3,-1) {\textcolor{white}{4}};

    \node at (4,3) {3};
    \node at (4,2) {4};
    \node at (4,1) {1};
    \node at (4,0) {2};
    
    \end{tikzpicture}
    \label{fig:intro co map}
    \caption{A (reduced) BPD for 1423 with its corresponding (non-reduced) co-BPD which traces out 3412.}
\end{figure}

\newpage

In this paper we aim to understand which permutations have the property that all of their co-bumpless pipe dreams are reduced.
Our main result is the following.

\begin{theorem}
\label{t:pattern-char}
A permutation $w$ has all reduced co-BPDs if and only if $w$ avoids the seven patterns in $\Pi = \{ 1423, 12543, 13254, 25143, 215643, 216543, 241653 \}$.
\end{theorem}

\begin{proof}
    This follows by combining Theorem~\ref{thm:avoidImpliesReduced} and Theorem~\ref{thm:reducedImpliesAvoid}.
\end{proof}

\noindent{\bf Outline:} In Section~\ref{s:prelim} we discuss background information.
The proof of Theorem~\ref{t:pattern-char} is comprised of Sections~\ref{s:config},~\ref{s:suff}, and ~\ref{s:ness}, where Section~\ref{s:suff} proves that having a non-reduced co-BPD implies containing and a pattern in $\Pi$, Section~\ref{s:ness} proves the converse, and Section~\ref{s:config} proves lemmas needed for both directions.
Lastly, Section~\ref{s:app} discusses applications for Theorem~\ref{t:pattern-char}.

\section{Preliminaries}\label{s:prelim}

A \emph{permutation} is an ordering of the set $[n]=\{1,2,\ldots, n\}$.
Let $S_n$ be the set of all permutations of $[n]$.
In this paper we write all our permutations in one-line notation.
An {\em inversion} of $w = w_1 \ldots w_n \in S_n$ is a pair $(i,j)$ with $i < j$ and $w_i > w_j$.
The number of inversions of $w$ is the {\em length} of $w$, denoted $\ell(w)$.
For example, $\ell(1432) = 3$.
A permutation $w\in S_n$ \emph{contains} $v\in S_k$ as a pattern if there is a $k$-element set of indices $i_1 < i_2 < \cdots < i_k$ such that $w_{i_r} < w_{i_s}$ if and only if $v_{r} < v_{s}$.
If $w$ does not contain $v$, then $w$ \emph{avoids} $v$.
For example $526134$ contains $1423$ as a pattern as $2634$ is order isomorphic to $1423$.
Further, $526134$ avoids $1234$ as a pattern as there is no increasing subsequence of length $4$.\\

A \emph{bumpless pipe dream} (BPD) of size $n$ is a tiling of an $n \times n$ grid with tiles from the set
\[ 
 \left\{ \; \raisebox{-0.2em}{\nowire} \; , \quad \raisebox{-0.2em}{\oldvwire} \; , \quad \raisebox{-0.2em}{\oldhwire} \; , \quad \raisebox{-0.2em}{\oldare} \; , \quad \raisebox{-0.2em}{\oldjay} \; , \quad \raisebox{-0.2em}{\cross} \;
 \right\},
\]
such that $n$ pipes enter from the bottom and exit to the right~\cite{lam2021back}. Labeling the pipes in increasing order from left-to-right as they enter, one obtains an {\em associated permutation} by reading the pipe labels from top-to-bottom as they exit, treating multiple crossings between a pair of pipes as turns if such an occurrence exists. A BPD is {\em reduced} if no two pipes cross more than once and {\em non-reduced} otherwise.

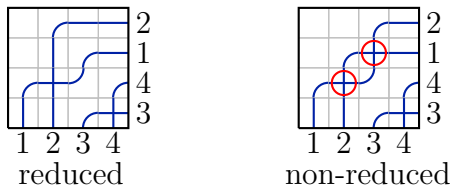
\[
\begin{tikzpicture}[scale=0.4]
\node at (0,3) {\nowire};\node at (1,3) {\oldare};\node at (2,3) {\oldhwire};\node at (3,3) {\oldhwire};\node at (0,2) {\nowire};\node at (1,2) {\oldvwire};\node at (2,2) {\oldare};\node at (3,2) {\oldhwire};\node at (0,1) {\oldare};\node at (1,1) {\cross};\node at (2,1) {\oldjay};\node at (3,1) {\oldare};\node at (0,0) {\oldvwire};\node at (1,0) {\oldvwire};\node at (2,0) {\oldare};\node at (3,0) {\cross};

\draw[thick] (-0.5,-0.5) -- (-0.5,3.5) -- (3.5,3.5) -- (3.5,-0.5) -- (-0.5,-0.5);

\node at (0,-1) {1};
\node at (1,-1) {2};
\node at (2,-1) {3};
\node at (3,-1) {4};

\node at (4,3) {2};
\node at (4,2) {1};
\node at (4,1) {4};
\node at (4,0) {3};

\node at (1.6,-2) {reduced};
\end{tikzpicture} \; \hspace{0.5in} \;
\begin{tikzpicture}[scale=0.4]

\node at (0,3) {\nowire};\node at (1,3) {\nowire};\node at (2,3) {\oldare};\node at (3,3) {\oldhwire};\node at (0,2) {\nowire};\node at (1,2) {\oldare};\node at (2,2) {\cross};\node at (3,2) {\oldhwire};\node at (0,1) {\oldare};\node at (1,1) {\cross};\node at (2,1) {\oldjay};\node at (3,1) {\oldare};\node at (0,0) {\oldvwire};\node at (1,0) {\oldvwire};\node at (2,0) {\oldare};\node at (3,0) {\cross};

\draw[red,thick] (2,2) circle (0.4);
\draw[red,thick] (1,1) circle (0.4);

\draw[thick] (-0.5,-0.5) -- (-0.5,3.5) -- (3.5,3.5) -- (3.5,-0.5) -- (-0.5,-0.5);

\node at (0,-1) {1};
\node at (1,-1) {2};
\node at (2,-1) {3};
\node at (3,-1) {4};

\node at (4,3) {2};
\node at (4,2) {1};
\node at (4,1) {4};
\node at (4,0) {3};

\node at (1.8,-2) {non-reduced};
\end{tikzpicture}
\]
Since double crossings are treated as turns, occasionally the second crossing tile of a pair of pipes will be instead rendered as a $\inlinebump$-tile.\\

Write $\text{BPD}(w)$ for the bumpless pipe dreams with associated permutation $w$.
Similarly, let $\textrm{bpd}(w)$ be the set of reduced bumpless pipe dreams with associated permutation $w$.
We refer to the $\inlineare$ and $\inlinejay$ tiles as {\em elbows}, specifically ``are'' and ``jay'' elbows, respectively.
We refer to $\inlinehwire$ tiles as ``flats''.
Every BPD is completely determined by the location of its elbows~\cite{weigandt2021bumpless}.

\begin{definition}
    A rectangular region of $B\in \BPD{w}$ is called an \emph{active} region if the top left corner is an $\inlineare$-tile, the bottom right corner is a $\inlinenowire$-tile and there are no elbows in the region besides potentially the corners.
\end{definition}

\begin{definition}
    A droop is a move on an active region of a BPD of one of the following forms:
    
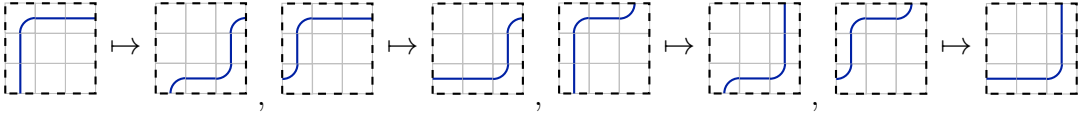
\[
\begin{tikzpicture}[scale= .4]
\node at (0,0) {\oldvwire};
\node at (1,0) {\nowire};
\node at (2,0) {\nowire};

\node at (0,1) {\oldvwire};
\node at (1,1) {\nowire};
\node at (2,1) {\nowire};

\node at (0,2) {\oldare};
\node at (1,2) {\oldhwire};
\node at (2,2) {\oldhwire};

\draw[thick, dashed] (-0.5,-0.5) -- (-0.5,2.5) -- (2.5,2.5) -- (2.5,-0.5) -- (-0.5,-0.5);

\node at (5,0) {\oldare};
\node at (6,0) {\oldhwire};
\node at (7,0) {\oldjay};

\node at (5,1) {\nowire};
\node at (6,1) {\nowire};
\node at (7,1) {\oldvwire};

\node at (5,2) {\nowire};
\node at (6,2) {\nowire};
\node at (7,2) {\oldare};

\node at (3.5, 1) {$\mapsto$};

\draw[thick, dashed] (4.5,-0.5) -- (4.5,2.5) -- (7.5,2.5) -- (7.5,-0.5) -- (4.5,-0.5);

\end{tikzpicture},
\begin{tikzpicture}[scale= .4]
\node at (0,0) {\oldjay};
\node at (1,0) {\nowire};
\node at (2,0) {\nowire};

\node at (0,1) {\oldvwire};
\node at (1,1) {\nowire};
\node at (2,1) {\nowire};

\node at (0,2) {\oldare};
\node at (1,2) {\oldhwire};
\node at (2,2) {\oldhwire};

\draw[thick, dashed] (-0.5,-0.5) -- (-0.5,2.5) -- (2.5,2.5) -- (2.5,-0.5) -- (-0.5,-0.5);

\node at (5,0) {\oldhwire};
\node at (6,0) {\oldhwire};
\node at (7,0) {\oldjay};

\node at (5,1) {\nowire};
\node at (6,1) {\nowire};
\node at (7,1) {\oldvwire};

\node at (5,2) {\nowire};
\node at (6,2) {\nowire};
\node at (7,2) {\oldare};

\node at (3.5, 1) {$\mapsto$};

\draw[thick, dashed] (4.5,-0.5) -- (4.5,2.5) -- (7.5,2.5) -- (7.5,-0.5) -- (4.5,-0.5);
\end{tikzpicture},
\begin{tikzpicture}[scale= .4]
\node at (0,0) {\oldvwire};
\node at (1,0) {\nowire};
\node at (2,0) {\nowire};

\node at (0,1) {\oldvwire};
\node at (1,1) {\nowire};
\node at (2,1) {\nowire};

\node at (0,2) {\oldare};
\node at (1,2) {\oldhwire};
\node at (2,2) {\oldjay};

\draw[thick, dashed] (-0.5,-0.5) -- (-0.5,2.5) -- (2.5,2.5) -- (2.5,-0.5) -- (-0.5,-0.5);

\node at (5,0) {\oldare};
\node at (6,0) {\oldhwire};
\node at (7,0) {\oldjay};

\node at (5,1) {\nowire};
\node at (6,1) {\nowire};
\node at (7,1) {\oldvwire};

\node at (5,2) {\nowire};
\node at (6,2) {\nowire};
\node at (7,2) {\oldvwire};

\node at (3.5, 1) {$\mapsto$};

\draw[thick, dashed] (4.5,-0.5) -- (4.5,2.5) -- (7.5,2.5) -- (7.5,-0.5) -- (4.5,-0.5);

\end{tikzpicture},
\begin{tikzpicture}[scale= .4]
\node at (0,0) {\oldjay};
\node at (1,0) {\nowire};
\node at (2,0) {\nowire};

\node at (0,1) {\oldvwire};
\node at (1,1) {\nowire};
\node at (2,1) {\nowire};

\node at (0,2) {\oldare};
\node at (1,2) {\oldhwire};
\node at (2,2) {\oldjay};

\draw[thick, dashed] (-0.5,-0.5) -- (-0.5,2.5) -- (2.5,2.5) -- (2.5,-0.5) -- (-0.5,-0.5);

\node at (5,0) {\oldhwire};
\node at (6,0) {\oldhwire};
\node at (7,0) {\oldjay};

\node at (5,1) {\nowire};
\node at (6,1) {\nowire};
\node at (7,1) {\oldvwire};

\node at (5,2) {\nowire};
\node at (6,2) {\nowire};
\node at (7,2) {\oldvwire};

\node at (3.5, 1) {$\mapsto$};

\draw[thick, dashed] (4.5,-0.5) -- (4.5,2.5) -- (7.5,2.5) -- (7.5,-0.5) -- (4.5,-0.5);

\end{tikzpicture}
\]
where all other pipes are unpictured pipes, but no elbow tiles within the pictured rectangle as it is an active region.
\end{definition}

\begin{definition}
    A rectangular region of $B\in \BPD{w}$ is called a \emph{blocked} region if it would be an active region except it contains an elbow within the region (besides as a corner).
    The pipe(s) of such elbows are called \emph{blocking pipes}.
    Note that the $\inlineare$-tile in the top left corner \textit{cannot} droop into the bottom right corner's $\inlinenowire$-tile.
\end{definition}

\begin{definition}
    A $K$-theoretic droop is a move on a BPD of one of the following forms:
    
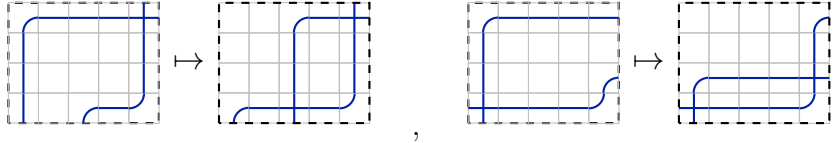
\[
\begin{tikzpicture}[scale= .4]
\node at (7,0) {\oldare};
\node at (8,0) {\oldhwire};
\node at (9,0) {\cross};
\node at (10,0) {\oldhwire};
\node at (11,0) {\oldjay};

\node at (7,1) {\nowire};
\node at (8,1) {\nowire};
\node at (9,1) {\oldvwire};
\node at (10,1) {\nowire};
\node at (11,1) {\oldvwire};

\node at (7,2) {\nowire};
\node at (8,2) {\nowire};
\node at (9,2) {\oldvwire};
\node at (10,2) {\nowire};
\node at (11,2) {\oldvwire};

\node at (7,3) {\nowire};
\node at (8,3) {\nowire};
\node at (9,3) {\oldare};
\node at (10,3) {\oldhwire};
\node at (11,3) {\cross};

\draw[thick, dashed] (-0.5,-0.5) -- (-0.5,3.5) -- (4.5,3.5) -- (4.5,-0.5) -- (-0.5,-0.5);

\node at (0,0) {\oldvwire};
\node at (1,0) {\nowire};
\node at (2,0) {\oldare};
\node at (3,0) {\oldhwire};
\node at (4,0) {\oldjay};

\node at (0,1) {\oldvwire};
\node at (1,1) {\nowire};
\node at (2,1) {\nowire};
\node at (3,1) {\nowire};
\node at (4,1) {\oldvwire};

\node at (0,2) {\oldvwire};
\node at (1,2) {\nowire};
\node at (2,2) {\nowire};
\node at (3,2) {\nowire};
\node at (4,2) {\oldvwire};

\node at (0,3) {\oldare};
\node at (1,3) {\oldhwire};
\node at (2,3) {\oldhwire};
\node at (3,3) {\oldhwire};
\node at (4,3) {\cross};

\node at (5.5, 1.5) {$\mapsto$};

\draw[thick, dashed] (6.5,-0.5) -- (6.5,3.5) -- (11.5,3.5) -- (11.5,-0.5) -- (6.5,-0.5);

\end{tikzpicture}\hspace{1em},\hspace{1em}
\begin{tikzpicture}[scale= .4]
\node at (7,0) {\cross};
\node at (8,0) {\oldhwire};
\node at (9,0) {\oldhwire};
\node at (10,0) {\oldhwire};
\node at (11,0) {\oldjay};

\node at (7,1) {\oldare};
\node at (8,1) {\oldhwire};
\node at (9,1) {\oldhwire};
\node at (10,1) {\oldhwire};
\node at (11,1) {\cross};

\node at (7,2) {\nowire};
\node at (8,2) {\nowire};
\node at (9,2) {\nowire};
\node at (10,2) {\nowire};
\node at (11,2) {\oldvwire};

\node at (7,3) {\nowire};
\node at (8,3) {\nowire};
\node at (9,3) {\nowire};
\node at (10,3) {\nowire};
\node at (11,3) {\oldare};

\draw[thick, dashed] (-0.5,-0.5) -- (-0.5,3.5) -- (4.5,3.5) -- (4.5,-0.5) -- (-0.5,-0.5);

\node at (0,0) {\cross};
\node at (1,0) {\oldhwire};
\node at (2,0) {\oldhwire};
\node at (3,0) {\oldhwire};
\node at (4,0) {\oldjay};

\node at (0,1) {\oldvwire};
\node at (1,1) {\nowire};
\node at (2,1) {\nowire};
\node at (3,1) {\nowire};
\node at (4,1) {\oldare};

\node at (0,2) {\oldvwire};
\node at (1,2) {\nowire};
\node at (2,2) {\nowire};
\node at (3,2) {\nowire};
\node at (4,2) {\nowire};

\node at (0,3) {\oldare};
\node at (1,3) {\oldhwire};
\node at (2,3) {\oldhwire};
\node at (3,3) {\oldhwire};
\node at (4,3) {\oldhwire};

\node at (5.5, 1.5) {$\mapsto$};

\draw[thick, dashed] (6.5,-0.5) -- (6.5,3.5) -- (11.5,3.5) -- (11.5,-0.5) -- (6.5,-0.5);

    \end{tikzpicture}
\]
where there may be other unpictured pipes, but no elbow tiles within the pictured rectangle.
\end{definition}
Lam, Lee and Shimozono showed that $\bpd{w}$ is connected by droop moves \cite{lam2021back}.
Weigandt showed that $\BPD{w}$ is connected by the combination of droop moves and $K$-theoretic droop moves \cite{weigandt2021bumpless}.

\begin{definition}[\cite{lam2021back}]
    The Schubert polynomial of a permutation $w$ is
    \begin{equation}
        \mathfrak{S}_w = \sum_{B\in\textrm{bpd}(w)} \prod_{i\in \raisebox{-0.2em}{\inlinenowire}(B)} x_i
    \end{equation}
    where $\inlinenowire(B)$ is a multiset defined by containing an $i$ for each blank tile $B$ has in row $i$.
\end{definition}
Grothendieck polynomials can be computed using the set of all BPDs.
\begin{definition}[\cite{weigandt2021bumpless}]
    The Grothendieck polynomial of a permutation $w$ is
    \begin{equation}
        \mathfrak{G}_w = \sum_{B\in\textrm{BPD}(w)} (-1)^{|\raisebox{-0.2em}{\inlinenowire}(B)|-\ell(w)} \prod_{i\in\raisebox{-0.2em}{\inlinenowire}(B)} x_i \prod_{j\in\raisebox{-0.2em}{\inlinejay}(B)} (1-x_j)
    \end{equation}
    where $\inlinejay(B)$ is a multiset defined by containing a $j$ for each jay elbow $B$ has in row $j$.
\end{definition}

In recent work, Weigandt defined a co-BPD object corresponding to each BPD. 

\begin{definition}[\cite{weigandt2025changingbasespipedream}]
\label{d:coBPD}
    For a given bumpless pipe dream $B$, its corresponding \emph{co-bumpless pipe dream} $\textrm{co}(B)$ is defined by exchanging tiles as follows:
    \[
    \oldare \; \raisebox{0.3em}{$\longleftrightarrow$} \; \oldel \; , \quad \oldjay \; \raisebox{0.3em}{$\longleftrightarrow$} \; \olden \; , \quad \oldvwire \; \raisebox{0.3em}{$\longleftrightarrow$} \; \nowire \; , \quad \; \oldhwire \; \raisebox{0.3em}{$\longleftrightarrow$} \; \cross \; .
    \]

\end{definition}
In other words, the locations of the elbows are the same, but now pipes enter from the top instead of the bottom while still exiting to the right.

\newpage 

Labeling pipes in increasing order from left-to-right as they enter, $\text{co}(B)$ {\em traces out} a permutation by reading the pipe labels from top-to-bottom as they exit, treating any multiple crossings as turns.
The permutation associated to $B$ may be different from the one $\text{co}(B)$ traces out.
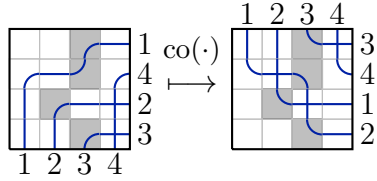
\[
    \begin{tikzpicture}[scale =.4]
    \node at (2,3) {\shade};
    \node at (2,2) {\shade};
    \node at (1,1) {\shade};
    \node at (2,0) {\shade};
    
    \node at (0,3) {\nowire};\node at (1,3) {\nowire};\node at (2,3) {\oldare};\node at (3,3) {\oldhwire};\node at (0,2) {\oldare};\node at (1,2) {\oldhwire};\node at (2,2) {\oldjay};\node at (3,2) {\oldare};\node at (0,1) {\oldvwire};\node at (1,1) {\oldare};\node at (2,1) {\oldhwire};\node at (3,1) {\cross};\node at (0,0) {\oldvwire};\node at (1,0) {\oldvwire};\node at (2,0) {\oldare};\node at (3,0) {\cross};

    \draw[thick] (-0.5,-0.5) -- (-0.5,3.5) -- (3.5,3.5) -- (3.5,-0.5) -- (-0.5,-0.5);

    \node at (0,-1) {1};
    \node at (1,-1) {2};
    \node at (2,-1) {3};
    \node at (3,-1) {4};

    \node at (0,4) {\textcolor{white}{1}};
    \node at (1,4) {\textcolor{white}{2}};
    \node at (2,4) {\textcolor{white}{3}};
    \node at (3,4) {\textcolor{white}{4}};

    \node at (4,3) {1};
    \node at (4,2) {4};
    \node at (4,1) {2};
    \node at (4,0) {3};
    
    \end{tikzpicture}
    \raisebox{3em}{$\overset{\raisebox{0.5em}{\small \text{co}($\cdot$)}}{\longmapsto}$}
    \begin{tikzpicture}[scale =.4]
    \node at (2,3) {\shade};
    \node at (2,2) {\shade};
    \node at (1,1) {\shade};
    \node at (2,0) {\shade};
    
    \node at (0,3) {\oldvwire};\node at (1,3) {\oldvwire};\node at (2,3) {\oldel};\node at (3,3) {\cross};\node at (0,2) {\oldel};\node at (1,2) {\cross};\node at (2,2) {\olden};\node at (3,2) {\oldel};\node at (0,1) {\nowire};\node at (1,1) {\oldel};\node at (2,1) {\cross};\node at (3,1) {\oldhwire};\node at (0,0) {\nowire};\node at (1,0) {\nowire};\node at (2,0) {\oldel};\node at (3,0) {\oldhwire};

    \draw[thick] (-0.5,-0.5) -- (-0.5,3.5) -- (3.5,3.5) -- (3.5,-0.5) -- (-0.5,-0.5);

    \node at (0,4) {1};
    \node at (1,4) {2};
    \node at (2,4) {3};
    \node at (3,4) {4};
    
    \node at (0,-1) {\textcolor{white}{1}};
    \node at (1,-1) {\textcolor{white}{2}};
    \node at (2,-1) {\textcolor{white}{3}};
    \node at (3,-1) {\textcolor{white}{4}};

    \node at (4,3) {3};
    \node at (4,2) {4};
    \node at (4,1) {1};
    \node at (4,0) {2};
    
    \end{tikzpicture}
    \label{fig:intro co map}
\]

A co-BPD is {\em reduced} if no two pipes cross more than once and {\em non-reduced} otherwise. 
Notice that $\text{co}(\cdot)$ does not necessarily preserve reducedness as seen in the example above.
We refer to the set $\{\text{co}(B) : B \in \text{BPD}(w)\}$ as the co-BPDs for $w$.
These objects were introduced in order to give a combinatorial description of changing bases between Grothendieck and Schubert polynomials.
The formula for expanding Grothendiecks into Schuberts is the following:

\begin{theorem}[\cite{weigandt2025changingbasespipedream}]\label{thm:g to s}
    Let $a_{w,v} = \#\{B \in \mathrm{BPD}(w) : \mathrm{co}(B) \text{ reduced \& traces out } v\}$. Then
    \[
    \mathfrak{G}_w = \sum_v (-1)^{\ell(v) - \ell(w)} a_{w,v} \cdot \mathfrak{S}_v \; .
    \]
\end{theorem}

Notice that Theorem~\ref{thm:g to s} specifically only sums over reduced co-BPDs.
This motivates the work of this paper as we determine for which $w$ are there any $B$ with $\co{B}$ tracing $v$ not being counted due to $\co{B}$ being non-reduced.

\section{Configuration lemmas}\label{s:config}

\begin{lemma}
\label{l:config}
For $B \in \BPD{w}$, one has that $\co{B}$ is non-reduced if and only if $B$ contains an instance of the following configuration:

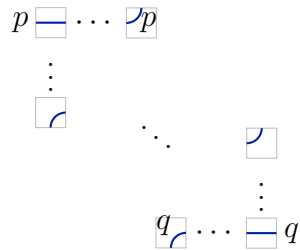
\[
\begin{tikzpicture}[scale=0.4]
\node at (0,0) {\oldhwire};
\node at (-1,0) {$p$};
\node at (3,0) {\oldjay};
\node at (3.25,0) {$p$};
\node at (0,-3) {\oldare};
\node at (3.5,-3.5) {$\ddots$};

\node at (7,-7) {\oldhwire};
\node at (8,-7) {$q$};
\node at (4,-7) {\oldare};
\node at (3.75, -6.75) {$q$};
\node at (7,-4) {\oldjay};

\node at (0,-1.5) {$\vdots$};
\node at (1.5,0) {$\cdots$};
\node at (7,-5.5) {$\vdots$};
\node at (5.5,-7) {$\cdots$};

\end{tikzpicture}
\]
where the closest elbow below the top flat is an $\inlineare$-elbow, which in turn is either the closest elbow left of the bottom flat or the closest elbow to its right is an $\inlinejay$-elbow, whose closest elbow below it is an $\inlineare$-elbow who has the same properties as the previous $\inlineare$-elbow and so forth.
Likewise, the closest elbow right of the top flat is an $\inlinejay$-elbow, which in turn is either the closest elbow above the bottom flat or the closest elbow below it is an $\inlineare$-elbow whose closest elbow to its right is an $\inlinejay$-elbow who has the same properties as the previous $\inlineare$-elbow.
Note that this means that for the $\inlineare$-elbows above there either is a chain of alternating elbows connecting them, or they are the same elbows; likewise for the $\inlinejay$-elbows.
\end{lemma}
\begin{proof}
A pair of pipes crosses twice in a co-BPD if and only if it contains an instance of the following configuration:
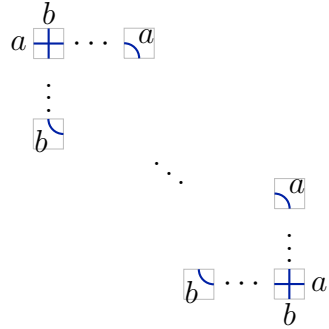
\[
\begin{tikzpicture}[scale=0.4]
\node at (0,0) {\cross};
\node at (0,1) {$b$};
\node at (-1,0) {$a$};
\node at (3,0) {\olden};
\node at (3.25,0.25) {$a$};
\node at (0,-3) {\oldel};
\node at (-0.25,-3.25) {$b$};
\node at (4,-4) {$\ddots$};

\node at (8,-8) {\cross};
\node at (8,-9) {$b$};
\node at (9,-8) {$a$};
\node at (5,-8) {\oldel};
\node at (4.75, -8.25) {$b$};
\node at (8,-5) {\olden};
\node at (8.25,-4.75) {$a$};

\node at (0,-1.5) {$\vdots$};
\node at (1.5,0) {$\cdots$};
\node at (8,-6.5) {$\vdots$};
\node at (6.5,-8) {$\cdots$};

\end{tikzpicture}
\]
(Note that the $b$ and $a$ swap in the bottom-right cross because this tile is treated as a $\inlinebump$ since these pipes have already crossed before.)
In order for both $\olden$-elbows to be connected there must either be a chain of alternating elbows such that each adjacent pair is in the same row or column, or $a$ must double cross with another pipe.
However, in the latter case we may instead consider the pair of $a$ and the pipe it double crossed before double crossing with $b$.
As the new double crossing is above and to the left this will eventually terminate.
Likewise for the $\oldel$-elbows and $b$ pipe.
From the definition of $\co{B}$ it follows that $B$ has the configuration in the Lemma statement.
Further, if $B$ contains the configuration in the Lemma statement then by definition $\co{B}$ has the configuration above with the possible exception that the second cross might not be between the same pair of pipes.
However, in that case either $a$ or $b$ must double cross with another pipe.
This due to the fact that the chain of alternating elbows has no elbows in between meaning the only way for $a$ or $b$ to not participate in the bottom right corner crossing is to double cross with another pipe.
\end{proof}

We present some examples of a BPD and its co-BPD to illustrate why all of the conditions above are necessary and sufficient.

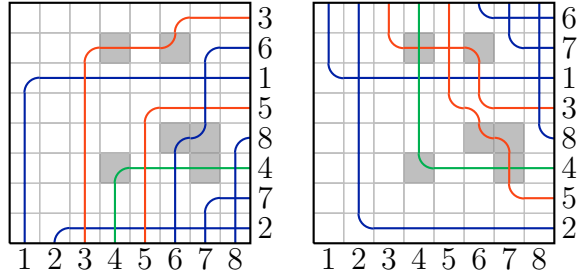
\begin{figure}[h!]
    \centering
    \begin{tikzpicture}[scale =.4]\node at (3,6) {\shade};\node at (5,6) {\shade};\node at (3,2) {\shade};\node at (6,2) {\shade};\node at (6,3) {\shade};\node at (5,3) {\shade};\grid{7}{7}\node at (0,7) {\nowire};\node at (1,7) {\nowire};\node at (2,7) {\nowire};\node at (3,7) {\nowire};\node at (4,7) {\nowire};\node at (5,7) {\are[Orange]};\node at (6,7) {\hwire[Orange]};\node at (7,7) {\hwire[Orange]};\node at (0,6) {\nowire};\node at (1,6) {\nowire};\node at (2,6) {\are[Orange]};\node at (3,6) {\hwire[Orange]};\node at (4,6) {\hwire[Orange]};\node at (5,6) {\jay[Orange]};\node at (6,6) {\are[Navy]};\node at (7,6) {\hwire[Navy]};\node at (0,5) {\are[Navy]};\node at (1,5) {\hwire[Navy]};\node at (2,5) {\hwire[Navy]};\node at (2,5) {\vwire[Orange]};\node at (3,5) {\hwire[Navy]};\node at (4,5) {\hwire[Navy]};\node at (5,5) {\hwire[Navy]};\node at (6,5) {\hwire[Navy]};\node at (6,5) {\vwire[Navy]};\node at (7,5) {\hwire[Navy]};\node at (0,4) {\vwire[Navy]};\node at (1,4) {\nowire};\node at (2,4) {\vwire[Orange]};\node at (3,4) {\nowire};\node at (4,4) {\are[Orange]};\node at (5,4) {\hwire[Orange]};\node at (6,4) {\hwire[Orange]};\node at (6,4) {\vwire[Navy]};\node at (7,4) {\hwire[Orange]};\node at (0,3) {\vwire[Navy]};\node at (1,3) {\nowire};\node at (2,3) {\vwire[Orange]};\node at (3,3) {\nowire};\node at (4,3) {\vwire[Orange]};\node at (5,3) {\are[Navy]};\node at (6,3) {\jay[Navy]};\node at (7,3) {\are[Navy]};\node at (0,2) {\vwire[Navy]};\node at (1,2) {\nowire};\node at (2,2) {\vwire[Orange]};\node at (3,2) {\are[Green]};\node at (4,2) {\hwire[Green]};\node at (4,2) {\vwire[Orange]};\node at (5,2) {\hwire[Green]};\node at (5,2) {\vwire[Navy]};\node at (6,2) {\hwire[Green]};\node at (7,2) {\hwire[Green]};\node at (7,2) {\vwire[Navy]};\node at (0,1) {\vwire[Navy]};\node at (1,1) {\nowire};\node at (2,1) {\vwire[Orange]};\node at (3,1) {\vwire[Green]};\node at (4,1) {\vwire[Orange]};\node at (5,1) {\vwire[Navy]};\node at (6,1) {\are[Navy]};\node at (7,1) {\hwire[Navy]};\node at (7,1) {\vwire[Navy]};\node at (0,0) {\vwire[Navy]};\node at (1,0) {\are[Navy]};\node at (2,0) {\hwire[Navy]};\node at (2,0) {\vwire[Orange]};\node at (3,0) {\hwire[Navy]};\node at (3,0) {\vwire[Green]};\node at (4,0) {\hwire[Navy]};\node at (4,0) {\vwire[Orange]};\node at (5,0) {\hwire[Navy]};\node at (5,0) {\vwire[Navy]};\node at (6,0) {\hwire[Navy]};\node at (6,0) {\vwire[Navy]};\node at (7,0) {\hwire[Navy]};\node at (7,0) {\vwire[Navy]};\draw[thick] (-0.5,-0.5) -- (-0.5,7.5) -- (7.5,7.5) -- (7.5,-0.5) -- (-0.5,-0.5);\node at (0,-1) {1};\node at (8,7) {3};\node at (1,-1) {2};\node at (8,6) {6};\node at (2,-1) {3};\node at (8,5) {1};\node at (3,-1) {4};\node at (8,4) {5};\node at (4,-1) {5};\node at (8,3) {8};\node at (5,-1) {6};\node at (8,2) {4};\node at (6,-1) {7};\node at (8,1) {7};\node at (7,-1) {8};\node at (8,0) {2};\end{tikzpicture}\;
    \begin{tikzpicture}[scale =.4]\node at (3,6) {\shade};\node at (5,6) {\shade};\node at (3,2) {\shade};\node at (6,2) {\shade};\node at (6,3) {\shade};\node at (5,3) {\shade};\grid{7}{7}\node at (0,7) {\vwire[Navy]};\node at (1,7) {\vwire[Navy]};\node at (2,7) {\vwire[Orange]};\node at (3,7) {\vwire[Green]};\node at (4,7) {\vwire[Orange]};\node at (5,7) {\el[Navy]};\node at (6,7) {\hwire[Navy]};\node at (6,7) {\vwire[Navy]};\node at (7,7) {\hwire[Navy]};\node at (7,7) {\vwire[Navy]};\node at (0,6) {\vwire[Navy]};\node at (1,6) {\vwire[Navy]};\node at (2,6) {\el[Orange]};\node at (3,6) {\hwire[Orange]};\node at (3,6) {\vwire[Green]};\node at (4,6) {\hwire[Orange]};\node at (4,6) {\vwire[Orange]};\node at (5,6) {\en[Orange]};\node at (6,6) {\el[Navy]};\node at (7,6) {\hwire[Navy]};\node at (7,6) {\vwire[Navy]};\node at (0,5) {\el[Navy]};\node at (1,5) {\hwire[Navy]};\node at (1,5) {\vwire[Navy]};\node at (2,5) {\hwire[Navy]};\node at (3,5) {\hwire[Navy]};\node at (3,5) {\vwire[Green]};\node at (4,5) {\hwire[Navy]};\node at (4,5) {\vwire[Orange]};\node at (5,5) {\hwire[Navy]};\node at (5,5) {\vwire[Orange]};\node at (6,5) {\hwire[Navy]};\node at (7,5) {\hwire[Navy]};\node at (7,5) {\vwire[Navy]};\node at (0,4) {\nowire};\node at (1,4) {\vwire[Navy]};\node at (2,4) {\nowire};\node at (3,4) {\vwire[Green]};\node at (4,4) {\el[Orange]};\node at (5,4) {\newire[Orange]};\node at (5,4) {\swwire[Orange]};\node at (6,4) {\hwire[Orange]};\node at (7,4) {\hwire[Orange]};\node at (7,4) {\vwire[Navy]};\node at (0,3) {\nowire};\node at (1,3) {\vwire[Navy]};\node at (2,3) {\nowire};\node at (3,3) {\vwire[Green]};\node at (4,3) {\nowire};\node at (5,3) {\el[Orange]};\node at (6,3) {\en[Orange]};\node at (7,3) {\el[Navy]};\node at (0,2) {\nowire};\node at (1,2) {\vwire[Navy]};\node at (2,2) {\nowire};\node at (3,2) {\el[Green]};\node at (4,2) {\hwire[Green]};\node at (5,2) {\hwire[Green]};\node at (6,2) {\hwire[Green]};\node at (6,2) {\vwire[Orange]};\node at (7,2) {\hwire[Green]};\node at (0,1) {\nowire};\node at (1,1) {\vwire[Navy]};\node at (2,1) {\nowire};\node at (3,1) {\nowire};\node at (4,1) {\nowire};\node at (5,1) {\nowire};\node at (6,1) {\el[Orange]};\node at (7,1) {\hwire[Orange]};\node at (0,0) {\nowire};\node at (1,0) {\el[Navy]};\node at (2,0) {\hwire[Navy]};\node at (3,0) {\hwire[Navy]};\node at (4,0) {\hwire[Navy]};\node at (5,0) {\hwire[Navy]};\node at (6,0) {\hwire[Navy]};\node at (7,0) {\hwire[Navy]};\draw[thick] (-0.5,-0.5) -- (-0.5,7.5) -- (7.5,7.5) -- (7.5,-0.5) -- (-0.5,-0.5);\node at (0,-1) {1};\node at (8,7) {6};\node at (1,-1) {2};\node at (8,6) {7};\node at (2,-1) {3};\node at (8,5) {1};\node at (3,-1) {4};\node at (8,4) {3};\node at (4,-1) {5};\node at (8,3) {8};\node at (5,-1) {6};\node at (8,2) {4};\node at (6,-1) {7};\node at (8,1) {5};\node at (7,-1) {8};\node at (8,0) {2};\end{tikzpicture}
    \caption{Note that an instance of the configuration does not necessarily guarantee that the same pipes will crossing at both pipes. However, the crossing pipes can only change if there is another instance of the configuration contained within the configuration.}
    \label{fig:double_config}
\end{figure}

\newpage

\begin{figure}[h!]
    \centering
    \begin{tikzpicture}[scale =.4]\node at (3,1) {\shade};\node at (5,1) {\shade};\node at (5,2) {\shade};\node at (1,5) {\shade};\node at (1,3) {\shade};\node at (2,5) {\shade};\grid{6}{6}\node at (0,6) {\nowire};\node at (1,6) {\nowire};\node at (2,6) {\are[Orange]};\node at (3,6) {\hwire[Orange]};\node at (4,6) {\hwire[Orange]};\node at (5,6) {\hwire[Orange]};\node at (6,6) {\hwire[Orange]};\node at (0,5) {\are[Orange]};\node at (1,5) {\hwire[Orange]};\node at (2,5) {\jay[Orange]};\node at (3,5) {\nowire};\node at (4,5) {\nowire};\node at (5,5) {\are[Green]};\node at (6,5) {\hwire[Green]};\node at (0,4) {\vwire[Orange]};\node at (1,4) {\nowire};\node at (2,4) {\are[Navy]};\node at (3,4) {\hwire[Navy]};\node at (4,4) {\hwire[Navy]};\node at (5,4) {\hwire[Navy]};\node at (5,4) {\vwire[Green]};\node at (6,4) {\hwire[Navy]};\node at (0,3) {\vwire[Orange]};\node at (1,3) {\are[Orange]};\node at (2,3) {\hwire[Orange]};\node at (2,3) {\vwire[Navy]};\node at (3,3) {\hwire[Orange]};\node at (4,3) {\hwire[Orange]};\node at (5,3) {\hwire[Orange]};\node at (5,3) {\vwire[Green]};\node at (6,3) {\hwire[Orange]};\node at (0,2) {\vwire[Orange]};\node at (1,2) {\vwire[Orange]};\node at (2,2) {\vwire[Navy]};\node at (3,2) {\nowire};\node at (4,2) {\are[Green]};\node at (5,2) {\jay[Green]};\node at (6,2) {\are[Navy]};\node at (0,1) {\vwire[Orange]};\node at (1,1) {\vwire[Orange]};\node at (2,1) {\vwire[Navy]};\node at (3,1) {\are[Green]};\node at (4,1) {\hwire[Green]};\node at (4,1) {\vwire[Green]};\node at (5,1) {\hwire[Green]};\node at (6,1) {\hwire[Green]};\node at (6,1) {\vwire[Navy]};\node at (0,0) {\vwire[Orange]};\node at (1,0) {\vwire[Orange]};\node at (2,0) {\vwire[Navy]};\node at (3,0) {\vwire[Green]};\node at (4,0) {\vwire[Green]};\node at (5,0) {\are[Navy]};\node at (6,0) {\hwire[Navy]};\node at (6,0) {\vwire[Navy]};\draw[thick] (-0.5,-0.5) -- (-0.5,6.5) -- (6.5,6.5) -- (6.5,-0.5) -- (-0.5,-0.5);\node at (0,-1) {1};\node at (7,6) {1};\node at (1,-1) {2};\node at (7,5) {5};\node at (2,-1) {3};\node at (7,4) {3};\node at (3,-1) {4};\node at (7,3) {2};\node at (4,-1) {5};\node at (7,2) {7};\node at (5,-1) {6};\node at (7,1) {4};\node at (6,-1) {7};\node at (7,0) {6};\end{tikzpicture}
    \begin{tikzpicture}[scale =.4]\node at (3,1) {\shade};\node at (5,1) {\shade};\node at (5,2) {\shade};\node at (1,5) {\shade};\node at (1,3) {\shade};\node at (2,5) {\shade};\grid{6}{6}\node at (0,6) {\vwire[Orange]};\node at (1,6) {\vwire[Orange]};\node at (2,6) {\el[Navy]};\node at (3,6) {\hwire[Navy]};\node at (3,6) {\vwire[Green]};\node at (4,6) {\hwire[Navy]};\node at (4,6) {\vwire[Green]};\node at (5,6) {\hwire[Navy]};\node at (5,6) {\vwire[Navy]};\node at (6,6) {\hwire[Navy]};\node at (6,6) {\vwire[Navy]};\node at (0,5) {\el[Orange]};\node at (1,5) {\hwire[Orange]};\node at (1,5) {\vwire[Orange]};\node at (2,5) {\en[Orange]};\node at (3,5) {\vwire[Green]};\node at (4,5) {\vwire[Green]};\node at (5,5) {\el[Navy]};\node at (6,5) {\hwire[Navy]};\node at (6,5) {\vwire[Navy]};\node at (0,4) {\nowire};\node at (1,4) {\vwire[Orange]};\node at (2,4) {\el[Orange]};\node at (3,4) {\hwire[Orange]};\node at (3,4) {\vwire[Green]};\node at (4,4) {\hwire[Orange]};\node at (4,4) {\vwire[Green]};\node at (5,4) {\hwire[Orange]};\node at (6,4) {\hwire[Orange]};\node at (6,4) {\vwire[Navy]};\node at (0,3) {\nowire};\node at (1,3) {\el[Orange]};\node at (2,3) {\hwire[Orange]};\node at (3,3) {\hwire[Orange]};\node at (3,3) {\vwire[Green]};\node at (4,3) {\hwire[Orange]};\node at (4,3) {\vwire[Green]};\node at (5,3) {\hwire[Orange]};\node at (6,3) {\hwire[Orange]};\node at (6,3) {\vwire[Navy]};\node at (0,2) {\nowire};\node at (1,2) {\nowire};\node at (2,2) {\nowire};\node at (3,2) {\vwire[Green]};\node at (4,2) {\el[Green]};\node at (5,2) {\en[Green]};\node at (6,2) {\el[Navy]};\node at (0,1) {\nowire};\node at (1,1) {\nowire};\node at (2,1) {\nowire};\node at (3,1) {\el[Green]};\node at (4,1) {\hwire[Green]};\node at (5,1) {\hwire[Green]};\node at (5,1) {\vwire[Green]};\node at (6,1) {\hwire[Green]};\node at (0,0) {\nowire};\node at (1,0) {\nowire};\node at (2,0) {\nowire};\node at (3,0) {\nowire};\node at (4,0) {\nowire};\node at (5,0) {\el[Green]};\node at (6,0) {\hwire[Green]};\draw[thick] (-0.5,-0.5) -- (-0.5,6.5) -- (6.5,6.5) -- (6.5,-0.5) -- (-0.5,-0.5);\node at (0,-1) {1};\node at (7,6) {3};\node at (1,-1) {2};\node at (7,5) {6};\node at (2,-1) {3};\node at (7,4) {1};\node at (3,-1) {4};\node at (7,3) {2};\node at (4,-1) {5};\node at (7,2) {7};\node at (5,-1) {6};\node at (7,1) {4};\node at (6,-1) {7};\node at (7,0) {5};\end{tikzpicture}
    \caption{The above example illustrates the importance of requiring chains of alternating elbows.
    This BPD does not have the configuration as there is not a chain of alternating elbows and thus the co-BPD is reduced.}
    \label{fig:issue}
\end{figure}
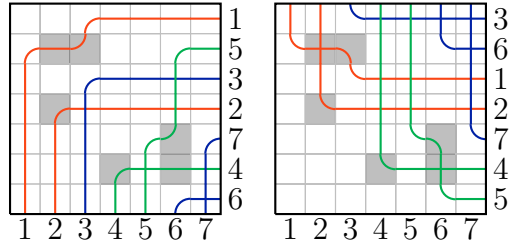

\begin{lemma}
\label{l:jay-two-pipes}
    If a pipe $r$ in a bumpless pipe dream contains a $\inlinejay$-elbow, there must be pipes $s$ and $t$ with $r < s < t$ that exit in the order $rts$.
\end{lemma}
\begin{proof}
    A bumpless pipe dream cannot have a $\inlinejay$-elbow in the last column, so there is a pipe $t$ that enters in the last column, and a pipe $s$ with $s < t$ that enters in the column where $r$ makes its $\inlinejay$-elbow. Observe that $s$ must make a $\inlineare$-elbow before the row the $\inlinejay$-elbow of $r$ is in to avoid a collision. So if $s$ has no $\inlinejay$-elbows, it follows that the pipes exit in the order $rts$, as $t$ exits out of the row that the $\inlinejay$-tile of $r$ is in. Otherwise, if $s$ has a $\inlinejay$-elbow, we repeat the argument recursively. Since the BPD is finite, this process eventually terminates.
\end{proof}

\begin{lemma}\label{lem:oneFlat}
    If a pipe $r$ has a $\inlinehwire$-tile and $\inlinejay$-elbow sandwiched between two $\inlineare$-elbows, the the BPD contains either a $1423$, $1432$, or $13254$ pattern where the $1$ is from the $r$ pipe.
\end{lemma}

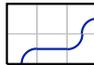
\begin{figure}[h!]
    \begin{tikzpicture}[scale =.4]\grid{2}{1};\node at (2,1) {\are[Navy]};\node at (0,0) {\are[Navy]};\node at (1,0) {\hwire[Navy]};\node at (2,0) {\jay[Navy]};\draw[thick] (-0.5,-0.5) -- (-0.5,1.5) -- (2.5,1.5) -- (2.5,-0.5) -- (-0.5,-0.5);\end{tikzpicture}
    \caption{An example of such a pipe $r$ from Lemma~\ref{lem:oneFlat}.}
    \label{fig:oneFlat}
\end{figure}

\begin{proof}
    Let $r$ be a pipe as described in the lemma statement.
    Note that in order for the pipe to be in such a position it must have drooped into row where there are two boxes it could droop into.
    Consider the BPD's Rothe BPD, in this BPD there must be two blank tiles below and right of the $r$ pipe.
    However, some of those blank tiles may be blocked by other pipes, we will consider each case.\\

    \textbf{Case 1:} Suppose neither blank tile is blocked.
    Then the right boundary pipe must enter after the two bottom boundary pipes and exit before them.
    Therefore, the BPD either has a $1423$ or $1432$ pattern with the $r$ serving as the $1$ pipe.

    \textbf{Case 2:} Suppose that only right blank tile is blocked.
    Then the right boundary pipe, $s$, for the unblocked blank tile must enter after the bottom boundary pipe for the unblocked blank and exit before that same pipe.
    Additionally, there must be a blank tile below and to the right of $s$ meaning it has two boundary pipes that exit in reverse order.
    Therefore, the BPD either has a $13254$, $13524$, or $13542$ pattern depending on when the bottom boundary pipe for the unblocked blank tile exits.
    This simplifies to the BPD having a  $13254$, $1423$, or $1432$ pattern.

    \textbf{Case 3:} Suppose both blank tiles are blocked.
    Then there must be another pipe below and to the right of pipe $r$ for which the conditions this lemma apply.
    As the BPD has a finite number of pipes eventually there will be a pipe $s$ for which a case other than this case applies.
    Take the pattern achieved in that case and replace the $s$ in the pattern with the $r$ pipe to achieve the same pattern.
\end{proof}

\begin{lemma}\label{lem:twoFlats}
    If a pipe $r$ has two $\inlinehwire$-tiles and one $\inlinejay$-elbow sandwiched between two $\inlineare$-elbows, the the BPD contains either a $1423$, $13254$, or $15432$ pattern where the $1$ is from the $r$ pipe.
\end{lemma}

\begin{figure}[h!]
    \begin{tikzpicture}[scale =.4]\grid{3}{1};\node at (3,1) {\are[Navy]};\node at (0,0) {\are[Navy]};\node at (1,0) {\hwire[Navy]};\node at (2,0) {\hwire[Navy]};\node at (3,0) {\jay[Navy]};\draw[thick] (-0.5,-0.5) -- (-0.5,1.5) -- (3.5,1.5) -- (3.5,-0.5) -- (-0.5,-0.5);\end{tikzpicture}
    \caption{An example of such a pipe $r$ from Lemma~\ref{lem:twoFlats}}
    \label{fig:twoFlats}
\end{figure}

\begin{proof}
    Let $r$ be a pipe as described in the lemma statement.
    Note that in order for the pipe to be in such a position it must have drooped into row where there are two boxes it could droop into.
    Consider the BPD's Rothe BPD, in this BPD there must be three blank tiles below and right of the $r$ pipe.
    However, some of those blank tiles may be blocked by other pipes, we will consider each case.\\

    \textbf{Case 1:} Suppose none of the blank tiles are blocked.
    Then the right boundary pipe enters after and exits before the three bottom boundary pipes.
    Thus, the BPD has a $1423$ or $15432$ pattern with the $1$ from the $r$ pipe, the largest number from the right boundary pipe, and the rest from the bottom boundary pipes.

    \textbf{Case 2:} Suppose one of the blank tiles is blocked.
    Then the right boundary pipe of the unblocked tiles enters after and exits before the two bottom boundary pipes of unblocked tiles.
    Further there must exist another blank tile below and right of the unblock tiles' right boundary pipe meaning that there are two pipes in reverse order that enter and exit after the right boundary pipe.
    Suppose the unblocked tiles' bottom boundary pipes exit in order, then the BPD has a $1423$ pattern with the $r$ pipe as the $1$.
    Therefore, suppose otherwise.
    If one of the unblocked tiles' bottom boundary pipes exits before the boundary pipes of the blocked tile, then the BPD has a $13254$ pattern. This pattern is made from (listed in the order they show up in the pattern) the $r$ pipe, the unblock tiles' right boundary pipe, one of the unblocked tiles' bottom boundary pipes, and the boundary pipes of the blocked tiles.
    If one of the unblocked tiles' bottom boundary pipes exits inbetween the boundary pipes of the blocked tile, then the BPD has a $1423$ pattern.
    This pattern is made from the $r$ pipe, the larger blocked tile boundary pipe, the inbetween boundary pipe, and the smaller blocked tile boundary pipe.
    If both unblocked tiles' bottom boundary pipes exit after the boundary pipes of the blocked tile, then the BPD has a $15432$ pattern.
    This pattern is made from the $r$ pipe, the blocked tile's boundary pipes, and the unblocked tiles' boundary pipes (as we assumed they occur in reverse order).

    \textbf{Case 3:} Suppose two of the blank tiles are blocked.
    Then the right boundary pipe of the unblocked blank tile, pipe $s$, enters after and exits before the bottom boundary pipe of the unblocked blank tile, pipe $t$.
    Further, $s$ satisfies the conditions of Lemma~\ref{lem:oneFlat}.
    If $s$ serves as the $1$ pipe to a $1423$ or $13254$ pattern we can replace $s$ with $r$ in those patterns and achieve our result, therefore assume $s$ serves as the $1$ pipe to a $1432$ pattern.
    Now we consider where $t$ ends up.
    If $t$ ends up before the $4$ in the $1432$ made from $s$, then the BPD has a $132654$ pattern from adding in the $r$ and $t$ pipes.
    This pattern contains a $13254$ pattern.
    If $t$ ends up between any of the pipes of the $1432$ made from $s$, then the BPD has a $1423$ made from $r$, the $4$ from the $1432$, $t$, and the $2$ from the $1432$.
    If $t$ ends up after all the pipes in the $1432$, then the BPD has a $136542$ pattern from adding in the $r$ and $t$ pipes.
    This pattern contains a $15432$ pattern.

    \textbf{Case 4:} Suppose both blank tiles are blocked.
    Then there must be another pipe below and to the right of pipe $r$ for which the conditions this lemma apply.
    As the BPD has a finite number of pipes eventually there will be a pipe $s$ for which a case other than this case applies.
    Take the pattern achieved in that case and replace the $s$ in the pattern with the $r$ pipe to achieve the same pattern.
\end{proof}

\begin{lemma}\label{lem:doubleFlat}
    If a pipe $r$ has a $\inlinehwire$-tile, $\inlinejay$-elbow, and $\inlinevwire$-tile sandwiched between two $\inlineare$-elbows, then the BPD contains either $1423$, $13254$, $14532$, or $15432$.
\end{lemma}

\begin{figure}[h!]
    \begin{tikzpicture}[scale =.4]\grid{2}{2};\node at (2,2) {\are[Navy]};\node at (2,1) {\vwire[Navy]};\node at (0,0) {\are[Navy]};\node at (1,0) {\hwire[Navy]};\node at (2,0) {\jay[Navy]};\draw[thick] (-0.5,-0.5) -- (-0.5,2.5) -- (2.5,2.5) -- (2.5,-0.5) -- (-0.5,-0.5);\end{tikzpicture}
    \caption{An example of such a pipe $r$ from Lemma~\ref{lem:doubleFlat}}
    \label{fig:doubleFlat}
\end{figure}

\begin{proof}
    Let $r$ be a pipe as described in the lemma statement.
    In order for the $r$ pipe to have a $\inlinehwire$-tile, $\inlinejay$-elbow, and $\inlinevwire$-tile sandwiched between two $\inlineare$-elbows at some point there must of been a square of blank tiles that $r$ could droop into.
    Consider the BPD's Rothe BPD, there the $r$ pipe must have four blank tiles below and right of it.
    However, for some of those blank tiles the $r$ pipe may not be able to immediately droop into as they are blocked by other pipes.
    We will consider each possible case of blocks separately.\\
    
    \textbf{Case 1:} Suppose there is nothing blocking the $r$ pipe from drooping into the bottom right corner of a square of four blank tiles.
    The right boundary of the square must be made by two pipes entering after and exiting before the two pipes that make the bottom boundary of the square.
    Thus, the BPD has a pattern of $1423$, $14532$, or $15432$ with the $r$ pipe forming the $1$.
    
    \textbf{Case 2:} Suppose only the bottom right corner of the blank square is blocked.
    The the boundary of the rest of the square creates a $1432$ pattern with the $r$ pipe as the $1$.
    Further, there must be a blank tile below and to the right of the rest of the blank tiles in the square in order for the square to eventually be formed.
    This forces there to be two pipes greater than all the other boundary pipes that exit in reverse order.
    Therefore, the BPD has a $143265$ pattern which in turn contains a $13254$ pattern.
    
    \textbf{Case 3:} Suppose that the right half of the blank square is blocked.
    Suppose that of the blocked blank tiles there is a pipe inbetween them that entered after the $r$ pipe.
    The the pipe making the bottom boundary of the blocked blank tiles must enter after and below the inbetween pipe and the top right boundary must enter even further right and exit above the inbetween pipe.
    Therefore, the BPD has a $1423$ pattern with the $r$ pipe as the $1$, the top right boundary as the $4$, the inbetween pipe as the $2$, and the bottom boundary as the $3$.
    Suppose instead that there is no pipe inbetween the blocked blank tiles.
    This forces the two pipes creating the bottom boundary of the square to be exit below the two pipes creating the right boundary.
    Thus, the BPD has a pattern of $1423$, $14532$, or $15432$ with the $r$ pipe forming the $1$.
    
    \textbf{Case 4:} Suppose that the bottom half of the blank square is blocked.
    Note that the two pipes that create the right boundary must enter right and exit above the two pipes that create the bottom boundary of the blank tiles below the blocking pipe.
    Therefore, the BPD has a pattern of $1423$, $14532$, or $15432$ with the $r$ pipe forming the $1$.
    
    \textbf{Case 5:} Suppose that only a single blank tile is unblocked.
    This forces the bottom and right boundary pipes of this tile to be in reverse order
    In order for the blank square to eventually be formed there must be another blank tile below and right of the two pipes bounding the unblocked blank tile.
    Then that blank tile also has two boundary pipes in reverse order and as the blank tile is below and right of the boundary pipes of the unblocked blank tile so are the these boundary pipes.
    Therefore, the BPD has a pattern of $13254$ where the $1$ is the $r$ pipe, the $2$ and $3$ are the unblocked boundary pipes, and the $4$ and $5$ are boundary pipes for one of the blocked blank tiles.
    
    \textbf{Case 6:} Suppose all the blank tiles are blocked.
    Then there must be another pipe below and to the right of pipe $r$ for which the conditions this lemma apply.
    As the BPD has a finite number of pipes eventually there will be a pipe $s$ for which a case other than this case applies.
    Take the pattern achieved in that case and replace the $s$ in the pattern with the $r$ pipe to achieve the same pattern.
\end{proof}

\section{non-reduced implies pattern containment}
\label{s:suff}
We split up the proof that if a BPD contains our configuration then it must contain one of our patterns into three cases based on the relation of two of the pipes that form the configuration.
Specifically we care about the pipes that contain the bottom $\inlineare$ and the flat above that tile of the configuration, we call these pipes the {\em witness pipes}, labeled as $p$ and $q$ where $p$ enters before $q$.
Note that whatever pipe contains the bottom $\inlineare$ must necessarily also contain the bottom flat of the configuration, but the same is not true with the bottom $\inlinejay$ and top flat.
If the $p$ and $q$ pipes do not cross before the configuration and the $q$ pipe does not droop after the configuration we will show that the permutation has a $1423$ or $13254$ pattern.
If the $p$ and $q$ pipes do not cross before the configuration and the $q$ pipe droops after the configuration we will show that the permutation has a $1423$, $12543$, $13254$, $215643$, or $216543$ pattern.
If the $p$ and $q$ pipes cross before the configuration we will show that the permutation has a $1423$, $12543$, $13254$, $25143$, $216543$, or $241653$ pattern.
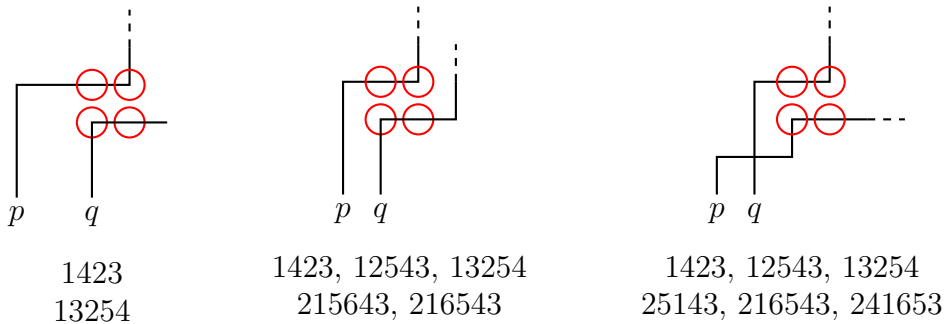
\[
\begin{tikzpicture}[scale=0.5]
    \draw[thick] (0,0) -- (0,3) -- (3,3) -- (3,4);
    \node at (0,-0.5) {$p$};
    \node at (2,-0.5) {$q$};
    \draw[red,thick] (2,3) circle (0.4);
    \draw[red,thick] (3,3) circle (0.4);
\draw[red,thick] (2,2) circle (0.4);
\draw[red,thick] (3,2) circle (0.4);

    \draw[thick,dashed] (3,4) -- (3,5);
    \draw[thick] (2,0) -- (2,2) -- (4,2);
    \node at (2,-2) {1423};
    \node at (2,-3) {13254};
\end{tikzpicture} \quad \quad \quad
\begin{tikzpicture}[scale=0.5]
    \draw[thick] (0,0) -- (0,3) -- (2,3) -- (2,4);
    \node at (0,-0.5) {$p$};
    \node at (1,-0.5) {$q$};
        \node at (1.5,-2) {1423, 12543, 13254};
    \node at (1.5,-3) {215643, 216543};
    \draw[red,thick] (1,3) circle (0.4);
    \draw[red,thick] (2,3) circle (0.4);
\draw[red,thick] (1,2) circle (0.4);
\draw[red,thick] (2,2) circle (0.4);
    \draw[thick,dashed] (2,4) -- (2,5);
    \draw[thick] (1,0) -- (1,2) -- (3,2) -- (3,3);
    \draw[thick,dashed] (3,3) -- (3,4);
\end{tikzpicture} \quad  \quad \quad
\begin{tikzpicture}[scale=0.5]
    \draw[thick] (0,1) -- (0,2) -- (2,2) -- (2,3) -- (4,3);
    \draw[thick] (1,1) -- (1,4) -- (3,4) -- (3,5);
    \node at (0,0.5) {$p$};
    \node at (1,0.5) {$q$};
    \node at (2,-1) {1423, 12543, 13254};
    \node at (2,-2) {25143, 216543, 241653};
    \draw[red,thick] (2,3) circle (0.4);
    \draw[red,thick] (3,3) circle (0.4);
\draw[red,thick] (2,4) circle (0.4);
\draw[red,thick] (3,4) circle (0.4);
    \draw[thick,dashed] (3,5) -- (3,6);
    \draw[thick,dashed] (4,3) -- (5,3);
\end{tikzpicture}
\]

\begin{proposition}\label{p:nodroop}
Suppose $\text{co}(B)$ is non-reduced for some $B \in \text{BPD}(w)$. If there is an instance of the configuration in $B$ where the bottom witness pipe has no $\inlinejay$-elbows after witnessing, then $w$ contains $1423$ or $13254$ as a pattern.
\end{proposition}

\begin{proof}
Since $\text{co}(B)$ is non-reduced, we know that $B$ has an instance of the configuration by Lemma~\ref{l:config}.
Let $p$ and $q$ be the witness pipes in our configuration, with $p$ chosen minimally and $q$ chosen maximally.
We first claim there must be a pipe $s$ with $q < s$ such that $s$ exits between $p$ and $q$.
To see this, observe that {\em some} pipe must exit out of the row, say $i_p$, where the pipe $p$ makes its $\inlinejay$-elbow of the configuration.
Note that $i_p$ must be strictly between where $p$ and $q$ exit since $q$ does not droop after.
For a pipe $v < q$ to exit out of row $i_p$, it must stay beneath $p$ and above $q$.
But observe the pipe $v$ necessarily makes a $\inlinejay$-elbow, say in row $i_v$, in order to reach row $i_p$.
So any pipe $v < q$ creates another empty row needing a pipe to exit out of.
Our claim then follows, as there must be some pipe $s$ with $q < s$ so that $s$ exits strictly between $p$ and $q$.
For the remainder of the proof, fix $s$ to be minimal among those pipes that we just showed to exist.
We now consider two cases.

\textbf{Case 1:} $p$ and $q$ cross.
As $p$ has a $\inlinejay$-tile by Lemma~\ref{l:jay-two-pipes}, there must be two pipes, say $t$ and $u$, with $p < t < u$ exiting in order $put$.
Then as there is a flat tile in the same row as the $\inlinejay$-tile of $p$ and in the same column as the $\inlineare$-elbow of $q$ we have that there is a pipe $a < q$ that exits above the $\inlinejay$-tile of $p$.
Additional, there must be a pipe that exits in the row of the $\inlinejay$-tile of $p$ that enters after $p$.
If that pipe is $u$, then we have that $a < q < t < u$ form a $1423$ pattern.
Otherwise there must be some other pipe $r$ that enters after $p$ and exits before $a$.
Therefore, $B$ has $13254$, $14253$, or $15243$ as a pattern depending on how $r$ enters in relation to $t$ and $u$.
Notice that the latter two contain $1423$ patterns.

{\bf Case 2:} $p$ and $q$ do not cross and there is a pipe $r$ with $q < r < s$ that stays beneath $q$ and exits after $q$. In this case, we see that pipes $p < q < r < s$ exit in the order $psqr$, giving a $1423$ pattern in $w$.

{\bf Case 3:} $p$ and $q$ do not cross and there is no such $r$ pipe as described above. 
Then each pipe $v$ with $q < v < s$ necessarily exits before $p$ since $s$ was chosen minimally.
We claim that the pipe $s$ must make a $\inlinejay$-elbow.
Suppose the $\inlinejay$-elbow of pipe $p$ in the configuration occurs in column $c_p$. 
Then we know $q < c_p \leq s$. If $c_p = s$ then $s$ must make a $\inlineare$-elbow before crossing $q$ since $q$ must have a $\inlinehwire$-tile in column $c_p$ as witness. 
But $s$ exits before $q$, therefore $s$ must make a $\inlinejay$-elbow at some point, as advertised.
Otherwise, if $c_p < s$, there is some pipe $v$ with $q < v < s$. 
This pipe $v$ must make a $\inlineare$-elbow before crossing $q$ for the same reason. 
But since $v$ must exit before $p$, we have that $v$ makes a $\inlinejay$-elbow at some point, say in column $c_v$.
Note that $c_v \leq s$, because otherwise $s$ would necessarily have a double-crossing with $v$.
So repeating the same argument as we did for $c_p \leq s$, it follows that $s$ must make a $\inlinejay$-elbow.
Further this $\inlinejay$-elbow must occur below $q$ as otherwise there would be an $\inlinejay$-elbow from $s$ above $p$ which is covered by Case 1.
Hence by Lemma~\ref{l:jay-two-pipes}, there must be two pipes, say $t$ and $u$, with $s < t < u$ exiting in the order $sut$.
Further, we have that $u$ and $t$ must exit below $q$.
Overall, we have $p < q < s < t < u$ that exit in the order $psqut$, which is an occurrence of a $13254$ pattern in $w$.
\end{proof}

\begin{proposition}\label{p:cross}
Suppose $\text{co}(B)$ is non-reduced for some $B \in \text{BPD}(w)$. If the witness pipes of a configuration in $B$ cross before the occurrence, then $w$ contains $1423$, $12543$, $13254$, $25143$, $216543$, or $241653$ as a pattern.
\end{proposition}

\begin{proof}
    Let $p$ and $q$ be the witness pipes of a configuration in $B$ where $p$ and $q$ cross before creating the configuration.
    We may assume $B$ is reduced meaning that $p$ must exit after $q$.
    Note that as $p$ and $q$ cross before the $\inlineare$-elbow $p$ contributes to the configuration $p$ must have a $\inlinejay$-elbow in the same column.\\

    \textbf{Case 1:} Assume that any pipes that exit between $q$ and $p$ entered before the $\inlineare$-elbow $p$ contributes to the configuration or between $\inlineare$-elbow and $\inlinejay$-elbow of the configuration.
    This requires $q$ to have a $\inlinejay$-tile after the flat in the configuration as the $p$ must exit in the row that contains either $q$'s last $\inlinejay$-tile or the $\inlinejay$-tile of a pipe exiting between $p$ and $q$.
    By then undrooping each pipe that is both above and left the $p$ and $q$ pipes we can then undroop the $\inlinejay$-tile of $p$ that is in the configuration.
    As a result $p$ will have two flat tiles prior to the $\inlinejay$-tile after the flat from the configuration meaning that we can apply Lemma~\ref{lem:twoFlats} to $p$.
    Adding the $q$ pipe to the $15432$ pattern we get that $B$ has a $1423$, $13254$, or $216543$ pattern.

    \textbf{Case 2:} Assume that the pipe $t$ which enters directly beneath the $\inlineare$-elbow $p$ contributes to the configuration exits between $p$ and $q$.
    Note that $t$ must have a $\inlineare$-elbow in the column of $p$'s $\inlinejay$-elbow in order to exit above that $\inlinejay$-elbow.
    Further, $t$ must have a $\inlinejay$-elbow after the flat $p$ contributes to the configuration, otherwise $t$ would form the configuration with $q$ (or could be undrooped to form the configuration) which is covered by Propositions~\ref{p:nodroop} and \ref{p:droop}.
    Either there is a flat in $t$ before this $\inlinejay$-elbow or $t$ has an earlier $\inlinejay$-elbow.
    In the latter case, by undrooping any pipes above and left of $p$ we can then undroop the $\inlinejay$-elbow $p$ contributes to the configuration.
    We can then undroop the earlier $\inlinejay$-elbow of $t$ to create a flat before the other $\inlinejay$-elbow.
    Therefore, we can apply Lemma~\ref{lem:oneFlat} to $t$.
    If $t$ take part in $1423$ or $13254$ pattern we are done, so assume $t$ forms a $1432$ pattern.
    By appending $q$ to the front of this pattern we get $B$ has a $12543$ pattern.

    \textbf{Case 3:} Assume that the pipe $t$ which enters directly beneath the $\inlinejay$-elbow $q$ contributes to the configuration exits between $p$ and $q$.
    Note that this means $t$ must be drooped, further the $\inlinejay$-tile created must be right of the $\inlinejay$ in the configuration.
    In order for $t$ to be drooped we must have had two pipes in reverse order below and after $t$.
    Further, there must always be a pipe that enters after $q$ and before $t$ that exits below $p$.
    To see this consider the pipe $s$ that exits directly below the $\inlineare$-elbow $p$ contributes to the configuration, by assumption this pipe either exits below $p$ or above $q$.
    In the former case we are done.
    In the latter case note that $s$ must have a $\inlinejay$-elbow left of the column $t$ enters, otherwise $s$ would double cross $t$.
    Repeat this argument for the pipe that enters in the column of that $\inlinejay$-elbow.
    Therefore, we have that the BPD has a $241365$, $241635$, or $241653$ pattern, with the former two containing a $1423$ pattern.

    \textbf{Case 4:} Assume that there is a pipe $t$ that enters after the $\inlinejay$-elbow $q$ contributes to the configuration and exits between $p$ and $q$.
    Note that there must always be two pipes that enter afer $q$ and before $t$ that exit below $q$.
    To see this consider the pipe $s$ that exits directly below the $\inlineare$-elbow $p$ contributes to the configuration, by assumption this pipe either exits below $p$ or above $q$.
    In the former case we are done.
    In the latter case note that $s$ must have a $\inlinejay$-elbow left of the column $t$ enters, otherwise $s$ would double cross $t$.
    Repeat this argument for the pipe that enters in the column of that $\inlinejay$-elbow.
    Further we can repeat this exact same argument for the pipe that exits directly below the $\inlinejay$-elbow $q$ contributes to the configuration, to get another pipe exiting below $q$ that enters in a different column from the other pipe exiting below $q$.
    Therefore, we have that the BPD has a $25143$ or $25134$ pattern.
\end{proof}

\newpage

\begin{proposition}\label{p:droop}
Suppose $\text{co}(B)$ is non-reduced for some $B \in \text{BPD}(w)$. If the witness pipes of a configuration in $B$ do not cross before witnessing and the bottom witness pipe makes a $\inlinejay$-elbow at some point after witnessing, then $w$ contains $1423$, $12543$, $13254$, $215643$, or $216543$ as a pattern.
\end{proposition}

\begin{proof}
    Since $\text{co}(B)$ is non-reduced, we know that $B$ has an instance of the configuration by Lemma~\ref{l:config}.
    Let $p$ and $q$ be the witness pipes in our configuration, with $p$ chosen minimally and $q$ chosen maximally.\\

    \textbf{Case 1:} Suppose pipes $p$ and $q$ never cross in $B$.
    Note that as $q$ has a flat from the configuration and then has an $\inlineare$-tile, from Lemma~\ref{lem:oneFlat} that $B$ contains a $1423$, $1432$, or $13254$ pattern with $q$ as the $1$ pipe.
    If $B$ has a $1423$ or $13254$ pattern we are done, so suppose $B$ has a $1432$ pattern.
    As $q$ is the $1$ pipe in the $1432$ pattern of $B$ and $p$ enters and exits before $q$ we have that $B$ has a $12543$ pattern.

    \textbf{Case 2:} Suppose pipes $p$ and $q$ cross in $B$, by assumption they must cross after they form the configuration.
    Assume that $q$ doesn't have a $\inlinevwire$-tile between the $\inlinejay$-tile that occurs after the flat in the configuration and the next $\inlineare$-tile of $q$.
    In this case $q$ must have another $\inlineare$-tile in order to cross $p$, this means that we if we undroop every pipe that enters and exits before $p$ then we can undroop that $\inlineare$-tile to create a $\inlinevwire$ as described earlier.
    Therefore, $B$ satisfies the conditions of Lemma~\ref{lem:doubleFlat} with pipe $q$.
    Along with the fact that in all those patterns we can add $p$ as the second pipe in the pattern, we get that $B$ has a $1423$, $13254$, $215643$, or $216543$ pattern.
\end{proof}

\begin{theorem}\label{thm:avoidImpliesReduced}
    If $w$ avoids all seven patterns in $\Pi$, then co-BPD$(w)$ has only reduced co-BPDs.
\end{theorem}

\begin{proof}
    By Propositions~\ref{p:nodroop}, \ref{p:cross}, or \ref{p:droop}, the contrapositive the theorem statement follows.
\end{proof}

\section{pattern implies non-reduced}\label{s:ness}

\begin{proposition}\label{prop:patt}
    If $\pi \in \Pi$, then there exists $B\in \BPD{\pi}$ such that co$(B)$ is non-reduced.
\end{proposition}

\begin{proof}
In Appendix~\ref{a:droops} there is a table along with pictures giving which droop moves to perform to get a non-reduced co-BPD for each permutation in the proposition statement.
\end{proof}

Let $D$ be a series of droop moves for the Rothe BPD of $\pi$.
Given a permutation $w$ with the pattern $\pi$ we can attempt the droop moves $D$ of $\pi$ by ignoring the pipes of $w$.
That is changing $(a,b)\ssearrow (c,d)$ to $(a',b')\ssearrow(c',d')$ where $a'$ (resp. $c'$) is $a$ (resp. $c$) plus the number of pipes not in $\pi$ that exit above pipe $a/c$ in $\pi$, $b'/d'$ is $b/d$ plus the number of pipes not in $\pi$ that exit left of pipe $b/d$ in $\pi$.
This new series of droop moves $D'$ is valid unless there is a pipe that enters and exits between the elbow and the blank tile of a desired droop move, which recall is a blocking pipe.
For this section, we say a blocking pipe is of the form $(x,y)$ where $x$ is the number of pipes in the original pattern above the blocking pipe and $y$ is the number of pipes in the original pattern to the left.
Note that for a droop move in $D$, $(a,b)\ssearrow (c,d)$, a blocking pipe for this droop move can be of the form $(x,y)$ for any $a\leq x < c$ and $b\leq y < d$.

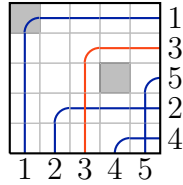
\begin{figure}[h!]
\begin{tikzpicture}[scale =.4]
\node at (0,4) {\shade};
\node at (0,4) {\oldare};
\node at (1,4) {\oldhwire};
\node at (2,4) {\oldhwire};
\node at (3,4) {\oldhwire};
\node at (4,4) {\oldhwire};
\node at (0,3) {\oldvwire};
\node at (1,3) {\nowire};
\node at (2,3) {\oldare[UF2]};
\node at (3,3) {\oldhwire[UF2]};
\node at (4,3) {\oldhwire[UF2]};
\node at (0,2) {\oldvwire};
\node at (1,2) {\nowire};
\node at (2,2) {\oldvwire[UF2]};
\node at (3,2) {\shade};
\node at (3,2) {\nowire};
\node at (4,2) {\oldare};
\node at (0,1) {\oldvwire};
\node at (1,1) {\oldare};
\node at (2,1) {\oldhwire};
\node at (2,1) {\oldvwire[UF2]};
\node at (3,1) {\oldhwire};
\node at (4,1) {\cross};
\node at (0,0) {\oldvwire};
\node at (1,0) {\oldvwire};
\node at (2,0) {\oldvwire[UF2]};
\node at (3,0) {\oldare};
\node at (4,0) {\cross};

\draw[thick] (-0.5,-0.5) -- (-0.5,4.5) -- (4.5,4.5) -- (4.5,-0.5) -- (-0.5,-0.5);

\node at (0,-1) {1};
\node at (1,-1) {2};
\node at (2,-1) {3};
\node at (3,-1) {4};
\node at (4,-1) {5};

\node at (5,4) {1};
\node at (5,3) {3};
\node at (5,2) {5};
\node at (5,1) {2};
\node at (5,0) {4};

\end{tikzpicture}
\caption{The $3$ pipe is a $(1,2)$ blocking pipe for the $1423$ pattern's droop move of $(1,1)\ssearrow(2,3)$.}
\end{figure}

\newpage

\begin{lemma}\label{lem:blocking}
    If $w$ contains a pattern $\pi \in \Pi$ and contains no blocking pipes for the corresponding pattern's droop moves to get a non-reduced co-BPD, then there exists $B\in \BPD{w}$ such that co$(B)$ is non-reduced.
\end{lemma}

\begin{proof}
    Let $\pi$ be the pattern $w$ contains.
    Then perform the series of droop moves to get a non-reduced co-BPD for the pattern adjusted for the extra pipes by increasing each row by the number of pipes in $w$ not a part of $\pi$ that exit above the drooped pipe and increasing each column by the number of pipes in $w$ not part of $\pi$ that enter left of the droop pipe.
    As there is no blocking pipe this series of droop moves is valid, further the configuration is preserved as the only difference is potentially the length of the droop move which changes nothing about the tiles involved in the configuration.
\end{proof}

\begin{proposition}\label{p:1423}
    If $w$ contains a pattern $\pi=1423$, then there exists $B\in \BPD{w}$ such that co$(B)$ is non-reduced.
\end{proposition}

\begin{proof}
    By Proposition~\ref{prop:patt} the BPD made from the Rothe diagram for $1423$ by the droop move $(1,1) \ssearrow (2,3)$ has a non-reduced co-BPD.
    We will assume that $w$ contains blocking pipes as otherwise the result is true by Lemma~\ref{lem:blocking}.
    If there is a blocking pipe of the form $(1,1)$ take that pipe as the $1$ in the pattern instead to get the pattern to have one less blocking pipe.
    Thus, we can assume all the blocking pipes are of the form $(1,2)$.
    
    From such a Rothe BPD we can get a BPD with non-reduced co-BPD by drooping the right-most blocking pipe into where the 1 pattern pipe of the pattern would droop if unblocked, then attempt to droop the 1 pattern pipe to where the elbow for the drooped blocking pipe was.
    Note that the flat from the 2 in the pattern will be right of the $\inlineare$ from the same pipe which in turn will be directly below the flat made when the 1 in the pattern eventually drops.
    Further, that flat will be below a $\inlinejay$ elbow from the last blocking pipe, which in turn is below a $\inlineare$ elbow from the same pipe.
    That elbow is then right of the $\inlinejay$ from the next blocking pipe if one exists.
    This pattern then continues until the $\inlineare$ elbow from a blocking pipe is right of the $\inlinejay$ elbow form the 1 in the pattern.
    Therefore, the BPD has the configuration.
\end{proof}

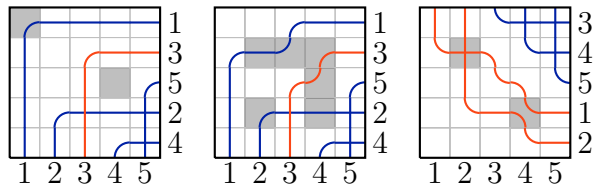
\begin{figure}[h!]
        \centering

        \begin{tikzpicture}[scale =.4]\node at (3,2) {\shade};\node at (0,4) {\shade};\grid{4}{4}\node at (0,4) {\are[Navy]};\node at (1,4) {\hwire[Navy]};\node at (2,4) {\hwire[Navy]};\node at (3,4) {\hwire[Navy]};\node at (4,4) {\hwire[Navy]};\node at (0,3) {\vwire[Navy]};\node at (1,3) {\nowire};\node at (2,3) {\are[Orange]};\node at (3,3) {\hwire[Orange]};\node at (4,3) {\hwire[Orange]};\node at (0,2) {\vwire[Navy]};\node at (1,2) {\nowire};\node at (2,2) {\vwire[Orange]};\node at (3,2) {\nowire};\node at (4,2) {\are[Navy]};\node at (0,1) {\vwire[Navy]};\node at (1,1) {\are[Navy]};\node at (2,1) {\hwire[Navy]};\node at (2,1) {\vwire[Orange]};\node at (3,1) {\hwire[Navy]};\node at (4,1) {\hwire[Navy]};\node at (4,1) {\vwire[Navy]};\node at (0,0) {\vwire[Navy]};\node at (1,0) {\vwire[Navy]};\node at (2,0) {\vwire[Orange]};\node at (3,0) {\are[Navy]};\node at (4,0) {\hwire[Navy]};\node at (4,0) {\vwire[Navy]};\draw[thick] (-0.5,-0.5) -- (-0.5,4.5) -- (4.5,4.5) -- (4.5,-0.5) -- (-0.5,-0.5);\node at (0,-1) {1};\node at (5,4) {1};\node at (1,-1) {2};\node at (5,3) {3};\node at (2,-1) {3};\node at (5,2) {5};\node at (3,-1) {4};\node at (5,1) {2};\node at (4,-1) {5};\node at (5,0) {4};\end{tikzpicture}
        \begin{tikzpicture}[scale =.4]\node at (1,3) {\shade};\node at (1,1) {\shade};\node at (3,1) {\shade};\node at (3,2) {\shade};\node at (3,3) {\shade};\node at (2,3) {\shade};\grid{4}{4}\node at (0,4) {\nowire};\node at (1,4) {\nowire};\node at (2,4) {\are[Navy]};\node at (3,4) {\hwire[Navy]};\node at (4,4) {\hwire[Navy]};\node at (0,3) {\are[Navy]};\node at (1,3) {\hwire[Navy]};\node at (2,3) {\jay[Navy]};\node at (3,3) {\are[Orange]};\node at (4,3) {\hwire[Orange]};\node at (0,2) {\vwire[Navy]};\node at (1,2) {\nowire};\node at (2,2) {\are[Orange]};\node at (3,2) {\jay[Orange]};\node at (4,2) {\are[Navy]};\node at (0,1) {\vwire[Navy]};\node at (1,1) {\are[Navy]};\node at (2,1) {\hwire[Navy]};\node at (2,1) {\vwire[Orange]};\node at (3,1) {\hwire[Navy]};\node at (4,1) {\hwire[Navy]};\node at (4,1) {\vwire[Navy]};\node at (0,0) {\vwire[Navy]};\node at (1,0) {\vwire[Navy]};\node at (2,0) {\vwire[Orange]};\node at (3,0) {\are[Navy]};\node at (4,0) {\hwire[Navy]};\node at (4,0) {\vwire[Navy]};\draw[thick] (-0.5,-0.5) -- (-0.5,4.5) -- (4.5,4.5) -- (4.5,-0.5) -- (-0.5,-0.5);\node at (0,-1) {1};\node at (5,4) {1};\node at (1,-1) {2};\node at (5,3) {3};\node at (2,-1) {3};\node at (5,2) {5};\node at (3,-1) {4};\node at (5,1) {2};\node at (4,-1) {5};\node at (5,0) {4};\end{tikzpicture}
        \begin{tikzpicture}[scale =.4]\node at (1,3) {\shade};\node at (3,1) {\shade};\grid{4}{4}\node at (0,4) {\vwire[Orange]};\node at (1,4) {\vwire[Orange]};\node at (2,4) {\el[Navy]};\node at (3,4) {\hwire[Navy]};\node at (3,4) {\vwire[Navy]};\node at (4,4) {\hwire[Navy]};\node at (4,4) {\vwire[Navy]};\node at (0,3) {\el[Orange]};\node at (1,3) {\hwire[Orange]};\node at (1,3) {\vwire[Orange]};\node at (2,3) {\en[Orange]};\node at (3,3) {\el[Navy]};\node at (4,3) {\hwire[Navy]};\node at (4,3) {\vwire[Navy]};\node at (0,2) {\nowire};\node at (1,2) {\vwire[Orange]};\node at (2,2) {\el[Orange]};\node at (3,2) {\en[Orange]};\node at (4,2) {\el[Navy]};\node at (0,1) {\nowire};\node at (1,1) {\el[Orange]};\node at (2,1) {\hwire[Orange]};\node at (3,1) {\newire[Orange]};\node at (3,1) {\swwire[Orange]};\node at (4,1) {\hwire[Orange]};\node at (0,0) {\nowire};\node at (1,0) {\nowire};\node at (2,0) {\nowire};\node at (3,0) {\el[Orange]};\node at (4,0) {\hwire[Orange]};\draw[thick] (-0.5,-0.5) -- (-0.5,4.5) -- (4.5,4.5) -- (4.5,-0.5) -- (-0.5,-0.5);\node at (0,-1) {1};\node at (5,4) {3};\node at (1,-1) {2};\node at (5,3) {4};\node at (2,-1) {3};\node at (5,2) {5};\node at (3,-1) {4};\node at (5,1) {1};\node at (4,-1) {5};\node at (5,0) {2};\end{tikzpicture}
        
        \caption{A $1423$ pattern with a $(1,2)$ blocking pipe.}
        \label{fig:1423Blocks}
    \end{figure}

\begin{proposition}\label{p:13254}
    If $w$ contains a pattern $\pi=13254$, then there exists $B\in \BPD{w}$ such that co$(B)$ is non-reduced.
\end{proposition}

\begin{proof}
    By Proposition~\ref{prop:patt} the BPD made from Rothe diagram for $13254$ by the droop moves $(2,3) \ssearrow (4,4)$ and $(1,1)\ssearrow (2,3)$ has a non-reduced co-BPD.
    We will assume that $w$ contains blocking pipes as otherwise the result is true by Lemma~\ref{lem:blocking}.
    If there is a blocking pipe of the form $(1,1)$ take that pipe as the $1$ in the pattern instead to get the pattern to have one less blocking pipe.
    Note that blocking pipes of the form $(2,3)$ can instead be considered to be the $3$ pipe of the pattern resulting in the old $3$ becoming a $(1,2)$ blocking pipe.
    Thus, we can assume all the blocking pipes are of the form $(1,2)$ or $(3,3)$.
    
    From such a Rothe BPD we can get a BPD with non-reduced co-BPD by drooping the left most $(3,3)$ blocking pipe into where the 3 pattern pipe of the pattern would droop if unblocked, then attempt to droop the 3 pattern pipe to where the elbow for the drooped blocking pipe was.
    If the 3 pattern pipe is still block repeat this process.
    Then attempt to droop the 1 pattern pipe to where the elbow for the 3 pattern pipe was.
    If this droop move is blocked by a $(1,2)$ blocking pipe instead droop the right most blocking pipe and then attempt to droop the 1 pattern pipe to where the elbow for the blocking pipe was.
    Repeat this process until the 1 pattern pipe is drooped.
    By dropping where the old elbow was we ensure that we create the chain of elbows required for the configuration, all that remains to check is that the flats witness elbows.
    Then by drooping the $(3,3)$ blocking pipes and the 3 pattern pipe eventually a empty space is opened above a flat in the 2 pattern pipe.
    Then either the 1 pattern pipe or a $(1,2)$ blocking pipe creates $\inlinejay$ above that flat and the flat has a $\inlineare$ to its right.
    Further, note that last $(1,2)$ blocking pipe will droop from a row that has an empty box above the $\inlineare$ from the 2 pattern pipe.
    Therefore, when the 1 pipe droops it will create a flat above the $\inlineare$ from the 2 pattern pipe and the flat will be left of the $\inlinejay$ created from the droop.
    Therefore, the BPD has the configuration.

    \begin{figure}[h!]
        \centering

        \begin{tikzpicture}[scale =.4]\node at (5,1) {\shade};\node at (3,4) {\shade};\grid{6}{6}\node at (0,6) {\are[Navy]};\node at (1,6) {\hwire[Navy]};\node at (2,6) {\hwire[Navy]};\node at (3,6) {\hwire[Navy]};\node at (4,6) {\hwire[Navy]};\node at (5,6) {\hwire[Navy]};\node at (6,6) {\hwire[Navy]};\node at (0,5) {\vwire[Navy]};\node at (1,5) {\nowire};\node at (2,5) {\are[Orange]};\node at (3,5) {\hwire[Orange]};\node at (4,5) {\hwire[Orange]};\node at (5,5) {\hwire[Orange]};\node at (6,5) {\hwire[Orange]};\node at (0,4) {\vwire[Navy]};\node at (1,4) {\nowire};\node at (2,4) {\vwire[Orange]};\node at (3,4) {\are[Navy]};\node at (4,4) {\hwire[Navy]};\node at (5,4) {\hwire[Navy]};\node at (6,4) {\hwire[Navy]};\node at (0,3) {\vwire[Navy]};\node at (1,3) {\are[Navy]};\node at (2,3) {\hwire[Navy]};\node at (2,3) {\vwire[Orange]};\node at (3,3) {\hwire[Navy]};\node at (3,3) {\vwire[Navy]};\node at (4,3) {\hwire[Navy]};\node at (5,3) {\hwire[Navy]};\node at (6,3) {\hwire[Navy]};\node at (0,2) {\vwire[Navy]};\node at (1,2) {\vwire[Navy]};\node at (2,2) {\vwire[Orange]};\node at (3,2) {\vwire[Navy]};\node at (4,2) {\are[Orange]};\node at (5,2) {\hwire[Orange]};\node at (6,2) {\hwire[Orange]};\node at (0,1) {\vwire[Navy]};\node at (1,1) {\vwire[Navy]};\node at (2,1) {\vwire[Orange]};\node at (3,1) {\vwire[Navy]};\node at (4,1) {\vwire[Orange]};\node at (5,1) {\nowire};\node at (6,1) {\are[Navy]};\node at (0,0) {\vwire[Navy]};\node at (1,0) {\vwire[Navy]};\node at (2,0) {\vwire[Orange]};\node at (3,0) {\vwire[Navy]};\node at (4,0) {\vwire[Orange]};\node at (5,0) {\are[Navy]};\node at (6,0) {\hwire[Navy]};\node at (6,0) {\vwire[Navy]};\draw[thick] (-0.5,-0.5) -- (-0.5,6.5) -- (6.5,6.5) -- (6.5,-0.5) -- (-0.5,-0.5);\node at (0,-1) {1};\node at (7,6) {1};\node at (1,-1) {2};\node at (7,5) {3};\node at (2,-1) {3};\node at (7,4) {4};\node at (3,-1) {4};\node at (7,3) {2};\node at (4,-1) {5};\node at (7,2) {5};\node at (5,-1) {6};\node at (7,1) {7};\node at (6,-1) {7};\node at (7,0) {6};\end{tikzpicture}
        \begin{tikzpicture}[scale =.4]\node at (3,4) {\shade};\node at (0,6) {\shade};\grid{6}{6}\node at (0,6) {\are[Navy]};\node at (1,6) {\hwire[Navy]};\node at (2,6) {\hwire[Navy]};\node at (3,6) {\hwire[Navy]};\node at (4,6) {\hwire[Navy]};\node at (5,6) {\hwire[Navy]};\node at (6,6) {\hwire[Navy]};\node at (0,5) {\vwire[Navy]};\node at (1,5) {\nowire};\node at (2,5) {\are[Orange]};\node at (3,5) {\hwire[Orange]};\node at (4,5) {\hwire[Orange]};\node at (5,5) {\hwire[Orange]};\node at (6,5) {\hwire[Orange]};\node at (0,4) {\vwire[Navy]};\node at (1,4) {\nowire};\node at (2,4) {\vwire[Orange]};\node at (3,4) {\nowire};\node at (4,4) {\are[Navy]};\node at (5,4) {\hwire[Navy]};\node at (6,4) {\hwire[Navy]};\node at (0,3) {\vwire[Navy]};\node at (1,3) {\are[Navy]};\node at (2,3) {\hwire[Navy]};\node at (2,3) {\vwire[Orange]};\node at (3,3) {\hwire[Navy]};\node at (4,3) {\hwire[Navy]};\node at (4,3) {\vwire[Navy]};\node at (5,3) {\hwire[Navy]};\node at (6,3) {\hwire[Navy]};\node at (0,2) {\vwire[Navy]};\node at (1,2) {\vwire[Navy]};\node at (2,2) {\vwire[Orange]};\node at (3,2) {\are[Navy]};\node at (4,2) {\jay[Navy]};\node at (5,2) {\are[Orange]};\node at (6,2) {\hwire[Orange]};\node at (0,1) {\vwire[Navy]};\node at (1,1) {\vwire[Navy]};\node at (2,1) {\vwire[Orange]};\node at (3,1) {\vwire[Navy]};\node at (4,1) {\are[Orange]};\node at (5,1) {\jay[Orange]};\node at (6,1) {\are[Navy]};\node at (0,0) {\vwire[Navy]};\node at (1,0) {\vwire[Navy]};\node at (2,0) {\vwire[Orange]};\node at (3,0) {\vwire[Navy]};\node at (4,0) {\vwire[Orange]};\node at (5,0) {\are[Navy]};\node at (6,0) {\hwire[Navy]};\node at (6,0) {\vwire[Navy]};\draw[thick] (-0.5,-0.5) -- (-0.5,6.5) -- (6.5,6.5) -- (6.5,-0.5) -- (-0.5,-0.5);\node at (0,-1) {1};\node at (7,6) {1};\node at (1,-1) {2};\node at (7,5) {3};\node at (2,-1) {3};\node at (7,4) {4};\node at (3,-1) {4};\node at (7,3) {2};\node at (4,-1) {5};\node at (7,2) {5};\node at (5,-1) {6};\node at (7,1) {7};\node at (6,-1) {7};\node at (7,0) {6};\end{tikzpicture}
        \begin{tikzpicture}[scale =.4]\node at (1,5) {\shade};\node at (1,3) {\shade};\node at (3,4) {\shade};\node at (3,3) {\shade};\node at (3,5) {\shade};\node at (2,5) {\shade};\grid{6}{6}\node at (0,6) {\nowire};\node at (1,6) {\nowire};\node at (2,6) {\are[Navy]};\node at (3,6) {\hwire[Navy]};\node at (4,6) {\hwire[Navy]};\node at (5,6) {\hwire[Navy]};\node at (6,6) {\hwire[Navy]};\node at (0,5) {\are[Navy]};\node at (1,5) {\hwire[Navy]};\node at (2,5) {\jay[Navy]};\node at (3,5) {\are[Orange]};\node at (4,5) {\hwire[Orange]};\node at (5,5) {\hwire[Orange]};\node at (6,5) {\hwire[Orange]};\node at (0,4) {\vwire[Navy]};\node at (1,4) {\nowire};\node at (2,4) {\are[Orange]};\node at (3,4) {\jay[Orange]};\node at (4,4) {\are[Navy]};\node at (5,4) {\hwire[Navy]};\node at (6,4) {\hwire[Navy]};\node at (0,3) {\vwire[Navy]};\node at (1,3) {\are[Navy]};\node at (2,3) {\hwire[Navy]};\node at (2,3) {\vwire[Orange]};\node at (3,3) {\hwire[Navy]};\node at (4,3) {\hwire[Navy]};\node at (4,3) {\vwire[Navy]};\node at (5,3) {\hwire[Navy]};\node at (6,3) {\hwire[Navy]};\node at (0,2) {\vwire[Navy]};\node at (1,2) {\vwire[Navy]};\node at (2,2) {\vwire[Orange]};\node at (3,2) {\are[Navy]};\node at (4,2) {\jay[Navy]};\node at (5,2) {\are[Orange]};\node at (6,2) {\hwire[Orange]};\node at (0,1) {\vwire[Navy]};\node at (1,1) {\vwire[Navy]};\node at (2,1) {\vwire[Orange]};\node at (3,1) {\vwire[Navy]};\node at (4,1) {\are[Orange]};\node at (5,1) {\jay[Orange]};\node at (6,1) {\are[Navy]};\node at (0,0) {\vwire[Navy]};\node at (1,0) {\vwire[Navy]};\node at (2,0) {\vwire[Orange]};\node at (3,0) {\vwire[Navy]};\node at (4,0) {\vwire[Orange]};\node at (5,0) {\are[Navy]};\node at (6,0) {\hwire[Navy]};\node at (6,0) {\vwire[Navy]};\draw[thick] (-0.5,-0.5) -- (-0.5,6.5) -- (6.5,6.5) -- (6.5,-0.5) -- (-0.5,-0.5);\node at (0,-1) {1};\node at (7,6) {1};\node at (1,-1) {2};\node at (7,5) {3};\node at (2,-1) {3};\node at (7,4) {4};\node at (3,-1) {4};\node at (7,3) {2};\node at (4,-1) {5};\node at (7,2) {5};\node at (5,-1) {6};\node at (7,1) {7};\node at (6,-1) {7};\node at (7,0) {6};\end{tikzpicture}
        \begin{tikzpicture}[scale =.4]\node at (1,5) {\shade};\node at (3,3) {\shade};\grid{6}{6}\node at (0,6) {\vwire[Orange]};\node at (1,6) {\vwire[Orange]};\node at (2,6) {\el[Navy]};\node at (3,6) {\hwire[Navy]};\node at (3,6) {\vwire[Navy]};\node at (4,6) {\hwire[Navy]};\node at (4,6) {\vwire[Navy]};\node at (5,6) {\hwire[Navy]};\node at (5,6) {\vwire[Navy]};\node at (6,6) {\hwire[Navy]};\node at (6,6) {\vwire[Navy]};\node at (0,5) {\el[Orange]};\node at (1,5) {\hwire[Orange]};\node at (1,5) {\vwire[Orange]};\node at (2,5) {\en[Orange]};\node at (3,5) {\el[Navy]};\node at (4,5) {\hwire[Navy]};\node at (4,5) {\vwire[Navy]};\node at (5,5) {\hwire[Navy]};\node at (5,5) {\vwire[Navy]};\node at (6,5) {\hwire[Navy]};\node at (6,5) {\vwire[Navy]};\node at (0,4) {\nowire};\node at (1,4) {\vwire[Orange]};\node at (2,4) {\el[Orange]};\node at (3,4) {\en[Orange]};\node at (4,4) {\el[Navy]};\node at (5,4) {\hwire[Navy]};\node at (5,4) {\vwire[Navy]};\node at (6,4) {\hwire[Navy]};\node at (6,4) {\vwire[Navy]};\node at (0,3) {\nowire};\node at (1,3) {\el[Orange]};\node at (2,3) {\hwire[Orange]};\node at (3,3) {\newire[Orange]};\node at (3,3) {\swwire[Orange]};\node at (4,3) {\hwire[Orange]};\node at (5,3) {\hwire[Orange]};\node at (5,3) {\vwire[Navy]};\node at (6,3) {\hwire[Orange]};\node at (6,3) {\vwire[Navy]};\node at (0,2) {\nowire};\node at (1,2) {\nowire};\node at (2,2) {\nowire};\node at (3,2) {\el[Orange]};\node at (4,2) {\en[Orange]};\node at (5,2) {\el[Navy]};\node at (6,2) {\hwire[Navy]};\node at (6,2) {\vwire[Navy]};\node at (0,1) {\nowire};\node at (1,1) {\nowire};\node at (2,1) {\nowire};\node at (3,1) {\nowire};\node at (4,1) {\el[Orange]};\node at (5,1) {\en[Orange]};\node at (6,1) {\el[Navy]};\node at (0,0) {\nowire};\node at (1,0) {\nowire};\node at (2,0) {\nowire};\node at (3,0) {\nowire};\node at (4,0) {\nowire};\node at (5,0) {\el[Orange]};\node at (6,0) {\hwire[Orange]};\draw[thick] (-0.5,-0.5) -- (-0.5,6.5) -- (6.5,6.5) -- (6.5,-0.5) -- (-0.5,-0.5);\node at (0,-1) {1};\node at (7,6) {3};\node at (1,-1) {2};\node at (7,5) {4};\node at (2,-1) {3};\node at (7,4) {5};\node at (3,-1) {4};\node at (7,3) {1};\node at (4,-1) {5};\node at (7,2) {6};\node at (5,-1) {6};\node at (7,1) {7};\node at (6,-1) {7};\node at (7,0) {2};\end{tikzpicture}

        \caption{A $13254$ pattern with $(1,2)$ and $(3,3)$ blocking pipes.}
        \label{fig:12543Blocks}
    \end{figure}
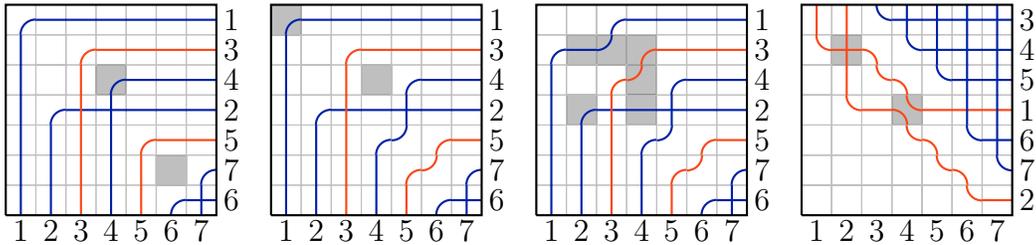
\end{proof}

\begin{proposition}\label{p:12543}
    If $w$ contains a pattern $\pi=12543$, then there exists $B\in \BPD{w}$ such that co$(B)$ is non-reduced.
\end{proposition}

\begin{proof}
    By Proposition~\ref{prop:patt} the BPD made from Rothe diagram for $12543$ by the droop moves $(2,2) \ssearrow (3,4)$ and $(1,1)\ssearrow (2,3)$ has a non-reduced co-BPD.
    We will assume that $w$ contains blocking pipes as otherwise the result is true by Lemma~\ref{lem:blocking}.
    If there is a blocking pipe of the form $(1,2)$ then the blocking pipe forms a $13254$ pattern with the $1$, $2$, $4$, and $5$ pipes from $\pi$, thus by Proposition~\ref{p:13254} we have our result.
    If there is a blocking pipe of the form $(1,1)$ take that pipe as the $1$ in the pattern instead to get the pattern to have one less blocking pipe.
    If there is a blocking pipe of the form $(2,2)$ take that pipe as the $2$ in the pattern instead to get the pattern to have one less blocking pipe.
    Thus, we can assume all the blocking pipes are of the form $(2,3)$.
    
    From such a Rothe BPD we can get a BPD with non-reduced co-BPD by drooping the right-most blocking pipe into where the 2 pattern pipe of the pattern would droop if unblocked, then attempt to droop the 2 pattern pipe to where the elbow for the drooped blocking pipe was.
    If the 2 pattern pipe is blocked then repeat this process.
    Note that after drooping all the blocking pipes that the 2 pattern pipe will droop into a two empty blocks in the same row.
    Thus, when the 2 pattern pipe droops into that block it will create another two empty blocks for the 1 pattern pipe to droop into which will create the configuration.
\end{proof}

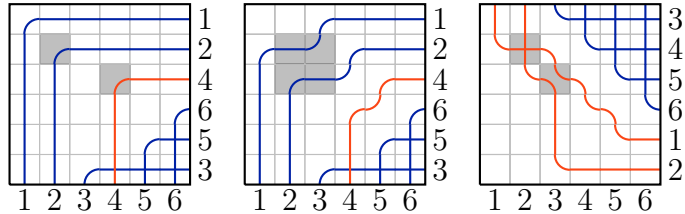
\begin{figure}[h!]
        \centering

        \begin{tikzpicture}[scale =.4]\node at (1,4) {\shade};\node at (3,3) {\shade};\grid{5}{5}\node at (0,5) {\are[Navy]};\node at (1,5) {\hwire[Navy]};\node at (2,5) {\hwire[Navy]};\node at (3,5) {\hwire[Navy]};\node at (4,5) {\hwire[Navy]};\node at (5,5) {\hwire[Navy]};\node at (0,4) {\vwire[Navy]};\node at (1,4) {\are[Navy]};\node at (2,4) {\hwire[Navy]};\node at (3,4) {\hwire[Navy]};\node at (4,4) {\hwire[Navy]};\node at (5,4) {\hwire[Navy]};\node at (0,3) {\vwire[Navy]};\node at (1,3) {\vwire[Navy]};\node at (2,3) {\nowire};\node at (3,3) {\are[Orange]};\node at (4,3) {\hwire[Orange]};\node at (5,3) {\hwire[Orange]};\node at (0,2) {\vwire[Navy]};\node at (1,2) {\vwire[Navy]};\node at (2,2) {\nowire};\node at (3,2) {\vwire[Orange]};\node at (4,2) {\nowire};\node at (5,2) {\are[Navy]};\node at (0,1) {\vwire[Navy]};\node at (1,1) {\vwire[Navy]};\node at (2,1) {\nowire};\node at (3,1) {\vwire[Orange]};\node at (4,1) {\are[Navy]};\node at (5,1) {\hwire[Navy]};\node at (5,1) {\vwire[Navy]};\node at (0,0) {\vwire[Navy]};\node at (1,0) {\vwire[Navy]};\node at (2,0) {\are[Navy]};\node at (3,0) {\hwire[Navy]};\node at (3,0) {\vwire[Orange]};\node at (4,0) {\hwire[Navy]};\node at (4,0) {\vwire[Navy]};\node at (5,0) {\hwire[Navy]};\node at (5,0) {\vwire[Navy]};\draw[thick] (-0.5,-0.5) -- (-0.5,5.5) -- (5.5,5.5) -- (5.5,-0.5) -- (-0.5,-0.5);\node at (0,-1) {1};\node at (6,5) {1};\node at (1,-1) {2};\node at (6,4) {2};\node at (2,-1) {3};\node at (6,3) {4};\node at (3,-1) {4};\node at (6,2) {6};\node at (4,-1) {5};\node at (6,1) {5};\node at (5,-1) {6};\node at (6,0) {3};\end{tikzpicture}
        \begin{tikzpicture}[scale =.4]\node at (1,4) {\shade};\node at (1,3) {\shade};\node at (2,4) {\shade};\node at (2,3) {\shade};\grid{5}{5}\node at (0,5) {\nowire};\node at (1,5) {\nowire};\node at (2,5) {\are[Navy]};\node at (3,5) {\hwire[Navy]};\node at (4,5) {\hwire[Navy]};\node at (5,5) {\hwire[Navy]};\node at (0,4) {\are[Navy]};\node at (1,4) {\hwire[Navy]};\node at (2,4) {\jay[Navy]};\node at (3,4) {\are[Navy]};\node at (4,4) {\hwire[Navy]};\node at (5,4) {\hwire[Navy]};\node at (0,3) {\vwire[Navy]};\node at (1,3) {\are[Navy]};\node at (2,3) {\hwire[Navy]};\node at (3,3) {\jay[Navy]};\node at (4,3) {\are[Orange]};\node at (5,3) {\hwire[Orange]};\node at (0,2) {\vwire[Navy]};\node at (1,2) {\vwire[Navy]};\node at (2,2) {\nowire};\node at (3,2) {\are[Orange]};\node at (4,2) {\jay[Orange]};\node at (5,2) {\are[Navy]};\node at (0,1) {\vwire[Navy]};\node at (1,1) {\vwire[Navy]};\node at (2,1) {\nowire};\node at (3,1) {\vwire[Orange]};\node at (4,1) {\are[Navy]};\node at (5,1) {\hwire[Navy]};\node at (5,1) {\vwire[Navy]};\node at (0,0) {\vwire[Navy]};\node at (1,0) {\vwire[Navy]};\node at (2,0) {\are[Navy]};\node at (3,0) {\hwire[Navy]};\node at (3,0) {\vwire[Orange]};\node at (4,0) {\hwire[Navy]};\node at (4,0) {\vwire[Navy]};\node at (5,0) {\hwire[Navy]};\node at (5,0) {\vwire[Navy]};\draw[thick] (-0.5,-0.5) -- (-0.5,5.5) -- (5.5,5.5) -- (5.5,-0.5) -- (-0.5,-0.5);\node at (0,-1) {1};\node at (6,5) {1};\node at (1,-1) {2};\node at (6,4) {2};\node at (2,-1) {3};\node at (6,3) {4};\node at (3,-1) {4};\node at (6,2) {6};\node at (4,-1) {5};\node at (6,1) {5};\node at (5,-1) {6};\node at (6,0) {3};\end{tikzpicture}
        \begin{tikzpicture}[scale =.4]\node at (1,4) {\shade};\node at (2,3) {\shade};\grid{5}{5}\node at (0,5) {\vwire[Orange]};\node at (1,5) {\vwire[Orange]};\node at (2,5) {\el[Navy]};\node at (3,5) {\hwire[Navy]};\node at (3,5) {\vwire[Navy]};\node at (4,5) {\hwire[Navy]};\node at (4,5) {\vwire[Navy]};\node at (5,5) {\hwire[Navy]};\node at (5,5) {\vwire[Navy]};\node at (0,4) {\el[Orange]};\node at (1,4) {\hwire[Orange]};\node at (1,4) {\vwire[Orange]};\node at (2,4) {\en[Orange]};\node at (3,4) {\el[Navy]};\node at (4,4) {\hwire[Navy]};\node at (4,4) {\vwire[Navy]};\node at (5,4) {\hwire[Navy]};\node at (5,4) {\vwire[Navy]};\node at (0,3) {\nowire};\node at (1,3) {\el[Orange]};\node at (2,3) {\newire[Orange]};\node at (2,3) {\swwire[Orange]};\node at (3,3) {\en[Orange]};\node at (4,3) {\el[Navy]};\node at (5,3) {\hwire[Navy]};\node at (5,3) {\vwire[Navy]};\node at (0,2) {\nowire};\node at (1,2) {\nowire};\node at (2,2) {\vwire[Orange]};\node at (3,2) {\el[Orange]};\node at (4,2) {\en[Orange]};\node at (5,2) {\el[Navy]};\node at (0,1) {\nowire};\node at (1,1) {\nowire};\node at (2,1) {\vwire[Orange]};\node at (3,1) {\nowire};\node at (4,1) {\el[Orange]};\node at (5,1) {\hwire[Orange]};\node at (0,0) {\nowire};\node at (1,0) {\nowire};\node at (2,0) {\el[Orange]};\node at (3,0) {\hwire[Orange]};\node at (4,0) {\hwire[Orange]};\node at (5,0) {\hwire[Orange]};\draw[thick] (-0.5,-0.5) -- (-0.5,5.5) -- (5.5,5.5) -- (5.5,-0.5) -- (-0.5,-0.5);\node at (0,-1) {1};\node at (6,5) {3};\node at (1,-1) {2};\node at (6,4) {4};\node at (2,-1) {3};\node at (6,3) {5};\node at (3,-1) {4};\node at (6,2) {6};\node at (4,-1) {5};\node at (6,1) {1};\node at (5,-1) {6};\node at (6,0) {2};\end{tikzpicture}
        
        \caption{A $12543$ pattern with a $(2,3)$ blocking pipe.}
        \label{fig:12543Blocks}
\end{figure}

\begin{proposition}\label{p:25143}
    If $w$ contains a pattern $\pi=25143$, then there exists $B\in \BPD{w}$ such that co$(B)$ is non-reduced.
\end{proposition}

\begin{proof}
    Note that the BPD made from Rothe diagram for $25143$ by the droop moves $(1,2) \ssearrow (2,4)$ and $(3,1)\ssearrow (4,3)$ has a non-reduced co-BPD.
    We will assume that $w$ contains blocking pipes as otherwise the result is true by Lemma~\ref{lem:blocking}.
    If there is a blocking pipe of the form $(3,2)$ then the blocking pipe forms a $1423$ pattern with the $2$, $5$, and $4$ pipes form $\pi$, thus by Proposition~\ref{p:1423} we have our result.
    If there is a blocking pipe of the form $(1,2)$ take that pipe as the $2$ in the pattern instead to get the pattern to have one less blocking pipe.
    If there is a blocking pipe of the form $(3,1)$ take that pipe as the $1$ in the pattern instead to get the pattern to have one less blocking pipe.
    Thus, we can assume all the blocking pipes are of the form $(1,3)$.
    
    From such a Rothe BPD we can get a BPD with non-reduced co-BPD by drooping the right-most blocking pipe into where the 2 pattern pipe of the pattern would droop if unblocked, then attempt to droop the 2 pattern pipe to where the elbow for the drooped blocking pipe was.
    If that droop move was blocked repeat this process.
    By dropping where the old elbow was we ensure that we create the chain of elbows required for the configuration, all that remains to check is that the flats witness elbows.
    Note that the 1 pattern pipe will have a flat below where the first blocking pipe drooped into and will have been drooped to create $\inlineare$ below an empty box right of the 2 pattern pipe, left of all the blocking pipes, and weakly above all the blocking pipes.
    Therefore, when drooped the 2 pattern pipe will create a flat left of the $\inlinejay$ made from the droop move and above the right-most $\inlineare$ of the 1 pattern pipe as once all the blocking pipes have been drooped there will be two empty boxes right of the 2 pattern pipe in the row the 2 pattern pipe droops into.
    Further, the flat from the 1 pattern pipe will be right of $\inlineare$ from the 1 pattern pipe and below $\inlineare$ from the first blocking pipe.
    Therefore, the BPD has the configuration.
\end{proof}

\begin{figure}[h!]
        \centering

        \begin{tikzpicture}[scale =.4]\node at (4,3) {\shade};\node at (1,5) {\shade};\grid{5}{5}\node at (0,5) {\nowire};\node at (1,5) {\are[Navy]};\node at (2,5) {\hwire[Navy]};\node at (3,5) {\hwire[Navy]};\node at (4,5) {\hwire[Navy]};\node at (5,5) {\hwire[Navy]};\node at (0,4) {\nowire};\node at (1,4) {\vwire[Navy]};\node at (2,4) {\nowire};\node at (3,4) {\are[Orange]};\node at (4,4) {\hwire[Orange]};\node at (5,4) {\hwire[Orange]};\node at (0,3) {\nowire};\node at (1,3) {\vwire[Navy]};\node at (2,3) {\nowire};\node at (3,3) {\vwire[Orange]};\node at (4,3) {\nowire};\node at (5,3) {\are[Navy]};\node at (0,2) {\nowire};\node at (1,2) {\vwire[Navy]};\node at (2,2) {\are[Navy]};\node at (3,2) {\hwire[Navy]};\node at (3,2) {\vwire[Orange]};\node at (4,2) {\hwire[Navy]};\node at (5,2) {\hwire[Navy]};\node at (5,2) {\vwire[Navy]};\node at (0,1) {\are[Navy]};\node at (1,1) {\hwire[Navy]};\node at (1,1) {\vwire[Navy]};\node at (2,1) {\jay[Navy]};\node at (3,1) {\vwire[Orange]};\node at (4,1) {\are[Navy]};\node at (5,1) {\hwire[Navy]};\node at (5,1) {\vwire[Navy]};\node at (0,0) {\vwire[Navy]};\node at (1,0) {\vwire[Navy]};\node at (2,0) {\are[Navy]};\node at (3,0) {\hwire[Navy]};\node at (3,0) {\vwire[Orange]};\node at (4,0) {\hwire[Navy]};\node at (4,0) {\vwire[Navy]};\node at (5,0) {\hwire[Navy]};\node at (5,0) {\vwire[Navy]};\draw[thick] (-0.5,-0.5) -- (-0.5,5.5) -- (5.5,5.5) -- (5.5,-0.5) -- (-0.5,-0.5);\node at (0,-1) {1};\node at (6,5) {2};\node at (1,-1) {2};\node at (6,4) {4};\node at (2,-1) {3};\node at (6,3) {6};\node at (3,-1) {4};\node at (6,2) {1};\node at (4,-1) {5};\node at (6,1) {5};\node at (5,-1) {6};\node at (6,0) {3};\end{tikzpicture}
        \begin{tikzpicture}[scale =.4]\node at (2,4) {\shade};\node at (2,2) {\shade};\node at (4,2) {\shade};\node at (4,3) {\shade};\node at (3,4) {\shade};\node at (4,4) {\shade};\grid{5}{5}\node at (0,5) {\nowire};\node at (1,5) {\nowire};\node at (2,5) {\nowire};\node at (3,5) {\are[Navy]};\node at (4,5) {\hwire[Navy]};\node at (5,5) {\hwire[Navy]};\node at (0,4) {\nowire};\node at (1,4) {\are[Navy]};\node at (2,4) {\hwire[Navy]};\node at (3,4) {\jay[Navy]};\node at (4,4) {\are[Orange]};\node at (5,4) {\hwire[Orange]};\node at (0,3) {\nowire};\node at (1,3) {\vwire[Navy]};\node at (2,3) {\nowire};\node at (3,3) {\are[Orange]};\node at (4,3) {\jay[Orange]};\node at (5,3) {\are[Navy]};\node at (0,2) {\nowire};\node at (1,2) {\vwire[Navy]};\node at (2,2) {\are[Navy]};\node at (3,2) {\hwire[Navy]};\node at (3,2) {\vwire[Orange]};\node at (4,2) {\hwire[Navy]};\node at (5,2) {\hwire[Navy]};\node at (5,2) {\vwire[Navy]};\node at (0,1) {\are[Navy]};\node at (1,1) {\hwire[Navy]};\node at (1,1) {\vwire[Navy]};\node at (2,1) {\jay[Navy]};\node at (3,1) {\vwire[Orange]};\node at (4,1) {\are[Navy]};\node at (5,1) {\hwire[Navy]};\node at (5,1) {\vwire[Navy]};\node at (0,0) {\vwire[Navy]};\node at (1,0) {\vwire[Navy]};\node at (2,0) {\are[Navy]};\node at (3,0) {\hwire[Navy]};\node at (3,0) {\vwire[Orange]};\node at (4,0) {\hwire[Navy]};\node at (4,0) {\vwire[Navy]};\node at (5,0) {\hwire[Navy]};\node at (5,0) {\vwire[Navy]};\draw[thick] (-0.5,-0.5) -- (-0.5,5.5) -- (5.5,5.5) -- (5.5,-0.5) -- (-0.5,-0.5);\node at (0,-1) {1};\node at (6,5) {2};\node at (1,-1) {2};\node at (6,4) {4};\node at (2,-1) {3};\node at (6,3) {6};\node at (3,-1) {4};\node at (6,2) {1};\node at (4,-1) {5};\node at (6,1) {5};\node at (5,-1) {6};\node at (6,0) {3};\end{tikzpicture}
        \begin{tikzpicture}[scale =.4]\node at (2,4) {\shade};\node at (4,2) {\shade};\grid{5}{5}\node at (0,5) {\vwire[Navy]};\node at (1,5) {\vwire[Orange]};\node at (2,5) {\vwire[Orange]};\node at (3,5) {\el[Navy]};\node at (4,5) {\hwire[Navy]};\node at (4,5) {\vwire[Navy]};\node at (5,5) {\hwire[Navy]};\node at (5,5) {\vwire[Navy]};\node at (0,4) {\vwire[Navy]};\node at (1,4) {\el[Orange]};\node at (2,4) {\hwire[Orange]};\node at (2,4) {\vwire[Orange]};\node at (3,4) {\en[Orange]};\node at (4,4) {\el[Navy]};\node at (5,4) {\hwire[Navy]};\node at (5,4) {\vwire[Navy]};\node at (0,3) {\vwire[Navy]};\node at (1,3) {\nowire};\node at (2,3) {\vwire[Orange]};\node at (3,3) {\el[Orange]};\node at (4,3) {\en[Orange]};\node at (5,3) {\el[Navy]};\node at (0,2) {\vwire[Navy]};\node at (1,2) {\nowire};\node at (2,2) {\el[Orange]};\node at (3,2) {\hwire[Orange]};\node at (4,2) {\newire[Orange]};\node at (4,2) {\swwire[Orange]};\node at (5,2) {\hwire[Orange]};\node at (0,1) {\el[Navy]};\node at (1,1) {\hwire[Navy]};\node at (2,1) {\en[Navy]};\node at (3,1) {\nowire};\node at (4,1) {\el[Orange]};\node at (5,1) {\hwire[Orange]};\node at (0,0) {\nowire};\node at (1,0) {\nowire};\node at (2,0) {\el[Navy]};\node at (3,0) {\hwire[Navy]};\node at (4,0) {\hwire[Navy]};\node at (5,0) {\hwire[Navy]};\draw[thick] (-0.5,-0.5) -- (-0.5,5.5) -- (5.5,5.5) -- (5.5,-0.5) -- (-0.5,-0.5);\node at (0,-1) {1};\node at (6,5) {4};\node at (1,-1) {2};\node at (6,4) {5};\node at (2,-1) {3};\node at (6,3) {6};\node at (3,-1) {4};\node at (6,2) {2};\node at (4,-1) {5};\node at (6,1) {3};\node at (5,-1) {6};\node at (6,0) {1};\end{tikzpicture}
        
        \caption{A $25143$ pattern with a $(1,3)$ blocking pipe.}
        \label{fig:12543Blocks}
\end{figure}
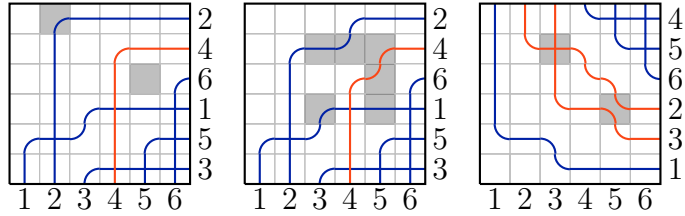


\begin{proposition}\label{p:241653}
    If $w$ contains a pattern $\pi=241653$, then there exists $B\in \BPD{w}$ such that co$(B)$ is non-reduced.
\end{proposition}

\begin{proof}
    Note that the BPD made from Rothe diagram for $25143$ by the droop moves $(2,4) \ssearrow (4,5)$, $(1,2)\ssearrow (2,4)$, and $(3,1)\ssearrow (4,3)$ has a non-reduced co-BPD.
    We will assume that $w$ contains blocking pipes as otherwise the result is true by Lemma~\ref{lem:blocking}.
    If there is a blocking pipe of the form $(3,2)$ then the blocking pipe forms a $1423$ pattern with the $2$, $4$, and $3$ pipes form $\pi$, thus by Proposition~\ref{p:1423} we have our result.
    If there is a blocking pipe of the form $(1,2)$ take that pipe as the $2$ in the pattern instead to get the pattern to have one less blocking pipe.
    If there is a blocking pipe of the form $(3,1)$ take that pipe as the $1$ in the pattern instead to get the pattern to have one less blocking pipe.
    If there is a blocking pipe of the form $(2,4)$ take the blocking pipe to instead be the $4$ pipe in the pattern and make the old $4$ pipe a $(1,3)$ blocking pipe instead.
    Thus, we can assume all the blocking pipes are of the form $(1,3)$ or $(3,4)$.
    
    From such a Rothe BPD we can get a BPD with non-reduced co-BPD by drooping the right-most $(3,4)$ blocking pipe into where the 4 pattern pipe of the pattern would droop if unblocked, then attempt to droop the 4 pattern pipe to where the elbow for the drooped $(3,4)$ blocking pipe was.
    If that droop move was blocked repeat this process.
    Then we can droop the right-most $(1,3)$ blocking pipe into where the elbow of the $4$ pipe was and then attempt to droop the $2$ pattern pipe into where the elbow of the $(1,3)$ blocking pipe was.
    If that droop move was blocked repeat this process.
    Finally, the $1$ pattern pipe can be drooped as normal.
    Note that by dropping where the old elbow was we ensure that we create the chain of elbows required for the configuration, all that remains to check is that the flats witness elbows.
    The 1 pattern pipe will have a flat below where the first $(1,3)$ blocking pipe drooped into and will have been drooped to create $\inlineare$ below an empty box right of the 2 pattern pipe, left of all the blocking pipes, and weakly above all the blocking pipes.
    Therefore, when drooped the 2 pattern pipe will create a flat left of the $\inlinejay$ made from the droop move and above the right-most $\inlineare$ of the 1 pattern pipe as once all the blocking pipes have been drooped there will be two empty boxes right of the 2 pattern pipe in the row the 2 pattern pipe droops into.
    Further, the flat from the 1 pattern pipe will be right of $\inlineare$ from the 1 pattern pipe and below $\inlineare$ from the first $(1,3)$ blocking pipe.
    Therefore, the BPD has the configuration.
\end{proof}

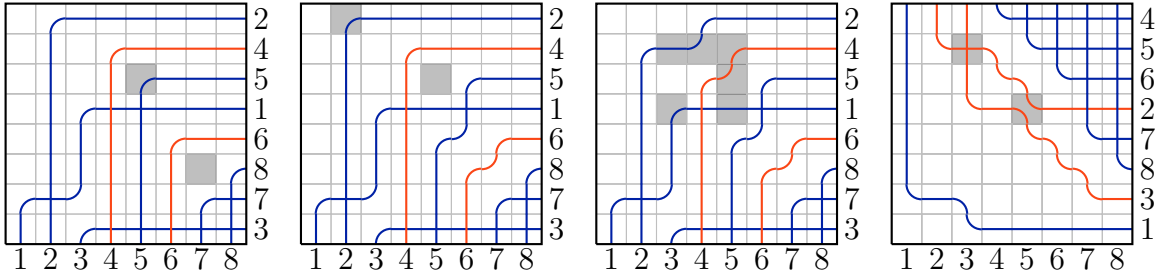
\begin{figure}[h!]
        \centering

        \begin{tikzpicture}[scale =.4]\node at (4,5) {\shade};\node at (6,2) {\shade};\grid{7}{7}\node at (0,7) {\nowire};\node at (1,7) {\are[Navy]};\node at (2,7) {\hwire[Navy]};\node at (3,7) {\hwire[Navy]};\node at (4,7) {\hwire[Navy]};\node at (5,7) {\hwire[Navy]};\node at (6,7) {\hwire[Navy]};\node at (7,7) {\hwire[Navy]};\node at (0,6) {\nowire};\node at (1,6) {\vwire[Navy]};\node at (2,6) {\nowire};\node at (3,6) {\are[Orange]};\node at (4,6) {\hwire[Orange]};\node at (5,6) {\hwire[Orange]};\node at (6,6) {\hwire[Orange]};\node at (7,6) {\hwire[Orange]};\node at (0,5) {\nowire};\node at (1,5) {\vwire[Navy]};\node at (2,5) {\nowire};\node at (3,5) {\vwire[Orange]};\node at (4,5) {\are[Navy]};\node at (5,5) {\hwire[Navy]};\node at (6,5) {\hwire[Navy]};\node at (7,5) {\hwire[Navy]};\node at (0,4) {\nowire};\node at (1,4) {\vwire[Navy]};\node at (2,4) {\are[Navy]};\node at (3,4) {\hwire[Navy]};\node at (3,4) {\vwire[Orange]};\node at (4,4) {\hwire[Navy]};\node at (4,4) {\vwire[Navy]};\node at (5,4) {\hwire[Navy]};\node at (6,4) {\hwire[Navy]};\node at (7,4) {\hwire[Navy]};\node at (0,3) {\nowire};\node at (1,3) {\vwire[Navy]};\node at (2,3) {\vwire[Navy]};\node at (3,3) {\vwire[Orange]};\node at (4,3) {\vwire[Navy]};\node at (5,3) {\are[Orange]};\node at (6,3) {\hwire[Orange]};\node at (7,3) {\hwire[Orange]};\node at (0,2) {\nowire};\node at (1,2) {\vwire[Navy]};\node at (2,2) {\vwire[Navy]};\node at (3,2) {\vwire[Orange]};\node at (4,2) {\vwire[Navy]};\node at (5,2) {\vwire[Orange]};\node at (6,2) {\nowire};\node at (7,2) {\are[Navy]};\node at (0,1) {\are[Navy]};\node at (1,1) {\hwire[Navy]};\node at (1,1) {\vwire[Navy]};\node at (2,1) {\jay[Navy]};\node at (3,1) {\vwire[Orange]};\node at (4,1) {\vwire[Navy]};\node at (5,1) {\vwire[Orange]};\node at (6,1) {\are[Navy]};\node at (7,1) {\hwire[Navy]};\node at (7,1) {\vwire[Navy]};\node at (0,0) {\vwire[Navy]};\node at (1,0) {\vwire[Navy]};\node at (2,0) {\are[Navy]};\node at (3,0) {\hwire[Navy]};\node at (3,0) {\vwire[Orange]};\node at (4,0) {\hwire[Navy]};\node at (4,0) {\vwire[Navy]};\node at (5,0) {\hwire[Navy]};\node at (5,0) {\vwire[Orange]};\node at (6,0) {\hwire[Navy]};\node at (6,0) {\vwire[Navy]};\node at (7,0) {\hwire[Navy]};\node at (7,0) {\vwire[Navy]};\draw[thick] (-0.5,-0.5) -- (-0.5,7.5) -- (7.5,7.5) -- (7.5,-0.5) -- (-0.5,-0.5);\node at (0,-1) {1};\node at (8,7) {2};\node at (1,-1) {2};\node at (8,6) {4};\node at (2,-1) {3};\node at (8,5) {5};\node at (3,-1) {4};\node at (8,4) {1};\node at (4,-1) {5};\node at (8,3) {6};\node at (5,-1) {6};\node at (8,2) {8};\node at (6,-1) {7};\node at (8,1) {7};\node at (7,-1) {8};\node at (8,0) {3};\end{tikzpicture}
        \begin{tikzpicture}[scale =.4]\node at (1,7) {\shade};\node at (4,5) {\shade};\grid{7}{7}\node at (0,7) {\nowire};\node at (1,7) {\are[Navy]};\node at (2,7) {\hwire[Navy]};\node at (3,7) {\hwire[Navy]};\node at (4,7) {\hwire[Navy]};\node at (5,7) {\hwire[Navy]};\node at (6,7) {\hwire[Navy]};\node at (7,7) {\hwire[Navy]};\node at (0,6) {\nowire};\node at (1,6) {\vwire[Navy]};\node at (2,6) {\nowire};\node at (3,6) {\are[Orange]};\node at (4,6) {\hwire[Orange]};\node at (5,6) {\hwire[Orange]};\node at (6,6) {\hwire[Orange]};\node at (7,6) {\hwire[Orange]};\node at (0,5) {\nowire};\node at (1,5) {\vwire[Navy]};\node at (2,5) {\nowire};\node at (3,5) {\vwire[Orange]};\node at (4,5) {\nowire};\node at (5,5) {\are[Navy]};\node at (6,5) {\hwire[Navy]};\node at (7,5) {\hwire[Navy]};\node at (0,4) {\nowire};\node at (1,4) {\vwire[Navy]};\node at (2,4) {\are[Navy]};\node at (3,4) {\hwire[Navy]};\node at (3,4) {\vwire[Orange]};\node at (4,4) {\hwire[Navy]};\node at (5,4) {\hwire[Navy]};\node at (5,4) {\vwire[Navy]};\node at (6,4) {\hwire[Navy]};\node at (7,4) {\hwire[Navy]};\node at (0,3) {\nowire};\node at (1,3) {\vwire[Navy]};\node at (2,3) {\vwire[Navy]};\node at (3,3) {\vwire[Orange]};\node at (4,3) {\are[Navy]};\node at (5,3) {\jay[Navy]};\node at (6,3) {\are[Orange]};\node at (7,3) {\hwire[Orange]};\node at (0,2) {\nowire};\node at (1,2) {\vwire[Navy]};\node at (2,2) {\vwire[Navy]};\node at (3,2) {\vwire[Orange]};\node at (4,2) {\vwire[Navy]};\node at (5,2) {\are[Orange]};\node at (6,2) {\jay[Orange]};\node at (7,2) {\are[Navy]};\node at (0,1) {\are[Navy]};\node at (1,1) {\hwire[Navy]};\node at (1,1) {\vwire[Navy]};\node at (2,1) {\jay[Navy]};\node at (3,1) {\vwire[Orange]};\node at (4,1) {\vwire[Navy]};\node at (5,1) {\vwire[Orange]};\node at (6,1) {\are[Navy]};\node at (7,1) {\hwire[Navy]};\node at (7,1) {\vwire[Navy]};\node at (0,0) {\vwire[Navy]};\node at (1,0) {\vwire[Navy]};\node at (2,0) {\are[Navy]};\node at (3,0) {\hwire[Navy]};\node at (3,0) {\vwire[Orange]};\node at (4,0) {\hwire[Navy]};\node at (4,0) {\vwire[Navy]};\node at (5,0) {\hwire[Navy]};\node at (5,0) {\vwire[Orange]};\node at (6,0) {\hwire[Navy]};\node at (6,0) {\vwire[Navy]};\node at (7,0) {\hwire[Navy]};\node at (7,0) {\vwire[Navy]};\draw[thick] (-0.5,-0.5) -- (-0.5,7.5) -- (7.5,7.5) -- (7.5,-0.5) -- (-0.5,-0.5);\node at (0,-1) {1};\node at (8,7) {2};\node at (1,-1) {2};\node at (8,6) {4};\node at (2,-1) {3};\node at (8,5) {5};\node at (3,-1) {4};\node at (8,4) {1};\node at (4,-1) {5};\node at (8,3) {6};\node at (5,-1) {6};\node at (8,2) {8};\node at (6,-1) {7};\node at (8,1) {7};\node at (7,-1) {8};\node at (8,0) {3};\end{tikzpicture}
        \begin{tikzpicture}[scale =.4]\node at (2,6) {\shade};\node at (2,4) {\shade};\node at (4,4) {\shade};\node at (4,5) {\shade};\node at (3,6) {\shade};\node at (4,6) {\shade};\grid{7}{7}\node at (0,7) {\nowire};\node at (1,7) {\nowire};\node at (2,7) {\nowire};\node at (3,7) {\are[Navy]};\node at (4,7) {\hwire[Navy]};\node at (5,7) {\hwire[Navy]};\node at (6,7) {\hwire[Navy]};\node at (7,7) {\hwire[Navy]};\node at (0,6) {\nowire};\node at (1,6) {\are[Navy]};\node at (2,6) {\hwire[Navy]};\node at (3,6) {\jay[Navy]};\node at (4,6) {\are[Orange]};\node at (5,6) {\hwire[Orange]};\node at (6,6) {\hwire[Orange]};\node at (7,6) {\hwire[Orange]};\node at (0,5) {\nowire};\node at (1,5) {\vwire[Navy]};\node at (2,5) {\nowire};\node at (3,5) {\are[Orange]};\node at (4,5) {\jay[Orange]};\node at (5,5) {\are[Navy]};\node at (6,5) {\hwire[Navy]};\node at (7,5) {\hwire[Navy]};\node at (0,4) {\nowire};\node at (1,4) {\vwire[Navy]};\node at (2,4) {\are[Navy]};\node at (3,4) {\hwire[Navy]};\node at (3,4) {\vwire[Orange]};\node at (4,4) {\hwire[Navy]};\node at (5,4) {\hwire[Navy]};\node at (5,4) {\vwire[Navy]};\node at (6,4) {\hwire[Navy]};\node at (7,4) {\hwire[Navy]};\node at (0,3) {\nowire};\node at (1,3) {\vwire[Navy]};\node at (2,3) {\vwire[Navy]};\node at (3,3) {\vwire[Orange]};\node at (4,3) {\are[Navy]};\node at (5,3) {\jay[Navy]};\node at (6,3) {\are[Orange]};\node at (7,3) {\hwire[Orange]};\node at (0,2) {\nowire};\node at (1,2) {\vwire[Navy]};\node at (2,2) {\vwire[Navy]};\node at (3,2) {\vwire[Orange]};\node at (4,2) {\vwire[Navy]};\node at (5,2) {\are[Orange]};\node at (6,2) {\jay[Orange]};\node at (7,2) {\are[Navy]};\node at (0,1) {\are[Navy]};\node at (1,1) {\hwire[Navy]};\node at (1,1) {\vwire[Navy]};\node at (2,1) {\jay[Navy]};\node at (3,1) {\vwire[Orange]};\node at (4,1) {\vwire[Navy]};\node at (5,1) {\vwire[Orange]};\node at (6,1) {\are[Navy]};\node at (7,1) {\hwire[Navy]};\node at (7,1) {\vwire[Navy]};\node at (0,0) {\vwire[Navy]};\node at (1,0) {\vwire[Navy]};\node at (2,0) {\are[Navy]};\node at (3,0) {\hwire[Navy]};\node at (3,0) {\vwire[Orange]};\node at (4,0) {\hwire[Navy]};\node at (4,0) {\vwire[Navy]};\node at (5,0) {\hwire[Navy]};\node at (5,0) {\vwire[Orange]};\node at (6,0) {\hwire[Navy]};\node at (6,0) {\vwire[Navy]};\node at (7,0) {\hwire[Navy]};\node at (7,0) {\vwire[Navy]};\draw[thick] (-0.5,-0.5) -- (-0.5,7.5) -- (7.5,7.5) -- (7.5,-0.5) -- (-0.5,-0.5);\node at (0,-1) {1};\node at (8,7) {2};\node at (1,-1) {2};\node at (8,6) {4};\node at (2,-1) {3};\node at (8,5) {5};\node at (3,-1) {4};\node at (8,4) {1};\node at (4,-1) {5};\node at (8,3) {6};\node at (5,-1) {6};\node at (8,2) {8};\node at (6,-1) {7};\node at (8,1) {7};\node at (7,-1) {8};\node at (8,0) {3};\end{tikzpicture}
        \begin{tikzpicture}[scale =.4]\node at (2,6) {\shade};\node at (4,4) {\shade};\grid{7}{7}\node at (0,7) {\vwire[Navy]};\node at (1,7) {\vwire[Orange]};\node at (2,7) {\vwire[Orange]};\node at (3,7) {\el[Navy]};\node at (4,7) {\hwire[Navy]};\node at (4,7) {\vwire[Navy]};\node at (5,7) {\hwire[Navy]};\node at (5,7) {\vwire[Navy]};\node at (6,7) {\hwire[Navy]};\node at (6,7) {\vwire[Navy]};\node at (7,7) {\hwire[Navy]};\node at (7,7) {\vwire[Navy]};\node at (0,6) {\vwire[Navy]};\node at (1,6) {\el[Orange]};\node at (2,6) {\hwire[Orange]};\node at (2,6) {\vwire[Orange]};\node at (3,6) {\en[Orange]};\node at (4,6) {\el[Navy]};\node at (5,6) {\hwire[Navy]};\node at (5,6) {\vwire[Navy]};\node at (6,6) {\hwire[Navy]};\node at (6,6) {\vwire[Navy]};\node at (7,6) {\hwire[Navy]};\node at (7,6) {\vwire[Navy]};\node at (0,5) {\vwire[Navy]};\node at (1,5) {\nowire};\node at (2,5) {\vwire[Orange]};\node at (3,5) {\el[Orange]};\node at (4,5) {\en[Orange]};\node at (5,5) {\el[Navy]};\node at (6,5) {\hwire[Navy]};\node at (6,5) {\vwire[Navy]};\node at (7,5) {\hwire[Navy]};\node at (7,5) {\vwire[Navy]};\node at (0,4) {\vwire[Navy]};\node at (1,4) {\nowire};\node at (2,4) {\el[Orange]};\node at (3,4) {\hwire[Orange]};\node at (4,4) {\newire[Orange]};\node at (4,4) {\swwire[Orange]};\node at (5,4) {\hwire[Orange]};\node at (6,4) {\hwire[Orange]};\node at (6,4) {\vwire[Navy]};\node at (7,4) {\hwire[Orange]};\node at (7,4) {\vwire[Navy]};\node at (0,3) {\vwire[Navy]};\node at (1,3) {\nowire};\node at (2,3) {\nowire};\node at (3,3) {\nowire};\node at (4,3) {\el[Orange]};\node at (5,3) {\en[Orange]};\node at (6,3) {\el[Navy]};\node at (7,3) {\hwire[Navy]};\node at (7,3) {\vwire[Navy]};\node at (0,2) {\vwire[Navy]};\node at (1,2) {\nowire};\node at (2,2) {\nowire};\node at (3,2) {\nowire};\node at (4,2) {\nowire};\node at (5,2) {\el[Orange]};\node at (6,2) {\en[Orange]};\node at (7,2) {\el[Navy]};\node at (0,1) {\el[Navy]};\node at (1,1) {\hwire[Navy]};\node at (2,1) {\en[Navy]};\node at (3,1) {\nowire};\node at (4,1) {\nowire};\node at (5,1) {\nowire};\node at (6,1) {\el[Orange]};\node at (7,1) {\hwire[Orange]};\node at (0,0) {\nowire};\node at (1,0) {\nowire};\node at (2,0) {\el[Navy]};\node at (3,0) {\hwire[Navy]};\node at (4,0) {\hwire[Navy]};\node at (5,0) {\hwire[Navy]};\node at (6,0) {\hwire[Navy]};\node at (7,0) {\hwire[Navy]};\draw[thick] (-0.5,-0.5) -- (-0.5,7.5) -- (7.5,7.5) -- (7.5,-0.5) -- (-0.5,-0.5);\node at (0,-1) {1};\node at (8,7) {4};\node at (1,-1) {2};\node at (8,6) {5};\node at (2,-1) {3};\node at (8,5) {6};\node at (3,-1) {4};\node at (8,4) {2};\node at (4,-1) {5};\node at (8,3) {7};\node at (5,-1) {6};\node at (8,2) {8};\node at (6,-1) {7};\node at (8,1) {3};\node at (7,-1) {8};\node at (8,0) {1};\end{tikzpicture}
        
        \caption{A $241653$ pattern with $(1,3)$ and $(3,4)$ blocking pipes.}
        \label{fig:241653Blocks}
\end{figure}

\begin{proposition}\label{p:216543}
    If $w$ contains a pattern $\pi=216543$, then there exists $B\in \BPD{w}$ such that co$(B)$ is non-reduced.
\end{proposition}

\begin{proof}
    Note that the BPD made from Rothe diagram for $216543$ by the droop moves $(1,2) \ssearrow (4,4)$ and $(2,1)\ssearrow (3,3)$ has a non-reduced co-BPD.
    We will assume that $w$ contains blocking pipes as otherwise the result is true by Lemma~\ref{lem:blocking}.
    If there is a blocking pipe of the form $(1,3)$ then the blocking pipe forms a $241653$ pattern with the $2$, $1$, $6$, $5$, and $3$ pipes form $\pi$, thus by Proposition~\ref{p:241653} we have our result.
    If there is a blocking pipe of the form $(3,2)$ or $(3,3)$ then the blocking pipe forms a $1423$ pattern with the $1$, $6$, and $4$ pipes form $\pi$, thus by Proposition~\ref{p:1423} we have our result.
    If there is a blocking pipe of the form $(1,2)$ take that pipe as the $2$ in the pattern instead to get the pattern to have one less blocking pipe.
    If there is a blocking pipe of the form $(2,1)$ take that pipe as the $1$ in the pattern instead to get the pattern to have one less blocking pipe.
    Thus, we can assume all the blocking pipes are of the form $(2,2)$ or $(2,3)$.
    
    Suppose two $(2,2)$ blocking pipes appear in $w$ in increasing order, then $w$ has $12543$ pattern by taking the two blocking pipes and the $6$, $4$, and $3$ pipes from $\pi$.
    Suppose two $(2,2)$ blocking pipes appear in $w$ in decreasing order, then $w$ has $216543$ pattern with two less blocking pipes by replacing the $21$ of $\pi$ with the two $(2,2)$ blocking pipes.
    Suppose two $(2,3)$ blocking pipes appear in $w$ in decreasing order, then $w$ has a $13254$ pattern by taking the $2$ pipe from $\pi$, the two $(2,3)$ blocking pipes, and the $6$ and $4$ pipes from $\pi$.
    Suppose a $(2,3)$ blocking pipe appears before a $(2,2)$ blocking pipe in $w$, then $w$ has a $1423$ pattern by taking the $1$ from the $2$ pipe from $\pi$, the two blocking pipes, and the $3$ pipe from $\pi$.
    Therefore, we must have the blocking pipes appear in increasing order in $w$ and have at most one $(2,2)$ blocking pipe.
    First droop the right-most $(2,3)$ blocking pipe into where the $2$ pipe wants to droop.
    Then if there any other $(2,3)$ blocking pipes droop them into the bottom-most box between it and the last drooped pipe.
    This process will eventually lead to a two by two group of empty boxes (maybe broken up by non-blocking pipes) as each droop move moves a vertical group of two empty boxes towards the column of empty boxes above the $3$ pattern pipe.
    If there is a $(2,2)$ blocking pipe droop into the bottom right corner of this group.
    Then the $2$ pattern pipe can droop into the bottom right corner and the $1$ pattern pipe can droop into where the top left corner was to form the configuration.
\end{proof}

\begin{figure}[h!]
        \centering

        \begin{tikzpicture}[scale =.4]\node at (1,7) {\shade};\node at (5,2) {\shade};\grid{7}{7}\node at (0,7) {\nowire};\node at (1,7) {\are[Navy]};\node at (2,7) {\hwire[Navy]};\node at (3,7) {\hwire[Navy]};\node at (4,7) {\hwire[Navy]};\node at (5,7) {\hwire[Navy]};\node at (6,7) {\hwire[Navy]};\node at (7,7) {\hwire[Navy]};\node at (0,6) {\are[Navy]};\node at (1,6) {\hwire[Navy]};\node at (1,6) {\vwire[Navy]};\node at (2,6) {\hwire[Navy]};\node at (3,6) {\hwire[Navy]};\node at (4,6) {\hwire[Navy]};\node at (5,6) {\hwire[Navy]};\node at (6,6) {\hwire[Navy]};\node at (7,6) {\hwire[Navy]};\node at (0,5) {\vwire[Navy]};\node at (1,5) {\vwire[Navy]};\node at (2,5) {\nowire};\node at (3,5) {\nowire};\node at (4,5) {\are[Orange]};\node at (5,5) {\hwire[Orange]};\node at (6,5) {\hwire[Orange]};\node at (7,5) {\hwire[Orange]};\node at (0,4) {\vwire[Navy]};\node at (1,4) {\vwire[Navy]};\node at (2,4) {\are[Orange]};\node at (3,4) {\hwire[Orange]};\node at (4,4) {\hwire[Orange]};\node at (4,4) {\vwire[Orange]};\node at (5,4) {\hwire[Orange]};\node at (6,4) {\hwire[Orange]};\node at (7,4) {\hwire[Orange]};\node at (0,3) {\vwire[Navy]};\node at (1,3) {\vwire[Navy]};\node at (2,3) {\vwire[Orange]};\node at (3,3) {\nowire};\node at (4,3) {\vwire[Orange]};\node at (5,3) {\nowire};\node at (6,3) {\nowire};\node at (7,3) {\are[Navy]};\node at (0,2) {\vwire[Navy]};\node at (1,2) {\vwire[Navy]};\node at (2,2) {\vwire[Orange]};\node at (3,2) {\nowire};\node at (4,2) {\vwire[Orange]};\node at (5,2) {\nowire};\node at (6,2) {\are[Navy]};\node at (7,2) {\hwire[Navy]};\node at (7,2) {\vwire[Navy]};\node at (0,1) {\vwire[Navy]};\node at (1,1) {\vwire[Navy]};\node at (2,1) {\vwire[Orange]};\node at (3,1) {\nowire};\node at (4,1) {\vwire[Orange]};\node at (5,1) {\are[Navy]};\node at (6,1) {\hwire[Navy]};\node at (6,1) {\vwire[Navy]};\node at (7,1) {\hwire[Navy]};\node at (7,1) {\vwire[Navy]};\node at (0,0) {\vwire[Navy]};\node at (1,0) {\vwire[Navy]};\node at (2,0) {\vwire[Orange]};\node at (3,0) {\are[Navy]};\node at (4,0) {\hwire[Navy]};\node at (4,0) {\vwire[Orange]};\node at (5,0) {\hwire[Navy]};\node at (5,0) {\vwire[Navy]};\node at (6,0) {\hwire[Navy]};\node at (6,0) {\vwire[Navy]};\node at (7,0) {\hwire[Navy]};\node at (7,0) {\vwire[Navy]};\draw[thick] (-0.5,-0.5) -- (-0.5,7.5) -- (7.5,7.5) -- (7.5,-0.5) -- (-0.5,-0.5);\node at (0,-1) {1};\node at (8,7) {2};\node at (1,-1) {2};\node at (8,6) {1};\node at (2,-1) {3};\node at (8,5) {5};\node at (3,-1) {4};\node at (8,4) {3};\node at (4,-1) {5};\node at (8,3) {8};\node at (5,-1) {6};\node at (8,2) {7};\node at (6,-1) {7};\node at (8,1) {6};\node at (7,-1) {8};\node at (8,0) {4};\end{tikzpicture}
        \begin{tikzpicture}[scale =.4]\node at (1,7) {\shade};\node at (3,4) {\shade};\grid{7}{7}\node at (0,7) {\nowire};\node at (1,7) {\are[Navy]};\node at (2,7) {\hwire[Navy]};\node at (3,7) {\hwire[Navy]};\node at (4,7) {\hwire[Navy]};\node at (5,7) {\hwire[Navy]};\node at (6,7) {\hwire[Navy]};\node at (7,7) {\hwire[Navy]};\node at (0,6) {\are[Navy]};\node at (1,6) {\hwire[Navy]};\node at (1,6) {\vwire[Navy]};\node at (2,6) {\hwire[Navy]};\node at (3,6) {\hwire[Navy]};\node at (4,6) {\hwire[Navy]};\node at (5,6) {\hwire[Navy]};\node at (6,6) {\hwire[Navy]};\node at (7,6) {\hwire[Navy]};\node at (0,5) {\vwire[Navy]};\node at (1,5) {\vwire[Navy]};\node at (2,5) {\nowire};\node at (3,5) {\nowire};\node at (4,5) {\nowire};\node at (5,5) {\are[Orange]};\node at (6,5) {\hwire[Orange]};\node at (7,5) {\hwire[Orange]};\node at (0,4) {\vwire[Navy]};\node at (1,4) {\vwire[Navy]};\node at (2,4) {\nowire};\node at (3,4) {\nowire};\node at (4,4) {\are[Orange]};\node at (5,4) {\hwire[Orange]};\node at (5,4) {\vwire[Orange]};\node at (6,4) {\hwire[Orange]};\node at (7,4) {\hwire[Orange]};\node at (0,3) {\vwire[Navy]};\node at (1,3) {\vwire[Navy]};\node at (2,3) {\are[Orange]};\node at (3,3) {\hwire[Orange]};\node at (4,3) {\jay[Orange]};\node at (5,3) {\vwire[Orange]};\node at (6,3) {\nowire};\node at (7,3) {\are[Navy]};\node at (0,2) {\vwire[Navy]};\node at (1,2) {\vwire[Navy]};\node at (2,2) {\vwire[Orange]};\node at (3,2) {\nowire};\node at (4,2) {\are[Orange]};\node at (5,2) {\jay[Orange]};\node at (6,2) {\are[Navy]};\node at (7,2) {\hwire[Navy]};\node at (7,2) {\vwire[Navy]};\node at (0,1) {\vwire[Navy]};\node at (1,1) {\vwire[Navy]};\node at (2,1) {\vwire[Orange]};\node at (3,1) {\nowire};\node at (4,1) {\vwire[Orange]};\node at (5,1) {\are[Navy]};\node at (6,1) {\hwire[Navy]};\node at (6,1) {\vwire[Navy]};\node at (7,1) {\hwire[Navy]};\node at (7,1) {\vwire[Navy]};\node at (0,0) {\vwire[Navy]};\node at (1,0) {\vwire[Navy]};\node at (2,0) {\vwire[Orange]};\node at (3,0) {\are[Navy]};\node at (4,0) {\hwire[Navy]};\node at (4,0) {\vwire[Orange]};\node at (5,0) {\hwire[Navy]};\node at (5,0) {\vwire[Navy]};\node at (6,0) {\hwire[Navy]};\node at (6,0) {\vwire[Navy]};\node at (7,0) {\hwire[Navy]};\node at (7,0) {\vwire[Navy]};\draw[thick] (-0.5,-0.5) -- (-0.5,7.5) -- (7.5,7.5) -- (7.5,-0.5) -- (-0.5,-0.5);\node at (0,-1) {1};\node at (8,7) {2};\node at (1,-1) {2};\node at (8,6) {1};\node at (2,-1) {3};\node at (8,5) {5};\node at (3,-1) {4};\node at (8,4) {3};\node at (4,-1) {5};\node at (8,3) {8};\node at (5,-1) {6};\node at (8,2) {7};\node at (6,-1) {7};\node at (8,1) {6};\node at (7,-1) {8};\node at (8,0) {4};\end{tikzpicture}
        \begin{tikzpicture}[scale =.4]\node at (1,5) {\shade};\node at (2,5) {\shade};\node at (2,4) {\shade};\node at (1,4) {\shade};\grid{7}{7}\node at (0,7) {\nowire};\node at (1,7) {\nowire};\node at (2,7) {\nowire};\node at (3,7) {\are[Navy]};\node at (4,7) {\hwire[Navy]};\node at (5,7) {\hwire[Navy]};\node at (6,7) {\hwire[Navy]};\node at (7,7) {\hwire[Navy]};\node at (0,6) {\nowire};\node at (1,6) {\nowire};\node at (2,6) {\are[Navy]};\node at (3,6) {\hwire[Navy]};\node at (3,6) {\vwire[Navy]};\node at (4,6) {\hwire[Navy]};\node at (5,6) {\hwire[Navy]};\node at (6,6) {\hwire[Navy]};\node at (7,6) {\hwire[Navy]};\node at (0,5) {\are[Navy]};\node at (1,5) {\hwire[Navy]};\node at (2,5) {\jay[Navy]};\node at (3,5) {\vwire[Navy]};\node at (4,5) {\nowire};\node at (5,5) {\are[Orange]};\node at (6,5) {\hwire[Orange]};\node at (7,5) {\hwire[Orange]};\node at (0,4) {\vwire[Navy]};\node at (1,4) {\are[Navy]};\node at (2,4) {\hwire[Navy]};\node at (3,4) {\jay[Navy]};\node at (4,4) {\are[Orange]};\node at (5,4) {\hwire[Orange]};\node at (5,4) {\vwire[Orange]};\node at (6,4) {\hwire[Orange]};\node at (7,4) {\hwire[Orange]};\node at (0,3) {\vwire[Navy]};\node at (1,3) {\vwire[Navy]};\node at (2,3) {\are[Orange]};\node at (3,3) {\hwire[Orange]};\node at (4,3) {\jay[Orange]};\node at (5,3) {\vwire[Orange]};\node at (6,3) {\nowire};\node at (7,3) {\are[Navy]};\node at (0,2) {\vwire[Navy]};\node at (1,2) {\vwire[Navy]};\node at (2,2) {\vwire[Orange]};\node at (3,2) {\nowire};\node at (4,2) {\are[Orange]};\node at (5,2) {\jay[Orange]};\node at (6,2) {\are[Navy]};\node at (7,2) {\hwire[Navy]};\node at (7,2) {\vwire[Navy]};\node at (0,1) {\vwire[Navy]};\node at (1,1) {\vwire[Navy]};\node at (2,1) {\vwire[Orange]};\node at (3,1) {\nowire};\node at (4,1) {\vwire[Orange]};\node at (5,1) {\are[Navy]};\node at (6,1) {\hwire[Navy]};\node at (6,1) {\vwire[Navy]};\node at (7,1) {\hwire[Navy]};\node at (7,1) {\vwire[Navy]};\node at (0,0) {\vwire[Navy]};\node at (1,0) {\vwire[Navy]};\node at (2,0) {\vwire[Orange]};\node at (3,0) {\are[Navy]};\node at (4,0) {\hwire[Navy]};\node at (4,0) {\vwire[Orange]};\node at (5,0) {\hwire[Navy]};\node at (5,0) {\vwire[Navy]};\node at (6,0) {\hwire[Navy]};\node at (6,0) {\vwire[Navy]};\node at (7,0) {\hwire[Navy]};\node at (7,0) {\vwire[Navy]};\draw[thick] (-0.5,-0.5) -- (-0.5,7.5) -- (7.5,7.5) -- (7.5,-0.5) -- (-0.5,-0.5);\node at (0,-1) {1};\node at (8,7) {2};\node at (1,-1) {2};\node at (8,6) {1};\node at (2,-1) {3};\node at (8,5) {5};\node at (3,-1) {4};\node at (8,4) {3};\node at (4,-1) {5};\node at (8,3) {8};\node at (5,-1) {6};\node at (8,2) {7};\node at (6,-1) {7};\node at (8,1) {6};\node at (7,-1) {8};\node at (8,0) {4};\end{tikzpicture}
        \begin{tikzpicture}[scale =.4]\node at (1,5) {\shade};\node at (2,4) {\shade};\grid{7}{7}\node at (0,7) {\vwire[Orange]};\node at (1,7) {\vwire[Orange]};\node at (2,7) {\vwire[Navy]};\node at (3,7) {\el[Navy]};\node at (4,7) {\hwire[Navy]};\node at (4,7) {\vwire[Navy]};\node at (5,7) {\hwire[Navy]};\node at (5,7) {\vwire[Navy]};\node at (6,7) {\hwire[Navy]};\node at (6,7) {\vwire[Navy]};\node at (7,7) {\hwire[Navy]};\node at (7,7) {\vwire[Navy]};\node at (0,6) {\vwire[Orange]};\node at (1,6) {\vwire[Orange]};\node at (2,6) {\el[Navy]};\node at (3,6) {\hwire[Navy]};\node at (4,6) {\hwire[Navy]};\node at (4,6) {\vwire[Navy]};\node at (5,6) {\hwire[Navy]};\node at (5,6) {\vwire[Navy]};\node at (6,6) {\hwire[Navy]};\node at (6,6) {\vwire[Navy]};\node at (7,6) {\hwire[Navy]};\node at (7,6) {\vwire[Navy]};\node at (0,5) {\el[Orange]};\node at (1,5) {\hwire[Orange]};\node at (1,5) {\vwire[Orange]};\node at (2,5) {\en[Orange]};\node at (3,5) {\nowire};\node at (4,5) {\vwire[Navy]};\node at (5,5) {\el[Navy]};\node at (6,5) {\hwire[Navy]};\node at (6,5) {\vwire[Navy]};\node at (7,5) {\hwire[Navy]};\node at (7,5) {\vwire[Navy]};\node at (0,4) {\nowire};\node at (1,4) {\el[Orange]};\node at (2,4) {\newire[Orange]};\node at (2,4) {\swwire[Orange]};\node at (3,4) {\en[Orange]};\node at (4,4) {\el[Navy]};\node at (5,4) {\hwire[Navy]};\node at (6,4) {\hwire[Navy]};\node at (6,4) {\vwire[Navy]};\node at (7,4) {\hwire[Navy]};\node at (7,4) {\vwire[Navy]};\node at (0,3) {\nowire};\node at (1,3) {\nowire};\node at (2,3) {\el[Orange]};\node at (3,3) {\newire[Orange]};\node at (3,3) {\swwire[Orange]};\node at (4,3) {\en[Orange]};\node at (5,3) {\nowire};\node at (6,3) {\vwire[Navy]};\node at (7,3) {\el[Navy]};\node at (0,2) {\nowire};\node at (1,2) {\nowire};\node at (2,2) {\nowire};\node at (3,2) {\vwire[Orange]};\node at (4,2) {\el[Orange]};\node at (5,2) {\en[Orange]};\node at (6,2) {\el[Navy]};\node at (7,2) {\hwire[Navy]};\node at (0,1) {\nowire};\node at (1,1) {\nowire};\node at (2,1) {\nowire};\node at (3,1) {\vwire[Orange]};\node at (4,1) {\nowire};\node at (5,1) {\el[Orange]};\node at (6,1) {\hwire[Orange]};\node at (7,1) {\hwire[Orange]};\node at (0,0) {\nowire};\node at (1,0) {\nowire};\node at (2,0) {\nowire};\node at (3,0) {\el[Orange]};\node at (4,0) {\hwire[Orange]};\node at (5,0) {\hwire[Orange]};\node at (6,0) {\hwire[Orange]};\node at (7,0) {\hwire[Orange]};\draw[thick] (-0.5,-0.5) -- (-0.5,7.5) -- (7.5,7.5) -- (7.5,-0.5) -- (-0.5,-0.5);\node at (0,-1) {1};\node at (8,7) {4};\node at (1,-1) {2};\node at (8,6) {3};\node at (2,-1) {3};\node at (8,5) {6};\node at (3,-1) {4};\node at (8,4) {5};\node at (4,-1) {5};\node at (8,3) {8};\node at (5,-1) {6};\node at (8,2) {7};\node at (6,-1) {7};\node at (8,1) {1};\node at (7,-1) {8};\node at (8,0) {2};\end{tikzpicture}
        
        \caption{A $216543$ pattern with $(2,2)$ and $(2,3)$  blocking pipes.}
        \label{fig:12543Blocks}
\end{figure}
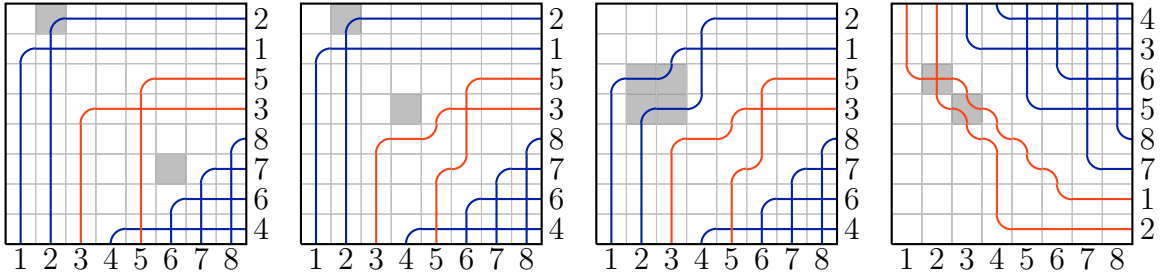

\begin{proposition}\label{p:215643}
    If $w$ contains a pattern $\pi=215643$, then there exists $B\in \BPD{w}$ such that co$(B)$ is non-reduced.
\end{proposition}

\begin{proof}
    The proof is omitted as it is exactly same as the proof for Proposition~\ref{p:216543} with the exception of different pipes taken to form the $241653$ or $1423$ patterns for those respective cases.
\end{proof}

This leads us to the final result of this section.

\newpage

\begin{theorem}\label{thm:reducedImpliesAvoid}
    If $\text{co-BPD}(w)$ has only reduced co-BPDs, then $w$ avoids all seven patterns from $\Pi$.
\end{theorem}
\begin{proof}
The result follows by contrapositive of Propositions~\ref{p:1423} through~\ref{p:215643}.
\end{proof}

\section{Concluding Remarks}\label{s:app}

As an application of Theorem~\ref{t:pattern-char} we have that if $w$ avoids all the patterns in $\Pi$, then every single co-BPD of $w$ contributes to the Schubert expansion of $\mathfrak{G}_w$.
This means that if we can state anything about the permutations of the co-BPDs of such a $w$, then we know those permutations appear in the Schubert expansion.
The following two corollaries give an example of such a result.

\begin{corollary}\label{c:nonreducedBPDcopatterns}
    If $B$ is a non-reduced BPD, then co($B$) traces out a permutation that contains at least one of pattern in $\Pi^r$, which is the reverse of each permutation in $\Pi$.
\end{corollary}

\begin{proof}
    Note that for any co-BPD, co($B$), flipping co($B$) vertically gives a BPD, f(co($B$)), whose permutation is the reverse of co($B$)'s permutation.
    If $B$ is non-reduced, then f(co($B$)) must have the configuration meaning f(co($B$)) contains a pattern in $\Pi$, thus co($B$) contains a pattern in $\Pi^r$.
\end{proof}

We can say a bit more in the case when, in addition to avoiding the patterns in $\Pi$, $w$ also {\em contains} a 2143 pattern. Note that Lemma 7.2 in \cite{weigandt2021bumpless} gives that $\text{bpd}(w) = \text{BPD}(w)$ whenever $w$ avoids the 2143 pattern, i.e., when $w$ is a {\em vexillary} permutation. The converse of this statement is true as well, which we now demonstrate via contrapositive. 

\begin{lemma}
    \label{l:all-red-implies-vex}
    If $w$ contains a 2143 pattern then $\text{bpd}(w) \subsetneq \text{BPD}(w)$, that is, there exist non-reduced bumpless pipe dreams for $w$.
\end{lemma}

\begin{proof}
    Suppose $w$ contains a 2143 pattern. There could be multiple instances of this pattern, but among all such occurrences, fix $i < j < k < \ell$ to be the one where $i$ and $j$ are as large as possible. Note that the pipes exit in the order $j \; i \; \ell \; k$. Consider the Rothe BPD $\widehat{B}$ for $w$. The inversion $(k, \ell)$ in $w$ means that on $\widehat{B}$ there will be a $\inlinenowire$ tile strictly southeast of the (unique) $\inlineare$ tiles of both $i$ and $j$ in $\widehat{B}$. We wish to droop $i$ into this $\inlinenowire$ tile, but there may be blocking pipes $j'$ with $j < j' < k$ that exit between $i$ and $\ell$, obstructing this droop. If there are such pipes, droop them into the $\inlinenowire$ tile in order from largest to smallest, moving the $\inlinenowire$ tile strictly northwest until we arrive at a BPD $B' \in \text{bpd}(w)$ that has a $\inlinenowire$ tile still strictly southeast of the $\inlineare$ tile in $i$ and $j$ and no longer has such blocking pipes. We may now droop $i$ into this $\inlinenowire$ tile to get a BPD $B'' \in \text{bpd}(w)$. The $\inlinenowire$ tile of $B'$ that $i$ drooped into is now a $\inlinejay$ tile in $B''$ that has a $\inlineare$ tile of $i$ above it in the same column. This $\inlineare$ tile of $i$ is strictly southeast of the $\inlineare$ tile of $j$. We wish to $K$-droop $j$ into this $\inlineare$ tile of $i$. In order to be a blocking pipe it would have to be some $j'$ with $j < j' < k$ that exits between $j$ and $i$, but then $j' \; i \; \ell \; k$ forms a $2143$ pattern, contradicting $j$ chosen to be maximal.
    So there are no obstructions blocking $j$ from making a $K$-droop move into this tile, resulting in a (non-reduced) $B \in \text{BPD}(w) \setminus \text{bpd}(w)$ as desired.
\end{proof}

We can use the previous two results to achieve the following corollary.


\begin{corollary}
    If $w$ contains a 2143 pattern and avoids each pattern in $\Pi$, then there exists $v$ containing a pattern in $\Pi^r$ such that $a_{w, v} > 0$ where
    \begin{equation*}
        \mathfrak{G}_w = \sum_v (-1)^{\ell(v)-\ell(w)}a_{w,v} \mathfrak{S}_v.
    \end{equation*}
\end{corollary}

\begin{proof}
    By Theorem~\ref{t:pattern-char} we have that for any $B\in\BPD{w}$, co$(B)$ is reduced. Therefore, by Theorem 1.1 in~\cite{weigandt2025changingbasespipedream} (cf. Theorem~\ref{thm:g to s} above), we have that if the permutation associated with co$(B)$ is $v$ then $a_{w,v}$ is nonzero.
    By Lemma~\ref{l:all-red-implies-vex}, we know $w$ has some non-reduced BPD, say $B$.
    By Corollary~\ref{c:nonreducedBPDcopatterns} we have that co$(B)$ whose associated permutation, $v$, contains a pattern in $\Pi^r$ meaning $a_{w,v}$ is nonzero. 
\end{proof}

\appendix

\section{Droop move sequences}
\label{a:droops}

Here we demonstrate the sequence of droop moves starting on the Rothe BPD of a given permutation in $\Pi$ that produces an instance of the configuration from Lemma~\ref{l:config}.
\begin{center}
        \begin{tabular}{c|c}
       Permutation  & Droop Moves\\
       \hline
        1423 & $(1,1) \ssearrow (2,3)$\\
        \hline
        12543 & $(2,2) \ssearrow (3,4)$\\
        & $(1,1) \ssearrow (2,3)$\\
        \hline
        13254 & $(2,3) \ssearrow (4,4)$\\
        & $(1,1) \ssearrow (2,3)$\\
        \hline
        25143 & $(1,2) \ssearrow (2,4)$\\
        & $(3,1) \ssearrow (4,3)$\\
        \hline
        215643 & $(1,2) \ssearrow (4,4)$\\
        & $(2,1) \ssearrow (3,3)$\\
        \hline
        216543 & $(1,2) \ssearrow (4,4)$\\
        & $(2,1) \ssearrow (3,3)$\\
        \hline
        241653 & $(2,4) \ssearrow (4,5)$\\
        & $(1,2) \ssearrow (2,4)$\\
        & $(3,1) \ssearrow (4,3)$
    \end{tabular}
    \end{center}

\[
\begin{tikzpicture}[scale=0.4]
\node at (0,0) {\oldvwire[UF2]};
\node at (1,0) {\oldvwire};
\node at (2,0) {\oldare};
\node at (3,0) {\cross};

\node at (0,1) {\oldvwire[UF2]};
\node at (1,1) {\oldare};
\node at (2,1) {\oldhwire};
\node at (3,1) {\cross};

\node at (0,2) {\oldvwire[UF2]};
\node at (1,2) {\nowire};
\node at (2,2) {\nowire};
\node at (3,2) {\oldare};

\node at (0,3) {\oldare[UF2]};
\node at (1,3) {\oldhwire[UF2]};
\node at (2,3) {\oldhwire[UF2]};
\node at (3,3) {\oldhwire[UF2]};

\draw[thick] (-0.5,-0.5) -- (-0.5,3.5) -- (3.5,3.5) -- (3.5,-0.5) -- (-0.5,-0.5);

\node at (0,-1) {1};
\node at (1,-1) {2};
\node at (2,-1) {3};
\node at (3,-1) {4};

\node at (4,3) {1};
\node at (4,2) {4};
\node at (4,1) {2};
\node at (4,0) {3};

\end{tikzpicture} \; \hspace{0.5in} \;
\begin{tikzpicture}[scale=0.4]
\node at (0,0) {\oldvwire[UF2]};
\node at (1,0) {\oldvwire};
\node at (2,0) {\oldare};
\node at (3,0) {\cross};

\node at (0,1) {\oldvwire[UF2]};
\node at (1,1) {\shade};
\node at (1,1) {\oldare};
\node at (2,1) {\shade};
\node at (2,1) {\oldhwire};
\node at (3,1) {\cross};

\node at (0,2) {\oldare[UF2]};
\node at (1,2) {\shade};
\node at (1,2) {\oldhwire[UF2]};
\node at (2,2) {\shade};
\node at (2,2) {\oldjay[UF2]};
\node at (3,2) {\oldare};

\node at (0,3) {\nowire};
\node at (1,3) {\nowire};
\node at (2,3) {\oldare[UF2]};
\node at (3,3) {\oldhwire[UF2]};

\draw[thick] (-0.5,-0.5) -- (-0.5,3.5) -- (3.5,3.5) -- (3.5,-0.5) -- (-0.5,-0.5);

\node at (0,-1) {1};
\node at (1,-1) {2};
\node at (2,-1) {3};
\node at (3,-1) {4};

\node at (4,3) {1};
\node at (4,2) {4};
\node at (4,1) {2};
\node at (4,0) {3};
\end{tikzpicture} \; \hspace{0.5in} \;
\begin{tikzpicture}[scale =.4]\node at (2,1) {\shade};\node at (1,2) {\shade};\node at (1,1) {\shade};\node at (2,2) {\shade};\grid{3}{3}\node at (0,3) {\vwire[Navy]};\node at (1,3) {\vwire[Navy]};\node at (2,3) {\el[Navy]};\node at (3,3) {\hwire[Navy]};\node at (3,3) {\vwire[Navy]};\node at (0,2) {\el[Navy]};\node at (1,2) {\hwire[Navy]};\node at (1,2) {\vwire[Navy]};\node at (2,2) {\en[Navy]};\node at (3,2) {\el[Navy]};\node at (0,1) {\nowire};\node at (1,1) {\el[Navy]};\node at (2,1) {\newire[Navy]};\node at (2,1) {\swwire[Navy]};\node at (3,1) {\hwire[Navy]};\node at (0,0) {\nowire};\node at (1,0) {\nowire};\node at (2,0) {\el[Navy]};\node at (3,0) {\hwire[Navy]};\draw[thick] (-0.5,-0.5) -- (-0.5,3.5) -- (3.5,3.5) -- (3.5,-0.5) -- (-0.5,-0.5);\node at (0,-1) {1};\node at (4,3) {3};\node at (1,-1) {2};\node at (4,2) {4};\node at (2,-1) {3};\node at (4,1) {1};\node at (3,-1) {4};\node at (4,0) {2};\end{tikzpicture} 
\]

\[
\begin{tikzpicture}[scale =.4]
\node at (0,4) {\oldare[UF2]};
\node at (1,4) {\oldhwire[UF2]};
\node at (2,4) {\oldhwire[UF2]};
\node at (3,4) {\oldhwire[UF2]};
\node at (4,4) {\oldhwire[UF2]};
\node at (0,3) {\oldvwire[UF2]};
\node at (1,3) {\oldare[UF2]};
\node at (2,3) {\oldhwire[UF2]};
\node at (3,3) {\oldhwire[UF2]};
\node at (4,3) {\oldhwire[UF2]};
\node at (0,2) {\oldvwire[UF2]};
\node at (1,2) {\oldvwire[UF2]};
\node at (2,2) {\nowire};
\node at (3,2) {\nowire};
\node at (4,2) {\oldare};
\node at (0,1) {\oldvwire[UF2]};
\node at (1,1) {\oldvwire[UF2]};
\node at (2,1) {\nowire};
\node at (3,1) {\oldare};
\node at (4,1) {\cross};
\node at (0,0) {\oldvwire[UF2]};
\node at (1,0) {\oldvwire[UF2]};
\node at (2,0) {\oldare};
\node at (3,0) {\cross};
\node at (4,0) {\cross};

\draw[thick] (-0.5,-0.5) -- (-0.5,4.5) -- (4.5,4.5) -- (4.5,-0.5) -- (-0.5,-0.5);

\node at (0,-1) {1};
\node at (1,-1) {2};
\node at (2,-1) {3};
\node at (3,-1) {4};
\node at (4,-1) {5};

\node at (5,4) {1};
\node at (5,3) {2};
\node at (5,2) {5};
\node at (5,1) {4};
\node at (5,0) {3};

\end{tikzpicture}
 \; \hspace{0.5in} \;
\begin{tikzpicture}[scale =.4]
\node at (0,4) {\oldare[UF2]};
\node at (1,4) {\oldhwire[UF2]};
\node at (2,4) {\oldhwire[UF2]};
\node at (3,4) {\oldhwire[UF2]};
\node at (4,4) {\oldhwire[UF2]};
\node at (0,3) {\oldvwire[UF2]};
\node at (1,3) {\nowire};
\node at (2,3) {\nowire};
\node at (3,3) {\oldare[UF2]};
\node at (4,3) {\oldhwire[UF2]};
\node at (0,2) {\oldvwire[UF2]};
\node at (1,2) {\oldare[UF2]};
\node at (2,2) {\oldhwire[UF2]};
\node at (3,2) {\oldjay[UF2]};
\node at (4,2) {\oldare};
\node at (0,1) {\oldvwire[UF2]};
\node at (1,1) {\oldvwire[UF2]};
\node at (2,1) {\nowire};
\node at (3,1) {\oldare};
\node at (4,1) {\cross};
\node at (0,0) {\oldvwire[UF2]};
\node at (1,0) {\oldvwire[UF2]};
\node at (2,0) {\oldare};
\node at (3,0) {\cross};
\node at (4,0) {\cross};

\draw[thick] (-0.5,-0.5) -- (-0.5,4.5) -- (4.5,4.5) -- (4.5,-0.5) -- (-0.5,-0.5);

\node at (0,-1) {1};
\node at (1,-1) {2};
\node at (2,-1) {3};
\node at (3,-1) {4};
\node at (4,-1) {5};

\node at (5,4) {1};
\node at (5,3) {2};
\node at (5,2) {5};
\node at (5,1) {4};
\node at (5,0) {3};

\end{tikzpicture} \; \hspace{0.5in} \;
\begin{tikzpicture}[scale =.4]
\node at (1,3) {\shade};
\node at (2,3) {\shade};
\node at (1,2) {\shade};
\node at (2,2) {\shade};

\node at (0,4) {\nowire};
\node at (1,4) {\nowire};
\node at (2,4) {\oldare[UF2]};
\node at (3,4) {\oldhwire[UF2]};
\node at (4,4) {\oldhwire[UF2]};
\node at (0,3) {\oldare[UF2]};
\node at (1,3) {\oldhwire[UF2]};
\node at (2,3) {\oldjay[UF2]};
\node at (3,3) {\oldare[UF2]};
\node at (4,3) {\oldhwire[UF2]};
\node at (0,2) {\oldvwire[UF2]};
\node at (1,2) {\oldare[UF2]};
\node at (2,2) {\oldhwire[UF2]};
\node at (3,2) {\oldjay[UF2]};
\node at (4,2) {\oldare};
\node at (0,1) {\oldvwire[UF2]};
\node at (1,1) {\oldvwire[UF2]};
\node at (2,1) {\nowire};
\node at (3,1) {\oldare};
\node at (4,1) {\cross};
\node at (0,0) {\oldvwire[UF2]};
\node at (1,0) {\oldvwire[UF2]};
\node at (2,0) {\oldare};
\node at (3,0) {\cross};
\node at (4,0) {\cross};

\draw[thick] (-0.5,-0.5) -- (-0.5,4.5) -- (4.5,4.5) -- (4.5,-0.5) -- (-0.5,-0.5);

\node at (0,-1) {1};
\node at (1,-1) {2};
\node at (2,-1) {3};
\node at (3,-1) {4};
\node at (4,-1) {5};

\node at (5,4) {1};
\node at (5,3) {2};
\node at (5,2) {5};
\node at (5,1) {4};
\node at (5,0) {3};

\end{tikzpicture} \; \hspace{0.5in} \;
\begin{tikzpicture}[scale =.4]\node at (1,3) {\shade};\node at (2,3) {\shade};\node at (2,2) {\shade};\node at (1,2) {\shade};\grid{4}{4}\node at (0,4) {\vwire[Navy]};\node at (1,4) {\vwire[Navy]};\node at (2,4) {\el[Navy]};\node at (3,4) {\hwire[Navy]};\node at (3,4) {\vwire[Navy]};\node at (4,4) {\hwire[Navy]};\node at (4,4) {\vwire[Navy]};\node at (0,3) {\el[Navy]};\node at (1,3) {\hwire[Navy]};\node at (1,3) {\vwire[Navy]};\node at (2,3) {\en[Navy]};\node at (3,3) {\el[Navy]};\node at (4,3) {\hwire[Navy]};\node at (4,3) {\vwire[Navy]};\node at (0,2) {\nowire};\node at (1,2) {\el[Navy]};\node at (2,2) {\newire[Navy]};\node at (2,2) {\swwire[Navy]};\node at (3,2) {\en[Navy]};\node at (4,2) {\el[Navy]};\node at (0,1) {\nowire};\node at (1,1) {\nowire};\node at (2,1) {\vwire[Navy]};\node at (3,1) {\el[Navy]};\node at (4,1) {\hwire[Navy]};\node at (0,0) {\nowire};\node at (1,0) {\nowire};\node at (2,0) {\el[Navy]};\node at (3,0) {\hwire[Navy]};\node at (4,0) {\hwire[Navy]};\draw[thick] (-0.5,-0.5) -- (-0.5,4.5) -- (4.5,4.5) -- (4.5,-0.5) -- (-0.5,-0.5);\node at (0,-1) {1};\node at (5,4) {3};\node at (1,-1) {2};\node at (5,3) {4};\node at (2,-1) {3};\node at (5,2) {5};\node at (3,-1) {4};\node at (5,1) {1};\node at (4,-1) {5};\node at (5,0) {2};\end{tikzpicture}
\]

\[
\begin{tikzpicture}[scale =.4]
\node at (0,4) {\oldare[UF2]};
\node at (1,4) {\oldhwire[UF2]};
\node at (2,4) {\oldhwire[UF2]};
\node at (3,4) {\oldhwire[UF2]};
\node at (4,4) {\oldhwire[UF2]};
\node at (0,3) {\oldvwire[UF2]};
\node at (1,3) {\nowire};
\node at (2,3) {\oldare[UF2]};
\node at (3,3) {\oldhwire[UF2]};
\node at (4,3) {\oldhwire[UF2]};
\node at (0,2) {\oldvwire[UF2]};
\node at (1,2) {\oldare};
\node at (2,2) {\oldhwire};
\node at (2,2) {\oldvwire[UF2]};
\node at (3,2) {\oldhwire};
\node at (4,2) {\oldhwire};
\node at (0,1) {\oldvwire[UF2]};
\node at (1,1) {\oldvwire};
\node at (2,1) {\oldvwire[UF2]};
\node at (3,1) {\nowire};
\node at (4,1) {\oldare};
\node at (0,0) {\oldvwire[UF2]};
\node at (1,0) {\oldvwire};
\node at (2,0) {\oldvwire[UF2]};
\node at (3,0) {\oldare};
\node at (4,0) {\cross};

\draw[thick] (-0.5,-0.5) -- (-0.5,4.5) -- (4.5,4.5) -- (4.5,-0.5) -- (-0.5,-0.5);

\node at (0,-1) {1};
\node at (1,-1) {2};
\node at (2,-1) {3};
\node at (3,-1) {4};
\node at (4,-1) {5};

\node at (5,4) {1};
\node at (5,3) {3};
\node at (5,2) {2};
\node at (5,1) {5};
\node at (5,0) {4};

\end{tikzpicture}
 \; \hspace{0.5in} \;
\begin{tikzpicture}[scale =.4]
\node at (0,4) {\oldare[UF2]};
\node at (1,4) {\oldhwire[UF2]};
\node at (2,4) {\oldhwire[UF2]};
\node at (3,4) {\oldhwire[UF2]};
\node at (4,4) {\oldhwire[UF2]};
\node at (0,3) {\oldvwire[UF2]};
\node at (1,3) {\nowire};
\node at (2,3) {\nowire};
\node at (3,3) {\oldare[UF2]};
\node at (4,3) {\oldhwire[UF2]};
\node at (0,2) {\oldvwire[UF2]};
\node at (1,2) {\oldare};
\node at (2,2) {\oldhwire};
\node at (3,2) {\oldhwire};
\node at (3,2) {\oldvwire[UF2]};
\node at (4,2) {\oldhwire};
\node at (0,1) {\oldvwire[UF2]};
\node at (1,1) {\oldvwire};
\node at (2,1) {\oldare[UF2]};
\node at (3,1) {\oldjay[UF2]};
\node at (4,1) {\oldare};
\node at (0,0) {\oldvwire[UF2]};
\node at (1,0) {\oldvwire};
\node at (2,0) {\oldvwire[UF2]};
\node at (3,0) {\oldare};
\node at (4,0) {\cross};

\draw[thick] (-0.5,-0.5) -- (-0.5,4.5) -- (4.5,4.5) -- (4.5,-0.5) -- (-0.5,-0.5);

\node at (0,-1) {1};
\node at (1,-1) {2};
\node at (2,-1) {3};
\node at (3,-1) {4};
\node at (4,-1) {5};

\node at (5,4) {1};
\node at (5,3) {3};
\node at (5,2) {2};
\node at (5,1) {5};
\node at (5,0) {4};

\end{tikzpicture} \; \hspace{0.5in} \;
\begin{tikzpicture}[scale =.4]
\node at (1,3) {\shade};
\node at (2,3) {\shade};
\node at (1,2) {\shade};
\node at (2,2) {\shade};

\node at (0,4) {\nowire};
\node at (1,4) {\nowire};
\node at (2,4) {\oldare[UF2]};
\node at (3,4) {\oldhwire[UF2]};
\node at (4,4) {\oldhwire[UF2]};
\node at (0,3) {\oldare[UF2]};
\node at (1,3) {\oldhwire[UF2]};
\node at (2,3) {\oldjay[UF2]};
\node at (3,3) {\oldare[UF2]};
\node at (4,3) {\oldhwire[UF2]};
\node at (0,2) {\oldvwire[UF2]};
\node at (1,2) {\oldare};
\node at (2,2) {\oldhwire};
\node at (3,2) {\oldhwire};
\node at (3,2) {\oldvwire[UF2]};
\node at (4,2) {\oldhwire};
\node at (0,1) {\oldvwire[UF2]};
\node at (1,1) {\oldvwire};
\node at (2,1) {\oldare[UF2]};
\node at (3,1) {\oldjay[UF2]};
\node at (4,1) {\oldare};
\node at (0,0) {\oldvwire[UF2]};
\node at (1,0) {\oldvwire};
\node at (2,0) {\oldvwire[UF2]};
\node at (3,0) {\oldare};
\node at (4,0) {\cross};

\draw[thick] (-0.5,-0.5) -- (-0.5,4.5) -- (4.5,4.5) -- (4.5,-0.5) -- (-0.5,-0.5);

\node at (0,-1) {1};
\node at (1,-1) {2};
\node at (2,-1) {3};
\node at (3,-1) {4};
\node at (4,-1) {5};

\node at (5,4) {1};
\node at (5,3) {3};
\node at (5,2) {2};
\node at (5,1) {5};
\node at (5,0) {4};

\end{tikzpicture} \; \hspace{0.5in} \;
\begin{tikzpicture}[scale =.4]\node at (1,3) {\shade};\node at (1,2) {\shade};\node at (2,2) {\shade};\node at (2,3) {\shade};\grid{4}{4}\node at (0,4) {\vwire[Navy]};\node at (1,4) {\vwire[Navy]};\node at (2,4) {\el[Navy]};\node at (3,4) {\hwire[Navy]};\node at (3,4) {\vwire[Navy]};\node at (4,4) {\hwire[Navy]};\node at (4,4) {\vwire[Navy]};\node at (0,3) {\el[Navy]};\node at (1,3) {\hwire[Navy]};\node at (1,3) {\vwire[Navy]};\node at (2,3) {\en[Navy]};\node at (3,3) {\el[Navy]};\node at (4,3) {\hwire[Navy]};\node at (4,3) {\vwire[Navy]};\node at (0,2) {\nowire};\node at (1,2) {\el[Navy]};\node at (2,2) {\newire[Navy]};\node at (2,2) {\swwire[Navy]};\node at (3,2) {\hwire[Navy]};\node at (4,2) {\hwire[Navy]};\node at (4,2) {\vwire[Navy]};\node at (0,1) {\nowire};\node at (1,1) {\nowire};\node at (2,1) {\el[Navy]};\node at (3,1) {\en[Navy]};\node at (4,1) {\el[Navy]};\node at (0,0) {\nowire};\node at (1,0) {\nowire};\node at (2,0) {\nowire};\node at (3,0) {\el[Navy]};\node at (4,0) {\hwire[Navy]};\draw[thick] (-0.5,-0.5) -- (-0.5,4.5) -- (4.5,4.5) -- (4.5,-0.5) -- (-0.5,-0.5);\node at (0,-1) {1};\node at (5,4) {3};\node at (1,-1) {2};\node at (5,3) {4};\node at (2,-1) {3};\node at (5,2) {1};\node at (3,-1) {4};\node at (5,1) {5};\node at (4,-1) {5};\node at (5,0) {2};\end{tikzpicture}
\]

\[
\begin{tikzpicture}[scale =.4]\grid{4}{4}\node at (0,4) {\nowire};\node at (1,4) {\are[Orange]};\node at (2,4) {\hwire[Orange]};\node at (3,4) {\hwire[Orange]};\node at (4,4) {\hwire[Orange]};\node at (0,3) {\nowire};\node at (1,3) {\vwire[Orange]};\node at (2,3) {\nowire};\node at (3,3) {\nowire};\node at (4,3) {\are[Navy]};\node at (0,2) {\are[Orange]};\node at (1,2) {\hwire[Orange]};\node at (1,2) {\vwire[Orange]};\node at (2,2) {\hwire[Orange]};\node at (3,2) {\hwire[Orange]};\node at (4,2) {\hwire[Orange]};\node at (4,2) {\vwire[Navy]};\node at (0,1) {\vwire[Orange]};\node at (1,1) {\vwire[Orange]};\node at (2,1) {\nowire};\node at (3,1) {\are[Navy]};\node at (4,1) {\hwire[Navy]};\node at (4,1) {\vwire[Navy]};\node at (0,0) {\vwire[Orange]};\node at (1,0) {\vwire[Orange]};\node at (2,0) {\are[Navy]};\node at (3,0) {\hwire[Navy]};\node at (3,0) {\vwire[Navy]};\node at (4,0) {\hwire[Navy]};\node at (4,0) {\vwire[Navy]};\draw[thick] (-0.5,-0.5) -- (-0.5,4.5) -- (4.5,4.5) -- (4.5,-0.5) -- (-0.5,-0.5);\node at (0,-1) {1};\node at (5,4) {2};\node at (1,-1) {2};\node at (5,3) {5};\node at (2,-1) {3};\node at (5,2) {1};\node at (3,-1) {4};\node at (5,1) {4};\node at (4,-1) {5};\node at (5,0) {3};\end{tikzpicture}
 \; \hspace{0.5in} \;
\begin{tikzpicture}[scale =.4]\grid{4}{4}\node at (0,4) {\nowire};\node at (1,4) {\nowire};\node at (2,4) {\nowire};\node at (3,4) {\are[Orange]};\node at (4,4) {\hwire[Orange]};\node at (0,3) {\nowire};\node at (1,3) {\are[Orange]};\node at (2,3) {\hwire[Orange]};\node at (3,3) {\jay[Orange]};\node at (4,3) {\are[Navy]};\node at (0,2) {\are[Orange]};\node at (1,2) {\hwire[Orange]};\node at (1,2) {\vwire[Orange]};\node at (2,2) {\hwire[Orange]};\node at (3,2) {\hwire[Orange]};\node at (4,2) {\hwire[Orange]};\node at (4,2) {\vwire[Navy]};\node at (0,1) {\vwire[Orange]};\node at (1,1) {\vwire[Orange]};\node at (2,1) {\nowire};\node at (3,1) {\are[Navy]};\node at (4,1) {\hwire[Navy]};\node at (4,1) {\vwire[Navy]};\node at (0,0) {\vwire[Orange]};\node at (1,0) {\vwire[Orange]};\node at (2,0) {\are[Navy]};\node at (3,0) {\hwire[Navy]};\node at (3,0) {\vwire[Navy]};\node at (4,0) {\hwire[Navy]};\node at (4,0) {\vwire[Navy]};\draw[thick] (-0.5,-0.5) -- (-0.5,4.5) -- (4.5,4.5) -- (4.5,-0.5) -- (-0.5,-0.5);\node at (0,-1) {1};\node at (5,4) {2};\node at (1,-1) {2};\node at (5,3) {5};\node at (2,-1) {3};\node at (5,2) {1};\node at (3,-1) {4};\node at (5,1) {4};\node at (4,-1) {5};\node at (5,0) {3};\end{tikzpicture}
 \; \hspace{0.5in} \;
\begin{tikzpicture}[scale =.4]\node at (2,3) {\shade};\node at (2,2) {\shade};\node at (3,3) {\shade};\node at (3,2) {\shade};\grid{4}{4}\node at (0,4) {\nowire};\node at (1,4) {\nowire};\node at (2,4) {\nowire};\node at (3,4) {\are[Orange]};\node at (4,4) {\hwire[Orange]};\node at (0,3) {\nowire};\node at (1,3) {\are[Orange]};\node at (2,3) {\hwire[Orange]};\node at (3,3) {\jay[Orange]};\node at (4,3) {\are[Navy]};\node at (0,2) {\nowire};\node at (1,2) {\vwire[Orange]};\node at (2,2) {\are[Orange]};\node at (3,2) {\hwire[Orange]};\node at (4,2) {\hwire[Orange]};\node at (4,2) {\vwire[Navy]};\node at (0,1) {\are[Orange]};\node at (1,1) {\hwire[Orange]};\node at (1,1) {\vwire[Orange]};\node at (2,1) {\jay[Orange]};\node at (3,1) {\are[Navy]};\node at (4,1) {\hwire[Navy]};\node at (4,1) {\vwire[Navy]};\node at (0,0) {\vwire[Orange]};\node at (1,0) {\vwire[Orange]};\node at (2,0) {\are[Navy]};\node at (3,0) {\hwire[Navy]};\node at (3,0) {\vwire[Navy]};\node at (4,0) {\hwire[Navy]};\node at (4,0) {\vwire[Navy]};\draw[thick] (-0.5,-0.5) -- (-0.5,4.5) -- (4.5,4.5) -- (4.5,-0.5) -- (-0.5,-0.5);\node at (0,-1) {1};\node at (5,4) {2};\node at (1,-1) {2};\node at (5,3) {5};\node at (2,-1) {3};\node at (5,2) {1};\node at (3,-1) {4};\node at (5,1) {4};\node at (4,-1) {5};\node at (5,0) {3};\end{tikzpicture}
 \; \hspace{0.5in} \;
\begin{tikzpicture}[scale =.4]\node at (2,3) {\shade};\node at (2,2) {\shade};\node at (3,3) {\shade};\node at (3,2) {\shade};\grid{4}{4}\node at (0,4) {\vwire[Navy]};\node at (1,4) {\vwire[Navy]};\node at (2,4) {\vwire[Navy]};\node at (3,4) {\el[Navy]};\node at (4,4) {\hwire[Navy]};\node at (4,4) {\vwire[Navy]};\node at (0,3) {\vwire[Navy]};\node at (1,3) {\el[Navy]};\node at (2,3) {\hwire[Navy]};\node at (2,3) {\vwire[Navy]};\node at (3,3) {\en[Navy]};\node at (4,3) {\el[Navy]};\node at (0,2) {\vwire[Navy]};\node at (1,2) {\nowire};\node at (2,2) {\el[Navy]};\node at (3,2) {\newire[Navy]};\node at (3,2) {\swwire[Navy]};\node at (4,2) {\hwire[Navy]};\node at (0,1) {\el[Navy]};\node at (1,1) {\hwire[Navy]};\node at (2,1) {\en[Navy]};\node at (3,1) {\el[Navy]};\node at (4,1) {\hwire[Navy]};\node at (0,0) {\nowire};\node at (1,0) {\nowire};\node at (2,0) {\el[Navy]};\node at (3,0) {\hwire[Navy]};\node at (4,0) {\hwire[Navy]};\draw[thick] (-0.5,-0.5) -- (-0.5,4.5) -- (4.5,4.5) -- (4.5,-0.5) -- (-0.5,-0.5);\node at (0,-1) {1};\node at (5,4) {4};\node at (1,-1) {2};\node at (5,3) {5};\node at (2,-1) {3};\node at (5,2) {2};\node at (3,-1) {4};\node at (5,1) {3};\node at (4,-1) {5};\node at (5,0) {1};\end{tikzpicture}
\]

\[
\begin{tikzpicture}[scale =.4]\grid{5}{5}\node at (0,5) {\nowire};\node at (1,5) {\are[Orange]};\node at (2,5) {\hwire[Orange]};\node at (3,5) {\hwire[Orange]};\node at (4,5) {\hwire[Orange]};\node at (5,5) {\hwire[Orange]};\node at (0,4) {\are[Orange]};\node at (1,4) {\hwire[Orange]};\node at (1,4) {\vwire[Orange]};\node at (2,4) {\hwire[Orange]};\node at (3,4) {\hwire[Orange]};\node at (4,4) {\hwire[Orange]};\node at (5,4) {\hwire[Orange]};\node at (0,3) {\vwire[Orange]};\node at (1,3) {\vwire[Orange]};\node at (2,3) {\nowire};\node at (3,3) {\nowire};\node at (4,3) {\are[Navy]};\node at (5,3) {\hwire[Navy]};\node at (0,2) {\vwire[Orange]};\node at (1,2) {\vwire[Orange]};\node at (2,2) {\nowire};\node at (3,2) {\nowire};\node at (4,2) {\vwire[Navy]};\node at (5,2) {\are[Navy]};\node at (0,1) {\vwire[Orange]};\node at (1,1) {\vwire[Orange]};\node at (2,1) {\nowire};\node at (3,1) {\are[Navy]};\node at (4,1) {\hwire[Navy]};\node at (4,1) {\vwire[Navy]};\node at (5,1) {\hwire[Navy]};\node at (5,1) {\vwire[Navy]};\node at (0,0) {\vwire[Orange]};\node at (1,0) {\vwire[Orange]};\node at (2,0) {\are[Navy]};\node at (3,0) {\hwire[Navy]};\node at (3,0) {\vwire[Navy]};\node at (4,0) {\hwire[Navy]};\node at (4,0) {\vwire[Navy]};\node at (5,0) {\hwire[Navy]};\node at (5,0) {\vwire[Navy]};\draw[thick] (-0.5,-0.5) -- (-0.5,5.5) -- (5.5,5.5) -- (5.5,-0.5) -- (-0.5,-0.5);\node at (0,-1) {1};\node at (6,5) {2};\node at (1,-1) {2};\node at (6,4) {1};\node at (2,-1) {3};\node at (6,3) {5};\node at (3,-1) {4};\node at (6,2) {6};\node at (4,-1) {5};\node at (6,1) {4};\node at (5,-1) {6};\node at (6,0) {3};\end{tikzpicture}
\; \hspace{0.3in} \;
\begin{tikzpicture}[scale =.4]\grid{5}{5}\node at (0,5) {\nowire};\node at (1,5) {\nowire};\node at (2,5) {\nowire};\node at (3,5) {\are[Orange]};\node at (4,5) {\hwire[Orange]};\node at (5,5) {\hwire[Orange]};\node at (0,4) {\are[Orange]};\node at (1,4) {\hwire[Orange]};\node at (2,4) {\hwire[Orange]};\node at (3,4) {\hwire[Orange]};\node at (3,4) {\vwire[Orange]};\node at (4,4) {\hwire[Orange]};\node at (5,4) {\hwire[Orange]};\node at (0,3) {\vwire[Orange]};\node at (1,3) {\nowire};\node at (2,3) {\nowire};\node at (3,3) {\vwire[Orange]};\node at (4,3) {\are[Navy]};\node at (5,3) {\hwire[Navy]};\node at (0,2) {\vwire[Orange]};\node at (1,2) {\are[Orange]};\node at (2,2) {\hwire[Orange]};\node at (3,2) {\jay[Orange]};\node at (4,2) {\vwire[Navy]};\node at (5,2) {\are[Navy]};\node at (0,1) {\vwire[Orange]};\node at (1,1) {\vwire[Orange]};\node at (2,1) {\nowire};\node at (3,1) {\are[Navy]};\node at (4,1) {\hwire[Navy]};\node at (4,1) {\vwire[Navy]};\node at (5,1) {\hwire[Navy]};\node at (5,1) {\vwire[Navy]};\node at (0,0) {\vwire[Orange]};\node at (1,0) {\vwire[Orange]};\node at (2,0) {\are[Navy]};\node at (3,0) {\hwire[Navy]};\node at (3,0) {\vwire[Navy]};\node at (4,0) {\hwire[Navy]};\node at (4,0) {\vwire[Navy]};\node at (5,0) {\hwire[Navy]};\node at (5,0) {\vwire[Navy]};\draw[thick] (-0.5,-0.5) -- (-0.5,5.5) -- (5.5,5.5) -- (5.5,-0.5) -- (-0.5,-0.5);\node at (0,-1) {1};\node at (6,5) {2};\node at (1,-1) {2};\node at (6,4) {1};\node at (2,-1) {3};\node at (6,3) {5};\node at (3,-1) {4};\node at (6,2) {6};\node at (4,-1) {5};\node at (6,1) {4};\node at (5,-1) {6};\node at (6,0) {3};\end{tikzpicture}
\; \hspace{0.3in}\;
\begin{tikzpicture}[scale =.4]\node at (1,3) {\shade};\node at (1,2) {\shade};\node at (2,3) {\shade};\node at (2,2) {\shade};\grid{5}{5}\node at (0,5) {\nowire};\node at (1,5) {\nowire};\node at (2,5) {\nowire};\node at (3,5) {\are[Orange]};\node at (4,5) {\hwire[Orange]};\node at (5,5) {\hwire[Orange]};\node at (0,4) {\nowire};\node at (1,4) {\nowire};\node at (2,4) {\are[Orange]};\node at (3,4) {\hwire[Orange]};\node at (3,4) {\vwire[Orange]};\node at (4,4) {\hwire[Orange]};\node at (5,4) {\hwire[Orange]};\node at (0,3) {\are[Orange]};\node at (1,3) {\hwire[Orange]};\node at (2,3) {\jay[Orange]};\node at (3,3) {\vwire[Orange]};\node at (4,3) {\are[Navy]};\node at (5,3) {\hwire[Navy]};\node at (0,2) {\vwire[Orange]};\node at (1,2) {\are[Orange]};\node at (2,2) {\hwire[Orange]};\node at (3,2) {\jay[Orange]};\node at (4,2) {\vwire[Navy]};\node at (5,2) {\are[Navy]};\node at (0,1) {\vwire[Orange]};\node at (1,1) {\vwire[Orange]};\node at (2,1) {\nowire};\node at (3,1) {\are[Navy]};\node at (4,1) {\hwire[Navy]};\node at (4,1) {\vwire[Navy]};\node at (5,1) {\hwire[Navy]};\node at (5,1) {\vwire[Navy]};\node at (0,0) {\vwire[Orange]};\node at (1,0) {\vwire[Orange]};\node at (2,0) {\are[Navy]};\node at (3,0) {\hwire[Navy]};\node at (3,0) {\vwire[Navy]};\node at (4,0) {\hwire[Navy]};\node at (4,0) {\vwire[Navy]};\node at (5,0) {\hwire[Navy]};\node at (5,0) {\vwire[Navy]};\draw[thick] (-0.5,-0.5) -- (-0.5,5.5) -- (5.5,5.5) -- (5.5,-0.5) -- (-0.5,-0.5);\node at (0,-1) {1};\node at (6,5) {2};\node at (1,-1) {2};\node at (6,4) {1};\node at (2,-1) {3};\node at (6,3) {5};\node at (3,-1) {4};\node at (6,2) {6};\node at (4,-1) {5};\node at (6,1) {4};\node at (5,-1) {6};\node at (6,0) {3};\end{tikzpicture}
\; \hspace{0.3in} \;
\begin{tikzpicture}[scale =.4]\node at (1,3) {\shade};\node at (1,2) {\shade};\node at (2,3) {\shade};\node at (2,2) {\shade};\grid{5}{5}\node at (0,5) {\vwire[Navy]};\node at (1,5) {\vwire[Navy]};\node at (2,5) {\vwire[Navy]};\node at (3,5) {\el[Navy]};\node at (4,5) {\hwire[Navy]};\node at (4,5) {\vwire[Navy]};\node at (5,5) {\hwire[Navy]};\node at (5,5) {\vwire[Navy]};\node at (0,4) {\vwire[Navy]};\node at (1,4) {\vwire[Navy]};\node at (2,4) {\el[Navy]};\node at (3,4) {\hwire[Navy]};\node at (4,4) {\hwire[Navy]};\node at (4,4) {\vwire[Navy]};\node at (5,4) {\hwire[Navy]};\node at (5,4) {\vwire[Navy]};\node at (0,3) {\el[Navy]};\node at (1,3) {\hwire[Navy]};\node at (1,3) {\vwire[Navy]};\node at (2,3) {\en[Navy]};\node at (3,3) {\nowire};\node at (4,3) {\el[Navy]};\node at (5,3) {\hwire[Navy]};\node at (5,3) {\vwire[Navy]};\node at (0,2) {\nowire};\node at (1,2) {\el[Navy]};\node at (2,2) {\newire[Navy]};\node at (2,2) {\swwire[Navy]};\node at (3,2) {\en[Navy]};\node at (4,2) {\nowire};\node at (5,2) {\el[Navy]};\node at (0,1) {\nowire};\node at (1,1) {\nowire};\node at (2,1) {\vwire[Navy]};\node at (3,1) {\el[Navy]};\node at (4,1) {\hwire[Navy]};\node at (5,1) {\hwire[Navy]};\node at (0,0) {\nowire};\node at (1,0) {\nowire};\node at (2,0) {\el[Navy]};\node at (3,0) {\hwire[Navy]};\node at (4,0) {\hwire[Navy]};\node at (5,0) {\hwire[Navy]};\draw[thick] (-0.5,-0.5) -- (-0.5,5.5) -- (5.5,5.5) -- (5.5,-0.5) -- (-0.5,-0.5);\node at (0,-1) {1};\node at (6,5) {4};\node at (1,-1) {2};\node at (6,4) {3};\node at (2,-1) {3};\node at (6,3) {5};\node at (3,-1) {4};\node at (6,2) {6};\node at (4,-1) {5};\node at (6,1) {1};\node at (5,-1) {6};\node at (6,0) {2};\end{tikzpicture}
\]

\[
\begin{tikzpicture}[scale =.4]\grid{5}{5}\node at (0,5) {\nowire};\node at (1,5) {\are[Orange]};\node at (2,5) {\hwire[Orange]};\node at (3,5) {\hwire[Orange]};\node at (4,5) {\hwire[Orange]};\node at (5,5) {\hwire[Orange]};\node at (0,4) {\are[Orange]};\node at (1,4) {\hwire[Orange]};\node at (1,4) {\vwire[Orange]};\node at (2,4) {\hwire[Orange]};\node at (3,4) {\hwire[Orange]};\node at (4,4) {\hwire[Orange]};\node at (5,4) {\hwire[Orange]};\node at (0,3) {\vwire[Orange]};\node at (1,3) {\vwire[Orange]};\node at (2,3) {\nowire};\node at (3,3) {\nowire};\node at (4,3) {\nowire};\node at (5,3) {\are[Navy]};\node at (0,2) {\vwire[Orange]};\node at (1,2) {\vwire[Orange]};\node at (2,2) {\nowire};\node at (3,2) {\nowire};\node at (4,2) {\are[Navy]};\node at (5,2) {\hwire[Navy]};\node at (5,2) {\vwire[Navy]};\node at (0,1) {\vwire[Orange]};\node at (1,1) {\vwire[Orange]};\node at (2,1) {\nowire};\node at (3,1) {\are[Navy]};\node at (4,1) {\hwire[Navy]};\node at (4,1) {\vwire[Navy]};\node at (5,1) {\hwire[Navy]};\node at (5,1) {\vwire[Navy]};\node at (0,0) {\vwire[Orange]};\node at (1,0) {\vwire[Orange]};\node at (2,0) {\are[Navy]};\node at (3,0) {\hwire[Navy]};\node at (3,0) {\vwire[Navy]};\node at (4,0) {\hwire[Navy]};\node at (4,0) {\vwire[Navy]};\node at (5,0) {\hwire[Navy]};\node at (5,0) {\vwire[Navy]};\draw[thick] (-0.5,-0.5) -- (-0.5,5.5) -- (5.5,5.5) -- (5.5,-0.5) -- (-0.5,-0.5);\node at (0,-1) {1};\node at (6,5) {2};\node at (1,-1) {2};\node at (6,4) {1};\node at (2,-1) {3};\node at (6,3) {6};\node at (3,-1) {4};\node at (6,2) {5};\node at (4,-1) {5};\node at (6,1) {4};\node at (5,-1) {6};\node at (6,0) {3};\end{tikzpicture}
\; \hspace{0.3in} \;
\begin{tikzpicture}[scale =.4]\grid{5}{5}\node at (0,5) {\nowire};\node at (1,5) {\nowire};\node at (2,5) {\nowire};\node at (3,5) {\are[Orange]};\node at (4,5) {\hwire[Orange]};\node at (5,5) {\hwire[Orange]};\node at (0,4) {\are[Orange]};\node at (1,4) {\hwire[Orange]};\node at (2,4) {\hwire[Orange]};\node at (3,4) {\hwire[Orange]};\node at (3,4) {\vwire[Orange]};\node at (4,4) {\hwire[Orange]};\node at (5,4) {\hwire[Orange]};\node at (0,3) {\vwire[Orange]};\node at (1,3) {\nowire};\node at (2,3) {\nowire};\node at (3,3) {\vwire[Orange]};\node at (4,3) {\nowire};\node at (5,3) {\are[Navy]};\node at (0,2) {\vwire[Orange]};\node at (1,2) {\are[Orange]};\node at (2,2) {\hwire[Orange]};\node at (3,2) {\jay[Orange]};\node at (4,2) {\are[Navy]};\node at (5,2) {\hwire[Navy]};\node at (5,2) {\vwire[Navy]};\node at (0,1) {\vwire[Orange]};\node at (1,1) {\vwire[Orange]};\node at (2,1) {\nowire};\node at (3,1) {\are[Navy]};\node at (4,1) {\hwire[Navy]};\node at (4,1) {\vwire[Navy]};\node at (5,1) {\hwire[Navy]};\node at (5,1) {\vwire[Navy]};\node at (0,0) {\vwire[Orange]};\node at (1,0) {\vwire[Orange]};\node at (2,0) {\are[Navy]};\node at (3,0) {\hwire[Navy]};\node at (3,0) {\vwire[Navy]};\node at (4,0) {\hwire[Navy]};\node at (4,0) {\vwire[Navy]};\node at (5,0) {\hwire[Navy]};\node at (5,0) {\vwire[Navy]};\draw[thick] (-0.5,-0.5) -- (-0.5,5.5) -- (5.5,5.5) -- (5.5,-0.5) -- (-0.5,-0.5);\node at (0,-1) {1};\node at (6,5) {2};\node at (1,-1) {2};\node at (6,4) {1};\node at (2,-1) {3};\node at (6,3) {6};\node at (3,-1) {4};\node at (6,2) {5};\node at (4,-1) {5};\node at (6,1) {4};\node at (5,-1) {6};\node at (6,0) {3};\end{tikzpicture}
\; \hspace{0.3in} \;
\begin{tikzpicture}[scale =.4]\node at (1,3) {\shade};\node at (2,3) {\shade};\node at (2,2) {\shade};\node at (1,2) {\shade};\grid{5}{5}\node at (0,5) {\nowire};\node at (1,5) {\nowire};\node at (2,5) {\nowire};\node at (3,5) {\are[Orange]};\node at (4,5) {\hwire[Orange]};\node at (5,5) {\hwire[Orange]};\node at (0,4) {\nowire};\node at (1,4) {\nowire};\node at (2,4) {\are[Orange]};\node at (3,4) {\hwire[Orange]};\node at (3,4) {\vwire[Orange]};\node at (4,4) {\hwire[Orange]};\node at (5,4) {\hwire[Orange]};\node at (0,3) {\are[Orange]};\node at (1,3) {\hwire[Orange]};\node at (2,3) {\jay[Orange]};\node at (3,3) {\vwire[Orange]};\node at (4,3) {\nowire};\node at (5,3) {\are[Navy]};\node at (0,2) {\vwire[Orange]};\node at (1,2) {\are[Orange]};\node at (2,2) {\hwire[Orange]};\node at (3,2) {\jay[Orange]};\node at (4,2) {\are[Navy]};\node at (5,2) {\hwire[Navy]};\node at (5,2) {\vwire[Navy]};\node at (0,1) {\vwire[Orange]};\node at (1,1) {\vwire[Orange]};\node at (2,1) {\nowire};\node at (3,1) {\are[Navy]};\node at (4,1) {\hwire[Navy]};\node at (4,1) {\vwire[Navy]};\node at (5,1) {\hwire[Navy]};\node at (5,1) {\vwire[Navy]};\node at (0,0) {\vwire[Orange]};\node at (1,0) {\vwire[Orange]};\node at (2,0) {\are[Navy]};\node at (3,0) {\hwire[Navy]};\node at (3,0) {\vwire[Navy]};\node at (4,0) {\hwire[Navy]};\node at (4,0) {\vwire[Navy]};\node at (5,0) {\hwire[Navy]};\node at (5,0) {\vwire[Navy]};\draw[thick] (-0.5,-0.5) -- (-0.5,5.5) -- (5.5,5.5) -- (5.5,-0.5) -- (-0.5,-0.5);\node at (0,-1) {1};\node at (6,5) {2};\node at (1,-1) {2};\node at (6,4) {1};\node at (2,-1) {3};\node at (6,3) {6};\node at (3,-1) {4};\node at (6,2) {5};\node at (4,-1) {5};\node at (6,1) {4};\node at (5,-1) {6};\node at (6,0) {3};\end{tikzpicture}
\; \hspace{0.3in}\;
\begin{tikzpicture}[scale =.4]\node at (1,3) {\shade};\node at (2,3) {\shade};\node at (2,2) {\shade};\node at (1,2) {\shade};\grid{5}{5}\node at (0,5) {\vwire[Navy]};\node at (1,5) {\vwire[Navy]};\node at (2,5) {\vwire[Navy]};\node at (3,5) {\el[Navy]};\node at (4,5) {\hwire[Navy]};\node at (4,5) {\vwire[Navy]};\node at (5,5) {\hwire[Navy]};\node at (5,5) {\vwire[Navy]};\node at (0,4) {\vwire[Navy]};\node at (1,4) {\vwire[Navy]};\node at (2,4) {\el[Navy]};\node at (3,4) {\hwire[Navy]};\node at (4,4) {\hwire[Navy]};\node at (4,4) {\vwire[Navy]};\node at (5,4) {\hwire[Navy]};\node at (5,4) {\vwire[Navy]};\node at (0,3) {\el[Navy]};\node at (1,3) {\hwire[Navy]};\node at (1,3) {\vwire[Navy]};\node at (2,3) {\en[Navy]};\node at (3,3) {\nowire};\node at (4,3) {\vwire[Navy]};\node at (5,3) {\el[Navy]};\node at (0,2) {\nowire};\node at (1,2) {\el[Navy]};\node at (2,2) {\newire[Navy]};\node at (2,2) {\swwire[Navy]};\node at (3,2) {\en[Navy]};\node at (4,2) {\el[Navy]};\node at (5,2) {\hwire[Navy]};\node at (0,1) {\nowire};\node at (1,1) {\nowire};\node at (2,1) {\vwire[Navy]};\node at (3,1) {\el[Navy]};\node at (4,1) {\hwire[Navy]};\node at (5,1) {\hwire[Navy]};\node at (0,0) {\nowire};\node at (1,0) {\nowire};\node at (2,0) {\el[Navy]};\node at (3,0) {\hwire[Navy]};\node at (4,0) {\hwire[Navy]};\node at (5,0) {\hwire[Navy]};\draw[thick] (-0.5,-0.5) -- (-0.5,5.5) -- (5.5,5.5) -- (5.5,-0.5) -- (-0.5,-0.5);\node at (0,-1) {1};\node at (6,5) {4};\node at (1,-1) {2};\node at (6,4) {3};\node at (2,-1) {3};\node at (6,3) {6};\node at (3,-1) {4};\node at (6,2) {5};\node at (4,-1) {5};\node at (6,1) {1};\node at (5,-1) {6};\node at (6,0) {2};\end{tikzpicture}
\]
\[
\begin{tikzpicture}[scale =.4]\grid{5}{5}\node at (0,5) {\nowire};\node at (1,5) {\are[Orange]};\node at (2,5) {\hwire[Orange]};\node at (3,5) {\hwire[Orange]};\node at (4,5) {\hwire[Orange]};\node at (5,5) {\hwire[Orange]};\node at (0,4) {\nowire};\node at (1,4) {\vwire[Orange]};\node at (2,4) {\nowire};\node at (3,4) {\are[Orange]};\node at (4,4) {\hwire[Orange]};\node at (5,4) {\hwire[Orange]};\node at (0,3) {\are[Orange]};\node at (1,3) {\hwire[Orange]};\node at (1,3) {\vwire[Orange]};\node at (2,3) {\hwire[Orange]};\node at (3,3) {\hwire[Orange]};\node at (3,3) {\vwire[Orange]};\node at (4,3) {\hwire[Orange]};\node at (5,3) {\hwire[Orange]};\node at (0,2) {\vwire[Orange]};\node at (1,2) {\vwire[Orange]};\node at (2,2) {\nowire};\node at (3,2) {\vwire[Orange]};\node at (4,2) {\nowire};\node at (5,2) {\are[Navy]};\node at (0,1) {\vwire[Orange]};\node at (1,1) {\vwire[Orange]};\node at (2,1) {\nowire};\node at (3,1) {\vwire[Orange]};\node at (4,1) {\are[Navy]};\node at (5,1) {\hwire[Navy]};\node at (5,1) {\vwire[Navy]};\node at (0,0) {\vwire[Orange]};\node at (1,0) {\vwire[Orange]};\node at (2,0) {\are[Navy]};\node at (3,0) {\hwire[Navy]};\node at (3,0) {\vwire[Orange]};\node at (4,0) {\hwire[Navy]};\node at (4,0) {\vwire[Navy]};\node at (5,0) {\hwire[Navy]};\node at (5,0) {\vwire[Navy]};\draw[thick] (-0.5,-0.5) -- (-0.5,5.5) -- (5.5,5.5) -- (5.5,-0.5) -- (-0.5,-0.5);\node at (0,-1) {1};\node at (6,5) {2};\node at (1,-1) {2};\node at (6,4) {4};\node at (2,-1) {3};\node at (6,3) {1};\node at (3,-1) {4};\node at (6,2) {6};\node at (4,-1) {5};\node at (6,1) {5};\node at (5,-1) {6};\node at (6,0) {3};\end{tikzpicture}
\;
\begin{tikzpicture}[scale =.4]\grid{5}{5}\node at (0,5) {\nowire};\node at (1,5) {\are[Orange]};\node at (2,5) {\hwire[Orange]};\node at (3,5) {\hwire[Orange]};\node at (4,5) {\hwire[Orange]};\node at (5,5) {\hwire[Orange]};\node at (0,4) {\nowire};\node at (1,4) {\vwire[Orange]};\node at (2,4) {\nowire};\node at (3,4) {\nowire};\node at (4,4) {\are[Orange]};\node at (5,4) {\hwire[Orange]};\node at (0,3) {\are[Orange]};\node at (1,3) {\hwire[Orange]};\node at (1,3) {\vwire[Orange]};\node at (2,3) {\hwire[Orange]};\node at (3,3) {\hwire[Orange]};\node at (4,3) {\hwire[Orange]};\node at (4,3) {\vwire[Orange]};\node at (5,3) {\hwire[Orange]};\node at (0,2) {\vwire[Orange]};\node at (1,2) {\vwire[Orange]};\node at (2,2) {\nowire};\node at (3,2) {\are[Orange]};\node at (4,2) {\jay[Orange]};\node at (5,2) {\are[Navy]};\node at (0,1) {\vwire[Orange]};\node at (1,1) {\vwire[Orange]};\node at (2,1) {\nowire};\node at (3,1) {\vwire[Orange]};\node at (4,1) {\are[Navy]};\node at (5,1) {\hwire[Navy]};\node at (5,1) {\vwire[Navy]};\node at (0,0) {\vwire[Orange]};\node at (1,0) {\vwire[Orange]};\node at (2,0) {\are[Navy]};\node at (3,0) {\hwire[Navy]};\node at (3,0) {\vwire[Orange]};\node at (4,0) {\hwire[Navy]};\node at (4,0) {\vwire[Navy]};\node at (5,0) {\hwire[Navy]};\node at (5,0) {\vwire[Navy]};\draw[thick] (-0.5,-0.5) -- (-0.5,5.5) -- (5.5,5.5) -- (5.5,-0.5) -- (-0.5,-0.5);\node at (0,-1) {1};\node at (6,5) {2};\node at (1,-1) {2};\node at (6,4) {4};\node at (2,-1) {3};\node at (6,3) {1};\node at (3,-1) {4};\node at (6,2) {6};\node at (4,-1) {5};\node at (6,1) {5};\node at (5,-1) {6};\node at (6,0) {3};\end{tikzpicture}
\;
\begin{tikzpicture}[scale =.4]\grid{5}{5}\node at (0,5) {\nowire};\node at (1,5) {\nowire};\node at (2,5) {\nowire};\node at (3,5) {\are[Orange]};\node at (4,5) {\hwire[Orange]};\node at (5,5) {\hwire[Orange]};\node at (0,4) {\nowire};\node at (1,4) {\are[Orange]};\node at (2,4) {\hwire[Orange]};\node at (3,4) {\jay[Orange]};\node at (4,4) {\are[Orange]};\node at (5,4) {\hwire[Orange]};\node at (0,3) {\are[Orange]};\node at (1,3) {\hwire[Orange]};\node at (1,3) {\vwire[Orange]};\node at (2,3) {\hwire[Orange]};\node at (3,3) {\hwire[Orange]};\node at (4,3) {\hwire[Orange]};\node at (4,3) {\vwire[Orange]};\node at (5,3) {\hwire[Orange]};\node at (0,2) {\vwire[Orange]};\node at (1,2) {\vwire[Orange]};\node at (2,2) {\nowire};\node at (3,2) {\are[Orange]};\node at (4,2) {\jay[Orange]};\node at (5,2) {\are[Navy]};\node at (0,1) {\vwire[Orange]};\node at (1,1) {\vwire[Orange]};\node at (2,1) {\nowire};\node at (3,1) {\vwire[Orange]};\node at (4,1) {\are[Navy]};\node at (5,1) {\hwire[Navy]};\node at (5,1) {\vwire[Navy]};\node at (0,0) {\vwire[Orange]};\node at (1,0) {\vwire[Orange]};\node at (2,0) {\are[Navy]};\node at (3,0) {\hwire[Navy]};\node at (3,0) {\vwire[Orange]};\node at (4,0) {\hwire[Navy]};\node at (4,0) {\vwire[Navy]};\node at (5,0) {\hwire[Navy]};\node at (5,0) {\vwire[Navy]};\draw[thick] (-0.5,-0.5) -- (-0.5,5.5) -- (5.5,5.5) -- (5.5,-0.5) -- (-0.5,-0.5);\node at (0,-1) {1};\node at (6,5) {2};\node at (1,-1) {2};\node at (6,4) {4};\node at (2,-1) {3};\node at (6,3) {1};\node at (3,-1) {4};\node at (6,2) {6};\node at (4,-1) {5};\node at (6,1) {5};\node at (5,-1) {6};\node at (6,0) {3};\end{tikzpicture}
\;
\begin{tikzpicture}[scale =.4]\node at (2,4) {\shade};\node at (2,3) {\shade};\node at (3,4) {\shade};\node at (3,3) {\shade};\grid{5}{5}\node at (0,5) {\nowire};\node at (1,5) {\nowire};\node at (2,5) {\nowire};\node at (3,5) {\are[Orange]};\node at (4,5) {\hwire[Orange]};\node at (5,5) {\hwire[Orange]};\node at (0,4) {\nowire};\node at (1,4) {\are[Orange]};\node at (2,4) {\hwire[Orange]};\node at (3,4) {\jay[Orange]};\node at (4,4) {\are[Orange]};\node at (5,4) {\hwire[Orange]};\node at (0,3) {\nowire};\node at (1,3) {\vwire[Orange]};\node at (2,3) {\are[Orange]};\node at (3,3) {\hwire[Orange]};\node at (4,3) {\hwire[Orange]};\node at (4,3) {\vwire[Orange]};\node at (5,3) {\hwire[Orange]};\node at (0,2) {\are[Orange]};\node at (1,2) {\hwire[Orange]};\node at (1,2) {\vwire[Orange]};\node at (2,2) {\jay[Orange]};\node at (3,2) {\are[Orange]};\node at (4,2) {\jay[Orange]};\node at (5,2) {\are[Navy]};\node at (0,1) {\vwire[Orange]};\node at (1,1) {\vwire[Orange]};\node at (2,1) {\nowire};\node at (3,1) {\vwire[Orange]};\node at (4,1) {\are[Navy]};\node at (5,1) {\hwire[Navy]};\node at (5,1) {\vwire[Navy]};\node at (0,0) {\vwire[Orange]};\node at (1,0) {\vwire[Orange]};\node at (2,0) {\are[Navy]};\node at (3,0) {\hwire[Navy]};\node at (3,0) {\vwire[Orange]};\node at (4,0) {\hwire[Navy]};\node at (4,0) {\vwire[Navy]};\node at (5,0) {\hwire[Navy]};\node at (5,0) {\vwire[Navy]};\draw[thick] (-0.5,-0.5) -- (-0.5,5.5) -- (5.5,5.5) -- (5.5,-0.5) -- (-0.5,-0.5);\node at (0,-1) {1};\node at (6,5) {2};\node at (1,-1) {2};\node at (6,4) {4};\node at (2,-1) {3};\node at (6,3) {1};\node at (3,-1) {4};\node at (6,2) {6};\node at (4,-1) {5};\node at (6,1) {5};\node at (5,-1) {6};\node at (6,0) {3};\end{tikzpicture}
\;
\begin{tikzpicture}[scale =.4]\node at (2,4) {\shade};\node at (2,3) {\shade};\node at (3,4) {\shade};\node at (3,3) {\shade};\grid{5}{5}\node at (0,5) {\vwire[Navy]};\node at (1,5) {\vwire[Navy]};\node at (2,5) {\vwire[Navy]};\node at (3,5) {\el[Navy]};\node at (4,5) {\hwire[Navy]};\node at (4,5) {\vwire[Navy]};\node at (5,5) {\hwire[Navy]};\node at (5,5) {\vwire[Navy]};\node at (0,4) {\vwire[Navy]};\node at (1,4) {\el[Navy]};\node at (2,4) {\hwire[Navy]};\node at (2,4) {\vwire[Navy]};\node at (3,4) {\en[Navy]};\node at (4,4) {\el[Navy]};\node at (5,4) {\hwire[Navy]};\node at (5,4) {\vwire[Navy]};\node at (0,3) {\vwire[Navy]};\node at (1,3) {\nowire};\node at (2,3) {\el[Navy]};\node at (3,3) {\newire[Navy]};\node at (3,3) {\swwire[Navy]};\node at (4,3) {\hwire[Navy]};\node at (5,3) {\hwire[Navy]};\node at (5,3) {\vwire[Navy]};\node at (0,2) {\el[Navy]};\node at (1,2) {\hwire[Navy]};\node at (2,2) {\en[Navy]};\node at (3,2) {\el[Navy]};\node at (4,2) {\en[Navy]};\node at (5,2) {\el[Navy]};\node at (0,1) {\nowire};\node at (1,1) {\nowire};\node at (2,1) {\vwire[Navy]};\node at (3,1) {\nowire};\node at (4,1) {\el[Navy]};\node at (5,1) {\hwire[Navy]};\node at (0,0) {\nowire};\node at (1,0) {\nowire};\node at (2,0) {\el[Navy]};\node at (3,0) {\hwire[Navy]};\node at (4,0) {\hwire[Navy]};\node at (5,0) {\hwire[Navy]};\draw[thick] (-0.5,-0.5) -- (-0.5,5.5) -- (5.5,5.5) -- (5.5,-0.5) -- (-0.5,-0.5);\node at (0,-1) {1};\node at (6,5) {4};\node at (1,-1) {2};\node at (6,4) {5};\node at (2,-1) {3};\node at (6,3) {2};\node at (3,-1) {4};\node at (6,2) {6};\node at (4,-1) {5};\node at (6,1) {3};\node at (5,-1) {6};\node at (6,0) {1};\end{tikzpicture}
\]

\bibliographystyle{alpha}
\bibliography{references}

@article {lascoux1982poly,
    AUTHOR = {Lascoux, Alain and Sch\"{u}tzenberger, Marcel-Paul},
     TITLE = {Polyn\^{o}mes de {S}chubert},
   JOURNAL = {C. R. Acad. Sci. Paris S\'{e}r. I Math.},
  FJOURNAL = {Comptes Rendus des S\'{e}ances de l'Acad\'{e}mie des Sciences.
              S\'{e}rie I. Math\'{e}matique},
    VOLUME = {294},
      YEAR = {1982},
    NUMBER = {13},
     PAGES = {447--450},
}

@article {lascoux82groth,
    AUTHOR = {Lascoux, Alain and Sch\"utzenberger, Marcel-Paul},
     TITLE = {Structure de {H}opf de l'anneau de cohomologie et de l'anneau
              de {G}rothendieck d'une vari\'et\'e{} de drapeaux},
   JOURNAL = {C. R. Acad. Sci. Paris S\'er. I Math.},
  FJOURNAL = {Comptes Rendus des S\'eances de l'Acad\'emie des Sciences.
              S\'erie I. Math\'ematique},
    VOLUME = {295},
      YEAR = {1982},
    NUMBER = {11},
     PAGES = {629--633},
      ISSN = {0249-6291},
   MRCLASS = {14M17},
  MRNUMBER = {686357},
}

@article {bergeron1993rc,
    AUTHOR = {Bergeron, Nantel and Billey, Sara},
     TITLE = {R{C}-graphs and {S}chubert polynomials},
   JOURNAL = {Experiment. Math.},
  FJOURNAL = {Experimental Mathematics},
    VOLUME = {2},
      YEAR = {1993},
    NUMBER = {4},
     PAGES = {257--269},
}

@article {billey1993comb,
    AUTHOR = {Billey, Sara C. and Jockusch, William and Stanley, Richard P.},
     TITLE = {Some combinatorial properties of {S}chubert polynomials},
   JOURNAL = {J. Algebraic Combin.},
  FJOURNAL = {Journal of Algebraic Combinatorics. An International Journal},
    VOLUME = {2},
      YEAR = {1993},
    NUMBER = {4},
     PAGES = {345--374},
}

@article {lam2021back,
    AUTHOR = {Lam, Thomas and Lee, Seung Jin and Shimozono, Mark},
     TITLE = {Back stable {S}chubert calculus},
   JOURNAL = {Compos. Math.},
  FJOURNAL = {Compositio Mathematica},
    VOLUME = {157},
      YEAR = {2021},
    NUMBER = {5},
     PAGES = {883--962},
}

@article {weigandt2021bumpless,
    AUTHOR = {Weigandt, Anna},
     TITLE = {Bumpless pipe dreams and alternating sign matrices},
   JOURNAL = {J. Combin. Theory Ser. A},
  FJOURNAL = {Journal of Combinatorial Theory. Series A},
    VOLUME = {182},
      YEAR = {2021},
     PAGES = {Paper No. 105470, 52pp.},
}

@article {lenart1999,
    AUTHOR = {Lenart, Cristian},
     TITLE = {Noncommutative {S}chubert calculus and {G}rothendieck
              polynomials},
   JOURNAL = {Adv. Math.},
  FJOURNAL = {Advances in Mathematics},
    VOLUME = {143},
      YEAR = {1999},
    NUMBER = {1},
     PAGES = {159--183},
      ISSN = {0001-8708,1090-2082},
   MRCLASS = {05E10 (14M15)},
  MRNUMBER = {1680646},
MRREVIEWER = {Ernesto\ Vallejo},
       DOI = {10.1006/aima.1998.1795},
       URL = {https://doi.org/10.1006/aima.1998.1795},
}

@article {lascoux2004,
    AUTHOR = {Lascoux, Alain},
     TITLE = {Schubert \& {G}rothendieck: un bilan bid\'ecennal},
   JOURNAL = {S\'em. Lothar. Combin.},
  FJOURNAL = {S\'eminaire Lotharingien de Combinatoire},
    VOLUME = {50},
      YEAR = {2003/04},
     PAGES = {Art. B50i, 32},
      ISSN = {1286-4889},
   MRCLASS = {05E15 (05E05 14M15 14N15)},
  MRNUMBER = {2118047},
MRREVIEWER = {Harry\ Tamvakis},
}

@article {fomin1994groth,
    AUTHOR = {Fomin, Sergey and Kirillov, Anatol},
     TITLE = {Grothendieck polynomials and the {Y}ang-{B}axter equation},
  JOURNAL = {6th International Conference on Formal Power Series and Algebraic Combinatorics},
      YEAR = {1994},
}

@misc{weigandt2025changingbasespipedream,
      title={Changing Bases with Pipe Dream Combinatorics}, 
      author={Anna Weigandt},
      year={2025},
      eprint={2506.07306},
      archivePrefix={arXiv},
      primaryClass={math.CO},
      url={https://arxiv.org/abs/2506.07306}, 
}
\end{document}